\documentclass{siamart220329}

\usepackage{graphicx}
\graphicspath{{./fig/}}

\usepackage{subcaption}
\usepackage[shortlabels]{enumitem}

\usepackage{amsmath, amsfonts, amssymb}
\usepackage{tikz-cd}
\usepackage{mathtools}

\usepackage[mathscr]{euscript}
\usepackage{dsfont}

\usepackage{tikz}
\usepackage{tikz-cd}
\usetikzlibrary{calc}

\usepackage{multiaudience}

\SetNewAudience{SIAM}
\SetNewAudience{arXiv}
\DefCurrentAudience{arXiv}

\newtheorem{example}[theorem]{Example}
\newtheorem{remark}[theorem]{Remark}
\newtheorem{assumption}[theorem]{Assumption}

\newcommand{\define}[1]{\emph{#1}}

\newcounter{refer}
\newcommand{\labellocal}[1]{\label{local:\therefer:#1}}
\newcommand{\reflocal}[1]{\ref{local:\therefer:#1}}
\newcommand{\creflocal}[1]{\cref{local:\therefer:#1}}

\newcommand{\Span}{\operatorname{span}}
\newcommand{\Ima}{\operatorname{Im}}
\newcommand{\Ker}{\operatorname{Ker}}
\newcommand{\re}{\operatorname{Re}}

\newcommand{\ind}[1]{\mathds{1}_{#1}}

\newcommand\restr[2]{{
		\left.\kern-\nulldelimiterspace 
		#1 
		\vphantom{\big|} 
		\right|_{#2} 
	}}
	
\newcommand{\RR}{\mathbb{R}}

\newcommand{\TT}{\mathbb{T}}
\newcommand{\NN}{\mathbb{N}}
\newcommand{\ZZ}{\mathbb{Z}}
\newcommand{\CC}{\mathbb{C}}

\newcommand{\CCg}[1]{\CC_{>#1}}
\newcommand{\CCge}[1]{\CC_{\geq#1}}
\newcommand{\CCl}[1]{\CC_{<#1}}
\newcommand{\CCle}[1]{\CC_{\leq#1}}
\newcommand{\RRg}[1]{\RR_{>#1}}
\newcommand{\RRge}[1]{\RR_{\geq#1}}

\newcommand{\RRle}[1]{\RR_{\leq#1}}
\newcommand{\ZZg}[1]{\ZZ_{>#1}}
\newcommand{\ZZge}[1]{\ZZ_{\geq#1}}

\newcommand{\sectl}[1]{\mathbb{S}_{<#1}}
\newcommand{\local}{\mathrm{loc}}
\newcommand\prop{input-resilient}
\def\propnoun{input-resilience}

\newcommand{\spec}{\operatorname{Spec}}
\newcommand{\domain}{\operatorname{Dom}}
\newcommand{\gradient}{\nabla}
\newcommand{\divergence}{\operatorname{div}}

\newcommand{\identity}{\mathrm{id}}

\newcommand{\Homom}{\operatorname{Hom}}
\newcommand{\Endom}{\operatorname{End}}

\newcommand{\hooktwoheadrightarrow}{%
	\hookrightarrow\mathrel{\mspace{-15mu}}\rightarrow
}

\newcommand{\sysclose}{\operatorname{Close}}

\newcommand{\vertiii}[1]{{\left\vert\kern-0.25ex\left\vert\kern-0.25ex\left\vert #1 
	\right\vert\kern-0.25ex\right\vert\kern-0.25ex\right\vert}}
    
\newenvironment{talign}
	{\let\displaystyle\textstyle\align}
	{\endalign}
\newenvironment{talign*}
	{\let\displaystyle\textstyle\csname align*\endcsname}
	{\endalign}

\newcommand{\change}[1]{{#1}}

\title{An input-output framework for stability and synchronization analysis of networks of infinite-dimensional linear systems\thanks{Submitted to the editors on April 20, 2024
\funding{This research was supported by the Natural Sciences and Engineering Research Council of Canada (NSERC).}}
}

\author{Tian Xia\thanks{The Edward S. Rogers Sr. Department of Electrical \& Computer Engineering, University of Toronto, Toronto, ON, M5S 3G4, Canada
    (\email{t.xia@mail.utoronto.ca}).}
\and Luca Scardovi\thanks{--- (\email{luca.scardovi@utoronto.ca}).}
}

\begin{document}

\maketitle

\newif\ifshowpages
\showpagestrue
\ifshowpages
	\pagestyle{plain}
	\thispagestyle{plain}
\else
	\pagestyle{empty}
	\thispagestyle{empty}
\fi

\begin{abstract}
	This paper presents a synchronization criterion for networks of infinite-dimensional linear systems, extending a previous result for finite-dimensional systems. 
	Our result, established in the general framework of input-output relations, requires an additional input-output stability property, compared to the finite-dimensional counterpart. 
	We show that this this property holds for a large class of infinite-dimensional systems including abstract Cauchy problems, parabolic partial differential equations, and time-delay differential equations.      
\end{abstract}




\section{Introduction}

In the past two decades, the increasing prevalence of applications involving networks and networked systems has brought significant attention to consensus or synchronization problems \cite{ren07}. 
In particular, a good amount of effort has been devoted to the understanding of synchronization in network of identical finite-dimensional linear systems \cite{scardovi07,sepulchre08,moreau05,scardovi09,li10}. In this setting, synchronization is effectively characterized by properties of the isolated systems and the interconnection, and controllers can be designed to enforce synchronization under mild assumptions.

Recently, there has been a growing interest in extending these results to systems evolving on infinite-dimensional state spaces \cite{demetriou13,pilloni16}.
Real-world processes described by partial differential equations (PDEs), such as fluid flow, heat transfer, electromagnetic waves, quantum systems, and biological systems, exemplify infinite-dimensional systems. Furthermore, time-delay differential equations can be modelled as transport PDEs.
The study of networks of infinite-dimensional systems is of interest due to its potential applications in battery management for networks of lithium-ion cells \cite{tang17}, coordination in groups of flexible spacecrafts \cite{chen19}, and control of epidemics \cite{bertaglia21}.

In this paper, we generalize a synchronization criterion (e.g. \cite[Thm. 1]{xia15}) for networks consisting of finite-dimensional systems
\begin{align} \label{eq:system}
	\left\{
	\begin{aligned}
		\dot{x}_j &= A x_j + B u_j, \\
		y_j &= C x_j,
	\end{aligned}
	\right.
	\qquad
	u_j = \sum_{i=1}^{n} \sigma_{ji} (y_i - y_j),
\end{align}
where $j=1,\ldots,n.$ Roughly, the synchronization criterion asserts that the outputs of the network \cref{eq:system} synchronize if and only if $A + \lambda BC$ is Hurwitz for all $\lambda$ in a finite set associated to the coupling coefficients $\sigma_{ji}$ \cite{xia15}. The synchronization criterion effectively reduces a network problem to a lower dimensional stability problem. 

A natural generalization is obtained by replacing the finite-dimensional dynamics \cref{eq:system} by an abstract Cauchy problem, an abstract model for many infinite-dimensional systems \cite{curtain95}. 
An abstract Cauchy problem is a generalization of \cref{eq:system}, where the operator $A$ generates a strongly continuous semigroup on a Banach space and $B$, $C$ are bounded operators. 
In this setting, the synchronization criterion remains valid \cite{demetriou13}. 
It is worth noticing, however, that the boundedness assumption on $B$ and $C$ excludes important classes of infinite-dimensional linear systems.
In particular, PDEs with boundary input and/or boundary output, which appear ubiquitously in applications, are typically described by ``extended'' abstract Cauchy problems where the operators $B$ and $C$ are unbounded \cite{tucsnak09}. 
One of the main results of this paper shows that the synchronization criterion holds also for this class of systems.

The literature on synchronization of infinite-dimensional systems focuses primarily on specific classes of PDEs.
Some references include \cite{li18,aguilar21,chen21} on wave equations, \cite{pilloni16} on heat equations, \cite{wu16,xia20,deutscher22} on second-order parabolic equations, and \cite{gabriel22} on coupled transport equations. 
General models of infinite-dimensional systems have been considered recently in \cite{ferrante21,singh22}.
In \cite{ferrante21}, the authors provide sufficient conditions for synchronization of abstract boundary-actuated system. 
In \cite{singh22}, the authors study extended abstract Cauchy problems and provide a synchronization criterion, however the criterion relies on a well-posedness assumption on the network that we do not require.

The synchronization criterion that we propose is based on very mild assumptions on the system dynamics. 
One difficulty when working in such a general setting, is that the interconnections of extended abstract Cauchy problems are not guaranteed to be extended abstract Cauchy problems.
To overcome this issue, our results will be framed in the context of input-output relations. In \Cref{sect:io_relation}, we show that the synchronization criterion for finite-dimensional systems still holds if the relation representing the system satisfies the regularity property in \cref{def:prop}. 
The subsequent portion of the paper focuses on establishing conditions under which this property holds for specific categories of infinite-dimensional systems including (extended) abstract Cauchy problems, parabolic PDEs, and time-delay ordinary differential equations.    

The paper is organized as follows.
In \Cref{sect:io_relation} we show that the synchronization criterion for finite-dimensional systems in \cite{xia15} is applicable to networks of input-output relations assuming that \cref{def:prop} is verified. 
Sufficient conditions for \cref{def:prop} are provided in \Cref{sect:input_resilient} (abstract Cauchy problems), in  \Cref{sect:parabolic} (parabolic PDEs), and \Cref{sect:delayed} (delay ODEs).
\change{Preliminary results appeared in \cite{xia21}}.

\section{Notation and Preliminaries}

The set of real (resp. complex) numbers is denoted by $\RR$ (resp. $\CC$). 
The set $\CCge{a}$ denotes $\{z \in \CC \mid \re z \ge a\}$, and sets $\CCg{a}$, $\CCl{a}$, $\CCle{a}$ are analogously defined. 
Similarly, $\RRg{a} = (a, \infty)$.
The sector $\{t e^{i\alpha} \mid t > 0,\, \alpha \in [0,\theta)\}$ is denoted by $\sectl{\theta}$.
The symbols $\CC^n$ and $\CC^{n \times m}$ respectively denote the $n$-dimensional vector space and the set of $n$ by $m$ matrices. 
Certain inequalities are expressed using the Vinogradov notation: given functions $f$ and $g$ on $X$, we say $f(x) \lesssim g(x)$ uniformly in $x \in S \subseteq X$ if there exists a constant $C > 0$ such that $f(x) \leq C g(x)$ for all $x \in S$. 
Let $S$ be a subset of $\RR^d$ whose interior $S^{\circ}$ is dense in $S$, and $Y$ be a Fr\'{e}chet space\footnote{A Fr\'{e}chet space is a vector space topologized by a countable family of semi-norms which is complete and Hausdorff.} 
topologized by semi-norms $\lVert \cdot\rVert_{Y_i}$ where $i \in \NN$. 
The partial derivative $\partial_j f$ of a function $f : S \rightarrow Y$ in direction $j$ is defined as usual in $S^{\circ}$ and by continuity on $S \setminus S^{\circ}$.
The vector space of $k$-times continuously differentiable functions with finite semi-norms
\begin{align*}
	\lVert f \rVert_{C^k(S, Y_i)} := \sup_{\lvert\alpha\rvert \leq k,\; x \in S} \lVert \partial^{\alpha} f(x)\rVert_{Y_i} \qquad \forall i \in \NN,
\end{align*}
and equipped with the topology induced by these semi-norms is denoted by $C^k(S, Y)$.
The notation $C^{k}_{\local}(S, Y)$ denotes the set of $k$-times continuously differentiable functions topologized by the semi-norms $\lVert \cdot \rVert_{C^k(K, Y_i)}$ where $i \in \NN$ and $K$ ranges over compact subsets of $S$.
Furthermore, $C^{\infty}(S, Y)$ (resp. $C^{\infty}_{\local}(S, Y)$) denotes the intersection $\bigcap_{k \in \NN} C^{k}(S, Y)$ (resp. $\bigcap_{k \in \NN} C^{k}_{\local}(S, Y)$), with the topology induced by including all semi-norms from every $C^{k}(S, Y)$ (resp. $C^{k}_{\local}(S, Y)$).
All spaces $C^{k}(S, Y)$ and $C^{k}_{\local}(S, Y)$, including $k = \infty$, are Fr\'{e}chet \change{(by the arguments of \cite[Sect. 1.46]{rudin91})}.
In the special case $k \in \ZZge{0}$ and $Y$ is a Banach space, $C^{k}(S, Y)$ is Banach.
The notation $C^{k}_{c}(S, Y)$ denotes the subspace of $C^{\infty}_{\local}(S, Y)$ with supports that are compact in $S$.
When $S$ is compact, $C^{k}_{\local}(S, Y) = C^{k}(S, Y)$.
The absence of parameter $Y$ means $Y = \CC$.
Both parameters $Y$ and $S$ may be omitted if they can be inferred from the context.
All vector spaces in this paper have $\CC$ as the scalar field. 

\subsection{Input-output relations}
In this paper, we will often use an abstract representation of open systems based on {\em input-output relations}.
A (LTI input-output) relation with input space $U$ and output space $Y$ (both Fr\'{e}chet spaces) is a linear subspace $\mathcal{P}$ of $C^{\infty}_{\local}(\TT, U) \times C^{\infty}_{\local}(\TT, Y)$ that is invariant under translation in time.
The set of times $\TT$ denotes either $\RRge{0}$, in which case $\mathcal{P}$ is called a continuous-time relation, or $\ZZge{0}$, in which case $\mathcal{P}$ is a discrete-time relation\footnote{The notation $C^{\infty}_{\local}(\ZZge{0})$ is defined as $C^{0}_{\local}(\ZZge{0})$ by convention.}.

A relation $\mathcal{P}$ represents an open system for which every $(u, y) \in \mathcal{P}$ is an admissible input-output pair. 
A $(U, Y)$-relation is a relation with input space $U$ and output space $Y$; an $X$-relation is an $(X, X)$-relation. A $(U, Y)$-relation $\mathcal{P}$ is \define{well-defined} if for every input $u \in U$, there exists $y\in Y$ such that $(u, y) \in \mathcal{P}$.

Any set of equations with input/output specifications defines a relation.
For example, the relation defined by $\dot{x} = Ax + B_0 u + B_1 \dot{u}$ and $y = Cx + Du$ with input $u$ and output $y$ is the set of $(u, y) \in C^{\infty}_{\local}(\RRge{0}, U) \times C^{\infty}_{\local}(\RRge{0}, Y)$ for which there exists $x \in C^{\infty}_{\local}(\RRge{0}, X)$ such that $(u, x, y)$ solves the equations.



We say that a function $f \in C^{\infty}_{\local}(\TT, X)$ vanishes if $\lim_{t \rightarrow \infty} f(t) = 0$ in $X$. Analogously, we say that $f \in C^{\infty}_{\local}(\RRge{0}, X)$ vanishes smoothly if $\lim_{t \rightarrow \infty} \partial_t^k f(t) = 0$ for every $k \in \ZZge{0}$.  

\begin{definition} \label{def:stability}
	An $(U, Y)$-relation $\mathcal{P}$ is
	\begin{itemize}
		\item[i)] \define{stable} \define{[smoothly stable]} if for every pair $(0, y) \in \mathcal{P}$, the output $y$ is vanishing [smoothly vanishing].
\item[ii)] \define{vanishing-input-vanishing-output (VIVO) [smoothly VIVO]} if for every pair $(u, y) \in \mathcal{P}$ where $u$ is vanishing [smoothly vanishing], the output $y$ is vanishing [smoothly vanishing].
	\end{itemize}
	The smooth versions of the definitions are only defined for continuous-time systems.
\end{definition}


\begin{remark}
	We say $\mathcal{P}$ is VIVO in $(U, Y)$ (or VIVO in $X$ if $U = Y = X$) to emphasize that convergence is measured using the topologies of $U$ and $Y$.
	For example, a function $f \in C^{\infty}_{\local}(\RRge{0} \times [0,1], \CC)$ may be interpreted as an element of $C^{\infty}_{\local}(\RRge{0}, C^0([0,1], \CC))$, $C^{\infty}_{\local}(\RRge{0}, L^1([0,1], \CC))$, etc. 
	Stability in $C^0$ (resp. $L^1$) means $\lim_{t \rightarrow \infty} \|f(t, \cdot)\|_{C^0} = 0$ (resp. $\lim_{t \rightarrow \infty} \|f(t, \cdot)\|_{L^1} = 0$). 
\end{remark}

\change{
Stability defined in \cref{def:stability} is equivalent to the usual definition of stability for finite-dimensional systems. 
Consider any observable\footnote{\change{A system $\dot{x} = Ax + Bu, \; y = Cx + Du$ and its observable subsystem always define the same relation. However, the same property does not hold for the controllable subsystem because the relation contains input-output pairs associated to non-zero initial conditions.}} state-space representation $\dot{x} = Ax + Bu,\, y = Cx + Du$. 
Under zero input, the state $x$ is vanishing iff the output $y$ is vanishing. 
Therefore, $\dot{x} = Ax + Bu,\, y = Cx + Du$ is internally stable iff its induced relation, with input $u$ and output $y$, is stable.
}


It is worth pointing out that smooth stability implies stability, but smooth VIVO does not imply VIVO in general.
An example is the relation defined by $y = \dot{u}$ with input $u$ and output $y$.
Relations like this appear when dealing with PDEs with boundary inputs, \change{because boundary inputs can be transformed into in-domain inputs with time derivatives} (see \Cref{sect:parabolicsystem_boundaryinput}).
This is the motivation for defining smooth VIVO.
\change{For relations satisfying the topological property defined below}, stability and VIVO always imply their smooth counterparts.

\begin{definition}
A $(U, Y)$-relation $\mathcal{P}$ is \define{closed} if it is closed\footnote{The topology of $C^{\infty}_{\local}(\TT, X)$ is generated by the zero neighbourhoods $\{f \mid \partial_t^k f(s) \in V \text{ for all } s \in K\}$, where $V$ varies over zero neighbourhoods of $X$ and $K$ varies over compact subsets of $\TT$.} as a subset of 
	\begin{align*}
		C^{\infty}_{\local}(\TT, U) \times C^{\infty}_{\local}(\TT, Y).
	\end{align*}
\end{definition}

\begin{proposition} \label{pro:relation_closed} \stepcounter{refer}
	Let $\mathcal{P}$ be a closed continuous-time $(U, Y)$-relation.
	If $(u, y) \in \mathcal{P}$, then $(\dot{u}, \dot{y}) \in \mathcal{P}$.
	\change{Hence}, $\mathcal{P}$ is smoothly stable (resp. smoothly VIVO) if it is stable (resp. VIVO).
\end{proposition}
\begin{proof}
	\change{For $\epsilon > 0$, define $u_{\epsilon}(t) = (u(t + \epsilon) - u(t)) / \epsilon$ and $y_{\epsilon}(t) = (y(t + \epsilon) - y(t)) / \epsilon$.}
	\change{For every $k \in \ZZge{0}$, the computation
	\begin{align*}
		\lim_{\epsilon \rightarrow 0} \partial_t^{k} u_{\epsilon}(t) = \lim_{\epsilon \rightarrow 0} \frac{1}{\epsilon} \int_{0}^{\epsilon} \partial_t^{k}\dot{u}(t+\tau) d\tau = \partial_t^{k}\dot{u}(t),
	\end{align*}
	shows that $\partial_t^k u_{\epsilon}$ converges to to $\partial_t^k \dot{u}$ in $C^{0}_{\local}(\RRge{0}, U)$.
	Thus, $\dot{u} = \lim_{\epsilon \rightarrow 0^+} u_{\epsilon}$ in $C^{\infty}_{\local}(\RRge{0}, U)$, and similarly $\dot{y} = \lim_{\epsilon \rightarrow 0^+} y_{\epsilon}$ in $C^{\infty}_{\local}(\RRge{0}, Y)$.
	By LTI properties, $\mathcal{P}$ contains the pairs $(u_{\epsilon}, y_{\epsilon})$; 
	by closedness, $\mathcal{P}$ also contains the limit $(\dot{u}, \dot{y})$.}
\end{proof}

\begin{remark}
	\change{
	The closedness property can be verified by expressing the relation as the zero-level set of certain constraints.
	For example, the relation $\mathcal{P}$ consisting of pairs $(u, x)$ solving a finite-dimensional system $\dot{x} = Ax + Bu$ is the intersection of the kernels of the continuous maps
	\begin{align*}
		F_t : C^{\infty}_{\local} \times C^{\infty}_{\local} \rightarrow \CC : (u, x) \mapsto Ax(t) + Bu(t) - \dot{x}(t), \qquad \forall t \geq 0.
	\end{align*}
	Therefore, $\mathcal{P}$ is a closed subspace of $C^{\infty}_{\local} \times C^{\infty}_{\local}$.
	}
	
	\change{
	Closedness can be also certified for any relation $\mathcal{Q}$ defined by a finite-dimensional system $\dot{x} = Ax + Bu,\, y = Cx$ with input $u$ and output $y$.
	Write $\mathcal{Q}$ as the sum of vector subspaces
	\begin{align*}
		\mathcal{Q}_i + \mathcal{Q}_f := \{(0, y) \mid y(t) = Ce^{At}x_0 \text{ for some } x_0 \in X\} + \{(u, y) \mid y = Ce^{At}B * u\},
	\end{align*} 
	where $\mathcal{Q}_i$ contains the internal responses and $\mathcal{Q}_f$ contains the forced responses.
	A similar argument shows that $\mathcal{Q}_f$ is closed in $C^{\infty}_{\local} \times C^{\infty}_{\local}$.
	Thus, the quotient space $(C^{\infty}_{\local} \times C^{\infty}_{\local}) / \mathcal{Q}_f$ is Hausdorff.
	Since $\mathcal{Q}_i$ is finite-dimensional, $(C^{\infty}_{\local} \times C^{\infty}_{\local}) / \mathcal{Q}$ is the quotient of $(C^{\infty}_{\local} \times C^{\infty}_{\local}) / \mathcal{Q}_f$ by a finite-dimensional subspace.
	Since finite-dimensional subspaces of Hausdorff spaces are closed \cite[Thm. 1.21]{rudin91}, the quotient space $(C^{\infty}_{\local} \times C^{\infty}_{\local}) / \mathcal{Q}$ is Hausdorff and the subspace $\mathcal{Q}$ is closed.
	}
\end{remark}

Compositions of relations are naturally defined. 
Let $\mathcal{P}$ be a $(U, Y)$-relation and $\mathcal{Q}$ a $(Y, Z)$-relation, their {\em serial composition} is
\begin{align*}
	\mathcal{Q} \circ \mathcal{P} &:= \left\{ (u, z) \mid \;\;
	\begin{aligned}
	&(u, y) \in \mathcal{P}
	\text{ and } (y, z) \in \mathcal{Q} \\ 
	&\text{for some } y \in C^{\infty}_{\local}(\TT, Y) 
	\end{aligned}
	\right\}.
\end{align*}
Assuming that $Z = U$, their closed-loop interconnection is
\begin{align*}
 	\sysclose(\mathcal{P}, \mathcal{Q}) := 
	\left\{ (u, y) \mid \;\;
	\begin{aligned}
	&(u + z, y) \in \mathcal{P}
	\text{ and } (y, z) \in \mathcal{Q} \\ 
	&\text{for some } z \in C^{\infty}_{\local}(\TT, U)
	\end{aligned}
	\right\}.
\end{align*}

We consider a network as a collection of mutually interacting subsystems, where each subsystem is a $(U, Y)$-relation $\mathcal{P}$.
A collection of $n$ subsystem without interaction is the $(U^n, Y^n)$-relation $\mathcal{P}^{\oplus n}$, defined as the $n$-fold Cartesian product of $\mathcal{P}$ in $U^n \times Y^n$.
The interaction is a feedback relation $\mathcal{L} \subseteq Y^n \times U^n$.
The network is the closed-loop system $\sysclose(\mathcal{P}^{\oplus n}, \mathcal{L})$, represented in the block diagram in \cref{fig:network}.

\begin{figure}[t]
    \centering
    \begin{minipage}{0.4\textwidth}
    \centering
    \change{
    \begin{tikzpicture}[scale=0.9]
	\tikzstyle{sys}=[rectangle, draw, minimum width=10mm, minimum height=8mm];
	\node[circle, draw, inner sep=0, radius=4mm] at (-1.5, 0) (add) {$+$};
	\node[sys] at (0,0) (P) {$\mathcal{P}^n$};
	\node[sys] at (0,-1) (Q) {$\mathcal{L}$};
	\coordinate (in) at (-2.5, 0);
	\coordinate (out) at (1.5,0);
	\draw[->] (in) -- node[above]{$\mathbf{u}$} (add);
	\draw[->] (P) -- (out) |- (Q);
	\draw[->] (Q) -| (add);
	\draw[->] (add) -- (P);
	\draw[->] (out) --node[above]{$\mathbf{y}$} ($(out) + (0.5, 0)$);
	\end{tikzpicture}
	}
	\caption{Network block diagram} \label{fig:network}
    \end{minipage} 
    \hspace{0.05\textwidth} %
    \begin{minipage}{0.5\textwidth}
     \centering
     \begin{tikzpicture}[scale=0.9]
		\tikzstyle{sys}=[rectangle, draw, minimum width=10mm, minimum height=8mm];
		\node[sys] at (0,0) (P) {$\mathcal{P}$};
		\node[sys] at (0,-1) (Q) {$\lambda_k$};
		\node[circle, draw, inner sep=0, radius=4mm] at (-1.5, 0) (add) {$+$};
		\node[sys] at (-3,0) (c) {$c_k$};
		\coordinate (out) at (1.5,0);
		\draw[->] (P) -- (out) |- (Q);
		\draw[->] (Q) -| (add);
		\draw[->] (add) -- (P);
		\draw[->] (c) -- (add);	
		\draw[->] (-4.5, 0) --node[above]{$y_{k+1}$} (c);
		\draw[->] (out) --node[above,]{$y_k$} ($(out) + (0.5,0)$);
	\end{tikzpicture}
	 \caption{Representation of $(c_k y_{k+1}, y_k) \in \sysclose(\mathcal{P}, \lambda_k)$} \label{fig:feed}
    \end{minipage}
\end{figure}

\section{Stability and synchronization analysis} \label{sect:io_relation}

This section presents a criterion for synchronization of networks composed of 
\change{$n$ identical copies of a $(U, Y)$-relation\footnote{\change{The relation $\mathcal{P}$ could represent a plant, or the composition of a plant with a controller.}} $\mathcal{P}$}. 
We assume that $U = Y =: X$ and $\mathcal{L} \subseteq X^n \times X^n$ is the $X^n$-relation defined by a linear map $L : \CC^n \rightarrow \CC^n$. 
The latter assumption means that the input of each subsystem is an $\CC$-linear combination of the $n$ outputs.
\change{By an abuse of notation, the induced relation $\mathcal{L} = \{(u, y) \mid y(t) = Lu(t)\}$ is identified with $L$.}
The following result provides a sufficient and a necessary condition for $\sysclose(\mathcal{P}^{\oplus n}, L)$ to be stable.

\begin{proposition} \label{pro:decouple} \stepcounter{refer}
	Given $L \in \CC^{n \times n}$ and an $X$-relation $\mathcal{P}$ such that\footnote{\change{In $\sysclose(\mathcal{P}, \lambda)$, the symbol $\lambda$ represents the relation $\{(u, y) \mid y(t) = \lambda u(t)\}$ induced by the multiplication operator $\lambda : X \rightarrow X$.}} $\sysclose(\mathcal{P}, \lambda)$ is well-defined for every $\lambda \in \CC$.
	Each statement below implies the next.
	\begin{enumerate}[(a)]
		\item\labellocal{item:a} $\sysclose(\mathcal{P}, \lambda)$ is smoothly VIVO for \change{every $\lambda$ in the spectrum $\spec(L)$ of $L$}.
		\item\labellocal{item:b} $\sysclose(\mathcal{P}^{\oplus n}, L)$ is stable.
		\item\labellocal{item:c} $\sysclose(\mathcal{P}, \lambda)$ is stable for every $\lambda \in \spec(L)$.
	\end{enumerate}
\end{proposition}
\begin{proof}
	Let $Q \in \CC^{n \times n}$ be an invertible matrix such that $Q^{-1}LQ$ is in Jordan normal form
	\begin{align*}
		Q^{-1}LQ &= \begin{bmatrix} \lambda_1 & c_1 & & & 0 \\ & \lambda_2 & c_2 & & \\ & & \ddots & \ddots & \\ & & & \lambda_{n-1} & c_{n-1} \\ 0 & & & & \lambda_{n} \end{bmatrix},
	\end{align*}
	where $\lambda_1, \ldots, \lambda_{n}$ are the eigenvalues of $L$.
	By sketching the closed-loop block diagram, it is easy to verify that
	\begin{align} \labellocal{eq:1}
		\begin{aligned}
		\sysclose(\mathcal{P}^{\oplus n}, L) &= Q \circ \sysclose(Q^{-1} \circ \mathcal{P}^{\oplus n} \circ Q,\, Q^{-1}LQ) \circ Q^{-1} \\
		&= Q \circ \sysclose(\mathcal{P}^{\oplus n},\, Q^{-1}LQ) \circ Q^{-1},
		\end{aligned}
	\end{align} 
	where the matrix $Q$ is identified with the $X^n$-relation induced by $Q : \CC^n \rightarrow \CC^n$.
	By the structure of $Q^{-1} L Q$, the outputs $\mathbf{y} = (y_1, \ldots, y_{n})$ of $\sysclose(\mathcal{P}^{\oplus n}, Q^{-1}LQ)$, under zero inputs, have components satisfying 
	\begin{equation} \labellocal{eq:2}
		(c_k y_{k+1}, y_k) \in \sysclose(\mathcal{P}, \lambda_k), 
	\end{equation}
	see \cref{fig:feed}.
	
	
	
	If \reflocal{item:a} holds, then every $y_k$ must be smoothly vanishing by backward induction on $k$, which implies \reflocal{item:b}.
	Furthermore, for every $\lambda \in \spec(L)$, there exists $k$ such that $\lambda_k = \lambda$ and $c_k = 0$.
	Since $\sysclose(\mathcal{P}, \lambda)$ is well-defined, any output $y_k$ of $\sysclose(\mathcal{P}, \lambda_k)$ extend to $(y_1, \ldots, y_n)$ whose components verify \creflocal{eq:2}.
	Therefore, $\sysclose(\mathcal{P}^{\oplus n}, Q^{-1}LQ)$ is not stable if any of $\sysclose(\mathcal{P}, \lambda_k)$ is not stable.
	By \creflocal{eq:1}, stabilitiy of $\sysclose(\mathcal{P}^{\oplus n}, Q^{-1}LQ)$ is equivalent to stability of $\sysclose(\mathcal{P}^{\oplus n}, L)$, hence \reflocal{item:b} implies \reflocal{item:c}.
\end{proof}

\begin{remark}
	\change{
	It is worth noting that \cref{pro:decouple} also holds with ``smoothly VIVO'' replaced by ``VIVO''.
	We formulate the result with ``smoothly VIVO'' because there are smoothly VIVO systems which are not VIVO.
	An example is the class of equations studied in \Cref{sect:parabolicsystem_boundaryinput} (see \cref{thm:parabolic_neumann_input}).
	In general, PDEs with boundary input or output are not VIVO, because they can be represented as relations defined by $\dot{x} = Ax + \sum_{j=0}^{1} B_j \partial_t^j u$, $y = \sum_{j=0}^{2} C_j \partial_t^j x + \sum_{j=0}^{1} D_j \partial_t^j u$ with bounded operators $B_j, C_j, D_j$ \cite[Proposition 2.3.25]{xia24}.
	It is intuitive that vanishing of $u$ does not guarantee vanishing of $y$; we must require vanishing of several time-derivatives of $u$.
	}
\end{remark}

Note that, in order to guarantee an equivalence between the properties (a) (b) and (c), we would need to add an implication between (c) and (a) in \cref{pro:decouple}. 
\change{This implication holds for finite-dimensional systems (see \cref{ex:input_resilience}), but not for general infinite-dimensional systems, as illustrated by the following example.}

\begin{example}
	Let $A : L^1([0,1]) \rightarrow L^1([0,1])$ be the multiplication operator taking $f$ to $g(\xi) = -\xi f(\xi)$.
	Consider the $L^1([0,1])$-relation defined by $\dot{x} = Ax + u$ with input $u$ and output $x$.
	This relation is stable.
	However, an explicit computation of the solution shows that $x$ does not vanish smoothly under the smoothly vanishing input $u(t) = e^{-\xi t}$.
\end{example}

\change{\cref{pro:decouple} can be stated as an equivalence, assuming the implication between (c) and (a).}

\begin{definition} \label{def:prop} 
	A relation $\mathcal{P}$ is \define{\prop{}} (in $X$) if it is well-defined and 
	if being stable (in $X$) implies being smoothly VIVO (in $X$). 
\end{definition}  

\begin{theorem}[Stability criterion] \label{cor:stab} \stepcounter{refer}
\hspace{-2mm}Let $\mathcal{P}$ be an $X$-relation such that $\sysclose(\mathcal{P}, \lambda)$ is \prop{} for all $\lambda \in \CC$,
then $\sysclose(\mathcal{P}^{\oplus n}, L)$ is stable in $X^n$ if and only if $\sysclose(\mathcal{P}, \lambda)$ is stable in $X$ for every $\lambda \in \spec(L)$.
\end{theorem}

%
%

Synchronization of a network is the stability of the diagonal subspace, hence it can also be related to stability properties of the subunits.

\begin{definition} \stepcounter{refer}
	An $X^n$-relation $\mathcal{P}$ \define{synchronizes} if every $y \in C^{\infty}_{\local}(\TT, X^n)$ such that $(0, y) \in \mathcal{P}$ converges\footnote{$y$ converging to a subset $S$ of a topological vector space $X$ means that for every zero neighbourhood $V$ of $X$, $y$ eventually remains in the set $S + V$.} to the diagonal subspace $\Delta_{X} := \{(x, \ldots, x) \mid x \in X\}$.
\end{definition}

\begin{theorem}[Synchronization criterion] \label{cor:sync} \stepcounter{refer}
	Let $\mathcal{P}$ be an $X$-relation such that $\sysclose(\mathcal{P}, \lambda)$ is \prop{} for all $\lambda \in \CC$,
	and $L \in \CC^{n \times n}$ be a matrix having eigenvector $\mathbf{1}_n := (1, \ldots, 1)$ associated to some eigenvalue $\lambda_1$,
	then $\sysclose(\mathcal{P}^{\oplus n}, L)$ synchronizes in $X^n$ if and only if $\sysclose(\mathcal{P}, \lambda)$ is stable for every $\lambda \in \spec(L)$ excluding one instance of $\lambda_1$.
\end{theorem}
\begin{proof}
	Let $Q \in \CC^{n \times n}$ be an invertible matrix whose first column is $\mathbf{1}_n$, and $\pi$ be the projection of $\CC^n = \CC \times \CC^{n-1}$ onto the second component.
	Since $\Ker (\pi Q^{-1}) = \Ima \mathbf{1}_n$, any function $y \in C^{\infty}_{\local}(\RRge{0}, X^n)$ synchronizes iff $\pi Q^{-1} y$ vanishes.
	Therefore, $\sysclose(\mathcal{P}^{\oplus n}, L)$ synchronizes iff $\pi Q^{-1} \circ \sysclose(\mathcal{P}^{\oplus n}, L)$ is stable.
	We may write
	\begin{align*}
		\pi Q^{-1} \circ \sysclose(\mathcal{P}^{\oplus n}, L) &= \pi \circ \sysclose(\mathcal{P}^{\oplus n}, Q^{-1}LQ) \circ Q^{-1}.
	\end{align*}
	Since the first column of $Q$ is the eigenvector $\mathbf{1}_n$ of $L$, the matrix $Q^{-1}LQ$ is a block upper-triangular matrix with blocks $\lambda_1$ and $\tilde{L} \in \CC^{(n-1) \times (n-1)}$.
	This implies that $\pi \circ \sysclose(\mathcal{P}^{\oplus n}, Q^{-1}LQ) = \sysclose(\mathcal{P}^{\oplus (n-1)}, \tilde{L}) \circ \pi$.
	Synchronization of $\sysclose(\mathcal{P}^{\oplus n}, L)$ is therefore equivalent to stability of $\sysclose(\mathcal{P}^{\oplus (n-1)}, \tilde{L})$.
	The conclusions follow by \cref{pro:decouple} and the fact that $\spec(L) = \spec(\tilde{L}) \cup \{\lambda_1\}$.
\end{proof}

\change{We point out that \cref{cor:stab} and \cref{cor:sync} hold for finite-dimensional systems because these systems are \prop{}.}

\begin{example} \label{ex:input_resilience}
	\change{
	We claim that any relation $\mathcal{P}$ defined by a finite-dimensional observable\footnote{\change{A system $\dot{x} = Ax + Bu, \; y = Cx$ and its observable subsystem always define the same relation. However, the same property does not hold for the controllable subsystem because the relation contains input-output pairs associated to non-zero initial conditions.}} system $\dot{x} = Ax + Bu, \; y = Cx$ is input-resilient.
	Stability of $\mathcal{P}$ implies that the eigenvalues of $A$ are contained in $\CCl{0}$.
	In view of the formula $y(t) = Ce^{At} x(0) + \int_{0}^{t} Ce^{A(t-\tau)}B u(\tau) d\tau$, 
	\begin{align*}
		\|y(t)\| &\lesssim \|e^{At}\| \|x(0)\| + \|e^{At / 2}\| \int_{0}^{t/2} \|e^{A(t/2-\tau)}\| d\tau \cdot \sup_{\tau \in [0,t/2]} \|u(\tau)\| \\
		&\qquad + \int_{t/2}^{t} \|e^{A(t-\tau)}\| d\tau \cdot \sup_{\tau \in [t/2,t]} \|u(\tau)\| \qquad \text{uniformly in } t \geq 0.
	\end{align*}
	Since $\|e^{At}\|$ decays exponentially, we see that $\mathcal{P}$ is VIVO.
	By \cref{pro:relation_closed}, $\mathcal{P}$ is also smoothly VIVO.
	}
	
	\change{Since $\mathcal{P}$ is \prop{}, \cref{cor:stab} and \cref{cor:sync} imply:
	\begin{itemize}
		\item $\sysclose(\mathcal{P}^{\oplus n}, L)$ is stable if and only if $A + \lambda BC$ is stable for every $\lambda \in \spec(L)$;
		\item Assuming that $L\mathbf{1}_n = \lambda_1 \mathbf{1}_n$, then $\sysclose(\mathcal{P}^{\oplus n}, L)$ synchronizes if and only if $A + \lambda BC$ is stable for every $\lambda \in \spec(L)$ excluding $\lambda_1$.
	\end{itemize}
	Note that these statements characterize stability and synchronization of the outputs of the observable system $\dot{x} = Ax + Bu$, $y = Cx$.
	Characterizations of the states can be obtained replacing $C$ by the identity and $B$ by $BC$.}
\end{example}

\change{We will leverage the arguments in \cref{ex:input_resilience} to establish \propnoun{} for several classes of infinite-dimensional systems in the next sections.}

\section{Relations defined by semigroups} \label{sect:input_resilient} 

We have seen in the previous section that \propnoun{} allows network stability and synchronization to be determined by a stability analysis of the isolated units. 
In this section, we will provide conditions under which abstract Cauchy problems define \prop{} relations.
These conditions will be used in the next sections for parabolic equations and delayed ordinary differential equations.

The relations we analyze are (continuous and discrete-time) relations defined by \cref{eq:transition_system} and \cref{eq:transition_io_system}.
They are shown to be well-defined (\cref{cor:transition_welldefined}), and \prop{} under certain compactness assumption on the transition semigroups (\cref{pro:transition_input_resilient}, \cref{pro:transition_io_input_resilient}).
A sufficient condition based on the semigroup generator is also presented (\cref{pro:acp_resilient}). 
We remind the definition of strongly continuous semigroup. 
\begin{definition} \label{def:semigroup_sc} \stepcounter{refer}
	A \define{(strongly continuous) semigroup} on a Banach $X$ is a mapping $\Phi : \TT \rightarrow \Endom(X)$, satisfying\footnote{$\Endom(X)$ is the space of continuous operators on $X$.}
	\begin{itemize} 
		\item(Semigroup properties) $\Phi(0) = \identity_{X}$ and $\Phi(t + s) = \Phi(t) \Phi(s)$ for all $t, s \in \TT$.
		\item(Strong continuity\footnote{The strong continuity assumption is vacuous when $\TT = \ZZge{0}$.}) $\Phi : \TT \rightarrow \Endom(X)$ is continuous when $\Endom(X)$ is given the strong operator topology (i.e., for all $x \in X$, the function $\TT \rightarrow X : t \mapsto \Phi(t)x$ is continuous).
	\end{itemize}
	If $\TT = \RRge{0}$, the \define{generator} of $\Phi$ is the operator $A$ defined by
	\[
		\domain(A) = \left\{x \in X \mid \lim_{t \rightarrow 0^+} \frac{\Phi(t) x - \Phi(0) x}{t} \text{ exists} \right\},\quad  Ax = \lim_{t \rightarrow 0^+} \frac{\Phi(t)x - \Phi(0)x}{t}.
	\]
	$\Phi(t)$ with generator $A$ is also denoted by $e^{At}$.
\end{definition}

We start our analysis by considering relations $\mathcal{P}$ defined as smooth solutions $(u, x) \in C^{\infty}_{\local}(\TT, X) \times C^{\infty}_{\local}(\TT, X)$ of
\begin{align} \label{eq:transition_system}
	x(t) &= \Phi(t) x(0) + \int_{(0, t]} \Phi(t-s) u(s) ds \qquad \forall t \in \TT,
\end{align}
where $\Phi : \TT \rightarrow \Endom(X)$ is a semigroup on a Banach space. 
In the case $\TT = \ZZge{0}$, the integral is taken with respect to the counting measure, and $\mathcal{P}$ coincides with the discrete-time relation defined by the difference equation $x(k+1) = \Phi(1)x(k) + u(k+1)$. 
In the case $\TT = \RRge{0}$, we will see that $\mathcal{P}$ is the continuous-time relation defined by the differential equation $\dot{x} = Ax + u$, where $A$ is the generator of $\Phi$. 


The strong continuity assumption in \cref{def:semigroup_sc} ensures that given any $u \in C^{0}_{\local}(\RRge{0}, X)$ and $x(0) \in X$, the function $x$ defined by \cref{eq:transition_system} belongs to $C^{0}_{\local}(\RRge{0}, X)$. Continuity of $x$ follows from continuity of $\Phi(t-s) u(s)$ in $(t, s) \in \{0 \leq s \leq t\}$, which is a consequence of the following well-known lemma.

\begin{lemma}[{\cite[Thm. 2.2 in Chpt. 1]{pazy83}}] \label{lem:semigroup_bound} \stepcounter{refer}
	Let $\Phi : \TT \rightarrow \Endom(X)$ be a (strongly continuous) semigroup on a Banach space $X$, then $\sup_{t \in K} \lVert \Phi(t) \rVert < \infty$ for every compact $K \subseteq \TT$.
\end{lemma}

Another consequence of \cref{lem:semigroup_bound} is 

\begin{proposition} \label{pro:transition_closed} \stepcounter{refer}
	The relation defined by \cref{eq:transition_system}, where $\Phi$ is a semigroup, is closed.
\end{proposition}
\begin{proof}
	For every $t \in \TT$, the map $C^{0}_{\local}(\TT, X) \times C^{0}_{\local}(\TT, X) \rightarrow X$ sending $(u, x)$ to $x(t) - \Phi(t) x(0) - \int_{[0,t]} \Phi(t-s) u(s) ds$ is continuous. 
	The relation is the intersection of the kernels of these continuous maps, restricted to $C^{\infty}_{\local} \times C^{\infty}_{\local}$.
\end{proof}

Next, we show that the continuous-time relation induced by \cref{eq:transition_system} coincides with the relation induced by the abstract Cauchy problem $\dot{x} = Ax + u$, where $A$ is the generator of $\Phi$. 
\change{
Some terminologies are needed to define the latter relation.
A pair $(u, x) \in C^{0}_{\local}(\RRge{0}, X) \times C^{0}_{\local}(\RRge{0}, X)$ is called a \define{classical admissible pair} of $\dot{x} = Ax + u$ if $x \in C^1_{\local}(\RRge{0}, X)$, $x(t) \in \domain(A)$, and the equation holds for all $t \geq 0$;
and $(u, x)$ is called a \define{mild admissible pair} of $\dot{x} = Ax + u$ if it satisfies \cref{eq:transition_system}.
These notions are related to \define{classical (resp. mild) solutions} of abstract Cauchy problems (e.g. \cite[Chpt. 3]{curtain95}):
a pair $(u, x)$ is a classical (resp. mild) admissible pair iff $u$ is continuous and $x$ is a classical (resp. mild) solution of $\dot{x} = Ax + u$.
}

\change{
The relation induced by $\dot{x} = Ax + u$ is defined as the set of classical admissible pairs $(u, x)$ in $C^{\infty}_{\local} \times C^{\infty}_{\local}$, 
and the relation induced by \cref{eq:transition_system} coincides with the set of mild admissible pairs $(u, x)$ in $C^{\infty}_{\local} \times C^{\infty}_{\local}$. 
}%
\begin{shownto}{SIAM}
	\change{The two relations are equal by the following standard fact from semigroup theory, whose proof is available in \cite[Lemma 4.4]{xia23}.}
\end{shownto}
\begin{shownto}{arXiv}
	The two relations are equal by the following standard fact from semigroup theory.
\end{shownto}

\begin{lemma} \label{pro:acp_classical_mild} \stepcounter{refer}
	Consider the abstract Cauchy problem $\dot{x} = Ax + u$ where $A$ generates the semigroup $e^{At}$.
	Every classical \change{admissible pair} $(u, x) \in C^{0}_{\local}(\RRge{0}, X) \times C^{0}_{\local}(\RRge{0}, X)$ is a mild \change{admissible pair}.
	Furthermore, the following statements are equivalent for every mild \change{admissible pair} $(u, x) \in \allowbreak C^{1}_{\local}(\RRge{0}, X) \times C^{0}_{\local}(\RRge{0}, X)$:
	\begin{enumerate}[(a)]
		\item\labellocal{item:a} $x \in C^{1}_{\local}(\RRge{0}, X)$.
		\item\labellocal{item:b} $x(0) \in \domain(A)$.
		\item\labellocal{item:c} $(u, x)$ is a classical \change{admissible pair}.
	\end{enumerate}
\end{lemma}

\begin{shownto}{arXiv}
\begin{proof}
	Let $(u, x)$ be a classical \change{admissible pair}, the function $[0, t] \rightarrow X : s \mapsto e^{A(t-s)} x(s)$ is continuous in $s$ and its derivative is $e^{A(t-s)} u(s)$ for every $s \in (0, t)$.
	\change{
	Integrating the derivative on $[0,t]$ yields $e^{A(t-t)}x(t) - e^{At}x(0)$, hence
	\begin{align} \labellocal{eq:1}
		x(t) &= e^{At} x(0) + \int_{0}^{t} e^{A(t-s)} u(s) ds \qquad \forall t \geq 0.
	\end{align}
	The pair $(u,x)$ therefore verifies the definition of a mild admissible pair.
	}
	
	\reflocal{item:a} $\Rightarrow$ \reflocal{item:b}. 
	\change{A mild admissible pair $(u, x)$ verifies \creflocal{eq:1}.
	The integral term is continuously differentiable because}
	\begin{align} \labellocal{eq:2}
		\partial_t \int_{0}^{t} e^{A(t-s)} u(s) ds = \partial_t \int_{0}^{t} e^{As} u(t - s) ds = e^{At} u(0) + \int_{0}^{t} e^{As} \dot{u}(t - s) ds
	\end{align}
	\change{is continuous in $t$.
	Assuming (a), the term $e^{At}x(0)$ in \creflocal{eq:1} must also be continuously differentiable.
	Differentiability of $e^{At}x(0)$ at $t = 0$ is equivalent to $x(0) \in \domain(A)$.
	}
	
	\reflocal{item:b} $\Rightarrow$ \reflocal{item:c}.
	\change{
	We know that $(u, x)$ solves \creflocal{eq:1}.
	By (b), 
	\begin{align*}
		Ae^{At}x(0) &= \lim_{h \rightarrow 0^+} \frac{e^{A(t+h)}x(0) - e^{At}x(0)}{h} = \partial_t e^{At}x(0).
	\end{align*}
	Furthermore, simple computations show
	\begin{align*}
		A \int_{0}^{t} e^{A(t-s)} u(s) ds &= \lim_{h \rightarrow 0^+} \frac{1}{h} \left( \int_{-h}^{t-h} e^{A(t-s)} u(s+h) ds - \int_{0}^{t} e^{A(t-s)} u(s) ds \right) \\
		&= e^{At} u(0) - u(t) + \int_{0}^{t} e^{A(t-s)} \dot{u}(s) ds \\
		&= -u(t) + \partial_t \int_{0}^{t} e^{A(t-s)} u(s) ds,
	\end{align*}
	where the second line uses the decomposition $\int_{-h}^{t-h} = \int_{-h}^{0} - \int_{t-h}^{t} + \int_{0}^{t}$, and the third line follows from \creflocal{eq:2}.
	We conclude that $Ax(t)$ exists and $\dot{x}(t) = Ax(t) + u(t)$.
	}
\end{proof}
\end{shownto}

A consequence of \cref{pro:acp_classical_mild} is the following characterization for smoothness of the mild \change{admissible pairs} of $\dot{x} = Ax + u$.

\begin{lemma} \label{pro:acp_smoothness} \stepcounter{refer}
	\change{Let $(u, x)$ be a mild admissible pair of $\dot{x} = Ax + u$, where $A$ generates a semigroup and $u \in C^{\infty}_{\local}(\RRge{0}, X)$.}
	For any $k \in \ZZge{1}$, $x \in C^{k}_{\local}(\RRge{0}, X)$ if and only if $(u, x)$ satisfies the $(k-1)$th-order compatibility condition:
	\begin{align*}
		x_0 &:= x(0) \in \domain(A), \\
		x_1 &:= Ax_0 + u(0) \in \domain(A), \\
		&\vdots \\
		x_k &:= Ax_{k-1} + \partial_t^{k-1} u(0) \in \domain(A).
	\end{align*}
\end{lemma}
\begin{proof}
	The case $k = 1$ is proved in \cref{pro:acp_classical_mild}. 
	Observe that if $(u, x) \in C^{1}_{\local} \times C^{1}_{\local}$ solves \cref{eq:transition_system}, then $(v, y) = (\dot{u}, \dot{x})$ is a mild \change{admissible pair} of $\dot{y} = Ay + v$ with initial condition $y(0) = Ax(0) + u(0)$.
	The conclusion follows by induction on $k$.
\end{proof}


Leveraging \cref{pro:acp_smoothness}, we show that relations defined by the abstract Cauchy problems $\dot{x} = Ax + u$ are well-defined.

\begin{proposition} \label{pro:acp_smooth_dense} \stepcounter{refer}
	Let $\mathcal{P}$ be the $X$-relation defined by $\dot{x} = Ax + u$, where $A$ generates a semigroup.
	For any $u \in C^{\infty}_{\local}(\RRge{0}, X)$, the set $S_{u} := \{x(0) \mid (u, x) \in \mathcal{P}\}$ is dense in $X$.
\end{proposition}
\begin{proof}
	The proof is divided into three cases: a) $u = 0$; b) $u(t) = 0$ near $t = 0$; c) general $u$.
	\change{
	For any $\bar{x} \in X$ and\footnote{The inclusion $\eta \in C^{\infty}_{c}(\RRg{0}, \RR)$ means that $\eta$ is supported on a compact subset of $\RRg{0}$, so it is zero on near $t = 0$.} $\eta \in C^{\infty}_{c}(\RRg{0}, \RR)$,
	define $x(t) = \int_{0}^{\infty} \eta(s) e^{A(t+s)} \bar{x} ds$.
	The pair $(0, x)$ is clearly a mild admissible pair of $\dot{x} = Ax + u$.
	Since $\bar{x} \in X$ is arbitrary and $\eta$ can be an approximate identity, $S_0$ is dense in $X$, proving case (a).
	By \cref{pro:acp_smoothness}, smoothness of a mild admissible pair $(u, x)$ depends only on its behaviour near $t = 0$.
	Therefore, for any $\bar{x} \in S_0$ and any smooth input $u$ equal to zero near $t = 0$, the unique mild admissible pair $(u, x)$ with initial condition $x(0) = \bar{x}$ is smooth.
	We conclude that $S_0 \subseteq S_u$, which proves (b).
	}
	In case (c), 
	extend $u$ to an element of $C^{\infty}_{\local}(\RR, X)$, also denoted by $u$ (the extension exists by 
	\begin{shownto}{SIAM}
		Borel's lemma \cite[Lemma 8.5]{xia23} 
	\end{shownto}
	\begin{shownto}{arXiv}
		\cref{lem:borel} in \Cref{sect:appendix}
	\end{shownto}
	). 
	For any $\epsilon > 0$, let $\eta_{\epsilon}$ be a smooth bump function equal to $0$ on $\RRle{-\epsilon}$ and equal to $1$ on $\RRge{0}$. Since $\eta_{\epsilon}(t)u(t) = 0$ near $t = -\epsilon$ and solutions $x \in C^{\infty}_{\local}(\RRge{-\epsilon}, X)$ of $\dot{x} = Ax + \eta_{\epsilon}u$ restrict to smooth solutions on $\RRge{0}$, we deduce that $e^{A\epsilon}a + \int_{-\epsilon}^{0}e^{-As} \eta_{\epsilon}(s)u(s) ds \in S_u$ for all $x(-\epsilon) \in S_{0}$. 
	Since $\lim_{\epsilon \rightarrow 0} e^{A\epsilon} a = a$ and 
	\begin{align*}
		\left\| \int_{-\epsilon}^{0}e^{-As} \eta_{\epsilon}(s)u(s) ds \right\| &\lesssim \epsilon \|u\|_{L^1([-1,0], X)} \qquad \text{uniformly in } \epsilon \in (0, 1)
	\end{align*}
	by \cref{lem:semigroup_bound}, we conclude that the closure of $S_{u}$ contains $S_{0}$.
\end{proof}

\begin{corollary} \label{cor:transition_welldefined} \stepcounter{refer}
	Any discrete-time or continuous-time relation defined by \cref{eq:transition_system} is well-defined.
\end{corollary}

The next result guarantees that, if the transition semigroup is compact, the relation defined by \cref{eq:transition_system} is \prop{}.    

\begin{theorem} \label{pro:transition_input_resilient} \stepcounter{refer}
	Let $\mathcal{P}$ be the relation defined by \cref{eq:transition_system},
	where $\Phi : \TT \rightarrow \Endom(X)$ is a semigroup.
	If $\Phi(1)$ is compact, then the following are equivalent:
	\begin{enumerate}[(a)]
		\item\labellocal{item:a} $\mathcal{P}$ is stable.
		\item\labellocal{item:b} If $(0, x) \in \mathcal{P}$ and $x(\cdot + 1) = \lambda x(\cdot)$ for some $|\lambda| \geq 1$, then $x = 0$ as a function.
		\item\labellocal{item:c} $\Phi(1)$ has no eigenvalue outside the open unit disk.
		\item\labellocal{item:d} $\mathcal{P}$ is VIVO.
		\item\labellocal{item:e} $\mathcal{P}$ is smoothly VIVO.
	\end{enumerate}
\end{theorem}
\begin{proof}
	\reflocal{item:b} $\Rightarrow$ \reflocal{item:c}.
	Assume for a contradiction that $\Phi(1)$ has an eigenvector $\bar{x} \in X$ corresponding to an eigenvalue $|\lambda| \geq 1$.
	In the case $\mathcal{P}$ is discrete-time, $x(k) = \Phi(k) \bar{x} = \lambda^k \bar{x}$ satisfies $(0, x) \in \mathcal{P}$.
	In the case $\mathcal{P}$ is continuous-time, define the continuous function $\tilde{x}(t) = \Phi(t)\bar{x}$.
	Pick any $\eta \in C^{\infty}_{c}(\RRg{0}, \RR)$, define the mollified function $x(t) = \int_{0}^{\infty} \tilde{x}(t + s) \eta(s) ds$, then $(0, x) \in \mathcal{P}$.
	Since $\tilde{x} \neq 0$, there exists $\eta$ such that $x \neq 0$.
	
	\reflocal{item:c} $\Rightarrow$ \reflocal{item:d}.
	Since non-zero spectra of compact operators are eigenvalues \cite[VII.7.1]{conway90}, the spectral radius of $\Phi(1)$ is strictly less than $1$. 
	By the spectral radius formula \cite[VII.3.8]{conway90}, $\lVert \Phi(k) \rVert^{1/k}$ converges to the spectral radius, hence the function $\ZZge{0} \rightarrow \RR : k \mapsto \lVert \Phi(k) \rVert$ decays exponentially. 
	By \cref{lem:semigroup_bound}, $\TT \rightarrow \RR : t \mapsto \lVert \Phi(t) \rVert$ decays exponentially \change{and belongs to $L^1$}.
	Because the convolution of an $L^1$ function with a vanishing function is vanishing, we conclude \change{by \cref{eq:transition_system}} that $\mathcal{P}$ is VIVO.
	
	\reflocal{item:d} $\Rightarrow$ \reflocal{item:e}
	follows from \cref{pro:relation_closed} \change{because $\mathcal{P}$ is closed} by \cref{pro:transition_closed}.
\end{proof}

Now we consider the following extension of \cref{eq:transition_system}
\begin{align} \label{eq:transition_io_system}
	\begin{aligned}
	x(t) &= \Phi(t) x(0) + \int_{(0, t]} \Phi(t-s) B u(s) ds, \\
	y(t) &= C x(t) + D u(t).
	\end{aligned}
\end{align}
where $\Phi : \TT \rightarrow \Endom(X)$ is a semigroup, and $B$, $C$, $D$ are continuous linear operators.
The relation $\mathcal{P}$ defined by \cref{eq:transition_io_system} with input $u$ and output $y$ consists of pairs $(u, y)$ that appear in smooth solutions $(u, x, y)$ of \cref{eq:transition_io_system}.
We note that $\mathcal{P}$ is the relation defined by $x(k + 1) = \Phi(1) x(k) + B u(k+1),\, y(k) = Cx(k) + Du(k)$ when $\TT = \ZZge{0}$. If $\TT = \RRge{0}$, $\mathcal{P}$ is the relation defined by 
\begin{align}\label{eq:io cauchy}
		\left\{
		\begin{aligned}
			\dot{x} &= Ax + Bu, \\
			y &= Cx +Du,
		\end{aligned}
		\right.
	\end{align}
where $A$ is the generator of $\Phi$.
Since $\mathcal{P}$ is the composition of the relation defined by \cref{eq:transition_system} with relations defined by continuous linear operators, $\mathcal{P}$ remains well-defined and a result analogous to \cref{pro:transition_input_resilient} is obtained.

\begin{theorem} \label{pro:transition_io_input_resilient} \stepcounter{refer}
	Let $\mathcal{P}$ be the relation defined by \cref{eq:transition_io_system},
	where $\Phi$ is a semigroup and $B$, $C$, $D$ are continuous linear operators.
	If $\Phi(1)$ is compact, then the following are equivalent:
	\begin{enumerate}[(a)]
		\item\labellocal{item:a} $\mathcal{P}$ is stable.
		\item\labellocal{item:c} \change{Every generalized eigenvector of $\Phi(1)$ associated to an eigenvalue outside the open unit disk} belongs to the unobservable subspace\footnote{Unobservable subspace of $(\Phi, C)$ is the largest subspace of $\Ker C$ invariant under every $\Phi(t)$.} of $(\Phi, C)$.
		\item\labellocal{item:d} $\mathcal{P}$ is VIVO.
		\item\labellocal{item:e} $\mathcal{P}$ is smoothly VIVO.
	\end{enumerate}
\end{theorem}
\begin{proof}
	We first consider the case when $(\Phi, C)$ is observable\footnote{$(\Phi, C)$ is observable if its unobservable subspace is trivial.}.
	In this case, \reflocal{item:c} is equivalent to $\Phi(1)$ having no eigenvalue outside the open unit disk. 
	We show that this is equivalent to the remaining statements. 
	If $\Phi(1)$ has no eigenvalue outside the open unit disk, then by \cref{pro:transition_input_resilient}, the relation $\mathcal{Q}$ defined by \cref{eq:transition_io_system} with input $u$ and output $x$ is stable, VIVO, and smoothly VIVO.
	It is then clear that $\mathcal{P}$ is also stable, VIVO, and smoothly VIVO. Conversely, if $\Phi(1)$ has an eigenvalue outside the open unit disk, then there exists non-zero $x \in C^{\infty}_{\local}(\RRge{0}, X)$ such that $(0, x) \in \mathcal{Q}$ and $x(t + 1) = \lambda x(t)$ for some $|\lambda| \geq 1$.
	Set $y(t) = Cx(t) = C\Phi(t)x(0)$, then $(0, y) \in \mathcal{P}$ and by observability, $y(\bar{t}) \neq 0$ for some $\bar{t}$.
	Since $y(t+k) = \lambda^k y(t)$, $y$ is non-vanishing.
	
	We consider now the general case, where the observability of $(\Phi, C)$ is not required. Let $V$ be the unobservable subspace of $(\Phi, C)$.
	The relation $\mathcal{P}$ defined by $(\Phi, B, C, D)$ coincides with the relation defined by $(\tilde{\Phi}, \tilde{B}, \tilde{C}, D)$, where $\tilde{\Phi}(t)$, $\tilde{B}$, and $\tilde{C}$ are given by the commutative diagram 
	\begin{equation*}
		\begin{tikzcd} 
			U \arrow[r, "B"] \arrow[dr, "\tilde{B}"'] & X \arrow[r, "\Phi(t)"] \arrow[d, two heads] & X \arrow[d, two heads] \arrow[r, "C"] & Y \\
			& X/V \arrow[r, "\tilde{\Phi}(t)"] & X/V \arrow[ur, "\tilde{C}"']
		\end{tikzcd} 
	\end{equation*}
	By \cref{lem:compact_quotient} below, \reflocal{item:c} is equivalent to $\tilde{\Phi}(1)$ having no eigenvalue outside the open unit disk.
	Furthermore, $\tilde{\Phi}(1)$ is compact and $(\tilde{\Phi}, \tilde{C})$ is observable.
	Therefore, the equivalence follows from the observable case.
\end{proof}

\begin{lemma} \label{lem:compact_quotient} \stepcounter{refer}
	Let $A$ be a compact operator on a Banach space $X$, $V$ be a closed $A$-invariant subspace, and $\tilde{A}$ be the quotient operator defined by the commutative diagram 
\begin{equation}\label{eq:comm}
   \begin{tikzcd} 
			X \arrow[r, "A"] \arrow[d, two heads] & X \arrow[d, two heads] \\
			X/V \arrow[r, "\tilde{A}"] & X/V
		\end{tikzcd}
\end{equation}
	
	The following statements hold true.
	\begin{enumerate}[(a)]
		\item\labellocal{item:a} $\tilde{A}$ is compact.
		\item\labellocal{item:b} $\spec(\tilde{A}) \subseteq \spec(A) \cup \{0\}$.
		\item\labellocal{item:c} $\lambda \in \spec(\tilde{A}) \setminus\{0\}$ if and only if $A$ has a generalized eigenvector $v \in \bigcup_{k} \Ker (\lambda - A)^k$ outside of $V$.
	\end{enumerate}
\end{lemma}
\begin{proof}
	Proof of \reflocal{item:a} follows from the commutative diagram \cref{eq:comm}.

	To prove \reflocal{item:b}, assume $\lambda \not\in \spec(A) \cup \{0\}$, then $\Ima(\lambda - A) = X$.
	This implies that $\Ima(\lambda - \tilde{A}) = X/V$, and $\Ker(\lambda - \tilde{A}) = \{0\}$ by properties of compact operators \cite[VII.7.10]{conway90}.
	Therefore, $\lambda \not\in \spec(\tilde{A})$.
	
Now we will prove \reflocal{item:c}. If $(\lambda - A)^k v = 0$ for some $v \not\in V$, then $(\lambda - \tilde{A})^k v = 0$ in $X/V$, which means $\lambda \in \spec(\tilde{A})$. Conversely, if $W = \bigcup_{k} \Ker(\lambda - A)^k \subseteq V$, then $\tilde{A}$ is a quotient of the quotient operator $A'$ defined in the commutative diagram 
\begin{equation*}
\begin{tikzcd} 
			X \arrow[d, "A"] \arrow[r, two heads] & X/W \arrow[d, "A'"] \arrow[r, two heads] & X/V \arrow[d, "\tilde{A}"] \\
			X \arrow[r, two heads] & X/W \arrow[r, two heads] & X/V
\end{tikzcd}
\end{equation*}
	By the spectral theorem of compact operators \cite[VII.4.5]{dunford58}, $W$ is the image of the spectral projection of $A$ to $\{\lambda\}$.
	Thus, $W$ has a closed $A$-invariant complement (i.e., the image of the spectral projection to $\spec(A) \setminus \{\lambda\}$).
	Therefore, $\spec(A')$ coincides with the spectrum of the restriction of $A$ to this complement. 
	We conclude that $\spec(A') = \spec(A) \setminus \{\lambda\}$, and by \reflocal{item:b}, that $\lambda \not\in \spec(\tilde{A})$.
\end{proof}

\change{Note that \cref{pro:transition_io_input_resilient} automatically implies \propnoun{} of finite-dimensional systems because finite-dimensional spaces are compact.}
Recall that the continuous-time relation $\mathcal{P}$ defined by \cref{eq:transition_io_system} are equivalent to the relation defined by the differential equation \cref{eq:io cauchy}.
\change{Next, we provide a sufficient condition for \propnoun{} of $\mathcal{P}$ in terms of the generator $A$ of the semigroup.}
We now recall the definition of \define{sectorial operators}. 

\begin{definition} \stepcounter{refer}
	An operator $A$ on a Banach space $X$ is sectorial if it is densely-defined, its resolvent set contains $\omega + \sectl{\theta}$ for some $\theta > \pi/2$ and $\omega \in \RR$, and
	\begin{align*}
		\lVert (\lambda - A)^{-1}\rVert &\lesssim \frac{1}{\lvert \lambda - \omega \rvert} \qquad \text{uniformly in } \lambda \in \omega + \sectl{\theta}.
	\end{align*}
	Note that $\lambda$ in $(\lambda - A)$ is the image of $\lambda \in \CC$ via the canonical embedding $\CC \rightarrow \Endom(X) : \lambda \mapsto \lambda \cdot \identity_{X}$.
\end{definition}

Every sectorial operator $A$ is guaranteed to generate a semigroup $e^{At}$ (\cite[Thm. 2.4 in Chpt. 2]{pazy83} or \cite[Prop. 4.3 in Chpt. 2]{engel00}), which can be expressed explicitly by the functional calculus formula
\begin{align} \label{eq:sectorial_calculus}
	e^{At} &= \frac{1}{2\pi i} \int_{\Gamma} e^{tz} (z - A)^{-1} dz,
\end{align}
where $\Gamma$ is the negatively-oriented boundary of a sector $\omega + \sectl{\theta}$ whose closure is contained in the resolvent set of $A$ with the restriction $\theta \in (\pi/2, \pi]$. The following technical lemma is required in the proof of \cref{pro:acp_resilient}.

\begin{lemma} \label{lem:semigroup_pointspec} \stepcounter{refer}
	Let $A$ be the generator of a semigroup on a Banach space $X$.
	Fix any $a \in X$, $\lambda \in \CC \setminus \{0\}$, and a closed subspace $V \subseteq X$.
	If the equation $(\lambda - e^{A}) x = a$ has a solution $x \not\in V$, then there exists $y \not\in V$ such that $(\mu - A) y = a$ for some $\mu \in \log \lambda$.
\end{lemma}
\begin{proof}
	For any $\mu \in \log \lambda$, the operator $A - \mu$ generates the semigroup $e^{-\mu t} e^{At}$, hence
	\begin{align*}
		(A - \mu) \int_{0}^{1} e^{-\mu t} e^{At} x dt &= e^{-\mu} (e^{A} - e^{\mu}) x = -\lambda^{-1} a. 
	\end{align*}
	It suffices to show that \change{some $y = \lambda \int_{0}^{1} e^{-\mu t} e^{At} x dt$} does not belong to $V$.
	Fix a $\mu_0 \in \log \lambda$, then $\log \lambda = \{\mu_0 + k 2\pi i\}_{k \in \ZZ}$.
	Set $\mu_k = \mu_0 + k2\pi i$ and $y_k = \lambda \int_{0}^{1} e^{-\mu_k t} e^{At} x dt$.
	Notice that $y_k$ are the Fourier coefficients of the function
	\begin{align*}
		f : [0,1] \rightarrow X : t \mapsto \lambda e^{-\mu_0 t} e^{At} x,
	\end{align*}
	hence $\sum_{k \in \ZZ} y_k e^{k2\pi i}$ converges to $f$ in $L^2([0,1], X)$.
	Since $f(0) \not\in V$ and $f$ is continuous, at least one of the $y_k$ does not belong to $V$.
\end{proof}
 
The lemma is used to prove
 
\begin{theorem} \label{pro:acp_resilient} \stepcounter{refer}
	Let $\mathcal{P}$ be the $(U, Y)$-relation defined by \cref{eq:io cauchy}. 
	Assume $A$ is a sectorial operator with compact resolvent \change{\cite[Def. A.4.24]{curtain95}}, and $B, C, D$ are continuous operators, then $e^A$ is compact and the following are equivalent:
	\begin{enumerate}[(a)] \itemsep 0mm
		\item\labellocal{item:a} $\mathcal{P}$ is stable.
		\item\labellocal{item:b} Every $v_0 \in X$ satisfying
		\begin{align} \labellocal{eq:1}
			(\mu_m - A) \dots (\mu_1 - A) v_0 &= 0,
		\end{align}
		for some $|\lambda| \geq 1$, $m \in \NN$, and $\mu_1, \ldots, \mu_m \in \log \lambda$, belongs to the unobservable subspace of $(e^{At}, C)$.
		\item\labellocal{item:c} $\mathcal{P}$ is VIVO.
		\item\labellocal{item:d} $\mathcal{P}$ is smoothly VIVO.
	\end{enumerate}
\end{theorem}
\begin{proof}
First we prove compactness of $e^A$.
	Let $R_{z} := (z - A)^{-1}$ for every $z$ in the resolvent set of $A$.
	By \cref{eq:sectorial_calculus} and the resolvent identity $R_{z} = R_{a} + (a - z) R_{a} R_{z}$,
	\begin{align*}
		2\pi i e^{A} &= \int_{\Gamma} e^{z} R_{z} dz
		= R_{a} \int_{\Gamma} e^{z} \left(1 + a R_{z} - z R_{z}\right) dz.
	\end{align*}
	The term $R_{a}$ can be factored out of the integral because the remaining term is integrable in the norm topology of $\Endom(X)$.
	Since $e^{A}$ is the composition of the compact operator $R_{a}$ with a bounded operator, it is compact.

	We now notice that $\mathcal{P}$ coincides with the relation defined by \cref{eq:transition_io_system} with $\Phi(t) = e^{At}$, so \reflocal{item:a}, \reflocal{item:c}, \reflocal{item:d} are equivalent.
	
	It remains to show that \reflocal{item:a} is equivalent to \reflocal{item:b}. Let $V$ denote the unobservable subspace of $(e^{At}, C)$. Assume $\mathcal{P}$ is not stable, then there is $u_0 \in X \setminus V$, $\lvert \lambda \rvert \geq 1$, and $m \in \ZZge{0}$ such that $(\lambda - e^{A})^m u_0 = 0$.
	 By repeated applications of \cref{lem:semigroup_pointspec}, there exists $v_0 \in X \setminus V$ satisfying \creflocal{eq:1}. Conversely, assume \creflocal{eq:1} holds for some $v_0 \in X \setminus V$. Without loss of generality we can assume   that $v_j := (\mu_j - A) v_{j-1}$ belongs to $V$ for every $j \in \{1, \ldots, m\}$.
	Since the subspace $\Span_{j \geq 0} \{v_j\}$ is finite-dimensional and $A$-invariant, the solution of $\dot{x} = Ax$ initialized at $v_0$ is smooth. 
	Let $W$ be the $A$-invariant subspace $\Span_{j \geq 1} \{v_j\}$.
	Since $A v_0 = \mu_1 v_0$ in $X / W$, the solution of $\dot{x} = Ax$ initialized at $v_0$ must be equal to $x(t) = e^{\mu_{1} t} v_0 + r(t)$ for some $r \in C^{\infty}_{\local}(\RRge{0}, W)$.
	Since $v_0 \not\in V$ and $W \subseteq V$, we conclude that $y(t) = Cx(t) = e^{\mu_1 t} C v_0$ is a non-vanishing output of $\mathcal{P}$.
\end{proof}

The hypothesis of \cref{pro:acp_resilient} on $A$ can be certified using the technique \change{described in \cite[Sect. 3.6.2]{sell02}}.
\begin{shownto}{SIAM}
	\change{We present a concise statement tailored for our later applications; a proof is provided in \cite[Lemma 4.14]{xia23}.
	This result will be used in the next section to prove \propnoun{} of parabolic systems.}
\end{shownto}
\begin{shownto}{arXiv}
	We present a concise statement tailored for our later applications, with an adapted proof.
	This result will be used in the next section to prove \propnoun{} of parabolic systems.
\end{shownto}

\begin{lemma} \label{lem:hermitian_sectorial} \stepcounter{refer}
	Given Hilbert spaces $X$ and $Y$ with conjugations\footnote{A conjugation on a Hilbert space $X$ is a conjugate-linear isometry $X \rightarrow X : x \mapsto \bar{x}$ such that $\bar{\bar{x}} = x$.}, and a bounded injective operator $\iota : Y \hookrightarrow X$ commuting with conjugations and having dense image.
	Denote by $\Phi$ the ($\CC$-linear) Riesz isomorphism $X \rightarrow X^* : x \mapsto \langle \bar{x}, \cdot \rangle_{X}$. 
	Let $Q : Y \rightarrow Y^*$ be a bounded operator.
	If there exists $\alpha \in \RR$ such that
	\begin{align} \labellocal{eq:conclusion}
		\lVert y\rVert_{Y}^2 &\lesssim - \re \langle Q y, \bar{y} \rangle + \alpha \lVert \iota y\rVert_{X}^2 \in \RRge{0}
	\end{align} 
	uniformly in $y \in Y$, 
	then the maximal operator $A$ completing the diagram below\footnote{We define $A x_1 = x_2$ whenever there exist elements on the remaining three spots of \creflocal{eq:commute} completing the diagram.} is a sectorial operator.
	\begin{equation} \labellocal{eq:commute}
	\begin{tikzcd}
		X \arrow[rrrd, dotted, bend right=10, "A"] & Y \arrow[l, hook', "\iota"'] \arrow[r, "Q"] & Y^* & X^* \arrow[l, hook', "\iota^*"'] \\
		& & & X \arrow[u, hook, two heads, "\Phi"']
	\end{tikzcd}
	\end{equation}
	Furthermore, if $\iota$ is compact, then $A$ has compact resolvent.
\end{lemma}

\begin{shownto}{arXiv}
\begin{proof}
	Notice that \creflocal{eq:commute} commutes when $A = \lambda \in \CC$ and $Q = \lambda(\iota^* \circ \Phi \circ \iota)$, and for this $Q$, $\langle Qy, \bar{y} \rangle = \lambda \langle \Phi \iota y, \overline{\iota y} \rangle = \lambda \|\iota y\|_{X}$.
	Therefore, \creflocal{eq:commute} remains commutative if $A$ is replaced by $A' = A + \alpha$ and $Q$ is replaced by $Q' = Q + \alpha(\iota^* \circ \Phi \circ \iota)$.
	Furthermore, $\lVert y\rVert_{Y}^2 \lesssim - \re \langle Q' y, \bar{y} \rangle$.
	Since $A$ is sectorial (resp. has compact resolvent) iff $A'$ does, we may assume without loss of generality that $\alpha = 0$.
	In this case, we conclude the proof by showing
	\begin{enumerate}[(a)] \itemsep 0mm
		\item\labellocal{item:a} $Q$ is invertible.
		\item\labellocal{item:b} $A$ is a densely-defined closed operator, whose resolvent set contains $0$.
		\item\labellocal{item:c} The resolvent set of $A$ contains $\CCg{0}$ and $\lVert (\lambda - A)^{-1}\rVert \lesssim (\re \lambda)^{-1}$ uniformly in $\lambda \in \CCg{0}$.
		In particular, for any $\phi < \pi/2$, $\lVert (\lambda - A)^{-1}\rVert \lesssim \lvert \lambda \rvert^{-1}$ for all $\lambda \in \sectl{\phi}$.
		\item\labellocal{item:d} $A$ is sectorial.
		\item\labellocal{item:e} If $\iota$ is compact, then $A$ has compact resolvent.
	\end{enumerate}
	
	Proof of \reflocal{item:a}. The hypothesis implies $\lVert y\rVert_{Y} \lesssim \lVert Qy\rVert_{Y^*}$, hence $Q$ is injective and $\Ima Q \subseteq Y^*$ is closed.
	Since $\langle Qy, \bar{y} \rangle \neq 0$ when $y \neq 0$, we deduce that $(\Ima Q)^{\perp} = 0$ and $\Ima Q = Y^*$.
	Therefore, $Q$ is invertible.
	
	Proof of \reflocal{item:b}.
	By \reflocal{item:a}, $A$ is the inverse of the bounded injective operator $\iota Q^{-1} \iota^* \Phi$.
	Therefore, $A$ is a closed operator whose resolvent set contains $0$.
	Furthermore, $\iota Q^{-1} \iota^* \Phi$ has dense image, as a composition of operators with dense images.
	We conclude that $A$ is densely-defined.
	
	Proof of \reflocal{item:c}.
	For every $\lambda \in \CCge{0}$, the operator $Q' = Q - \lambda(\iota^* \circ \Phi \circ \iota)$ satisfies $\|y\|_{Y}^2 \lesssim - \re \langle Q'y, y \rangle$.
	Applying \reflocal{item:b} to $Q'$, we conclude that $A - \lambda$ is invertible for all $\lambda \in \CCge{0}$. By definition, every $x \in \domain(A)$ is equal to $\iota y$ for some $y \in X$.
	Furthermore, $\langle Ax, x \rangle_{X} = \langle \Phi Ax, \bar{x} \rangle_{X^*, X} = \langle \iota^* \Phi A\iota y, \bar{y} \rangle = \langle Qy, \bar{y} \rangle$.
	We deduce that
	\begin{align*}
		\langle (\lambda - A) x, x \rangle_{X} &= \lambda^* \langle x, x \rangle_{X} - \langle Qy, y \rangle \\
		\re \langle (\lambda - A) x, x \rangle_{X} &\geq \re \lambda^* \lVert x\rVert_{X}^2.
	\end{align*}
	This implies that $\lVert (\lambda - A)x\rVert_{X} \geq \re \lambda \lVert x\rVert_X$ for all $x \in \domain(A)$, which yields the desired inequality.
	
Proof of \reflocal{item:d}. Because $Q$ is bounded and $\lVert y\rVert_{Y}^2 \lesssim -\re \langle Qy, \bar{y} \rangle$, the image of the unit sphere in $Y$ under $y \mapsto \langle Qy, \bar{y} \rangle$ is bounded and contained in $\CCle{-\epsilon}$.
	Therefore, for small $\theta > 0$, both $Q' = e^{\pm i\theta}Q$ verify $\lVert y\rVert_{Y}^2 \lesssim -\re \langle Q'y, \bar{y} \rangle$.
	Applying \reflocal{item:c} to $Q' = e^{\pm i \theta} Q$ shows that that the resolvent sets of $A' = e^{\pm i \theta} A$ contain $\CCg{0}$ and $\lVert (\lambda - A')^{-1}\rVert \lesssim \lvert \lambda \rvert^{-1}$ uniformly in $\lambda \in \sectl{(\pi - \theta)/2}$.
	Therefore, $(\lambda-A)^{-1}$ exists and $\lVert (\lambda - A)^{-1}\rVert \lesssim \lvert \lambda \rvert^{-1}$ uniformly in $\lambda \in \sectl{(\pi + \theta)/2}$.
	
	Proof of \reflocal{item:e} follows from $A^{-1} = \iota Q^{-1} \iota^* \Phi$.
\end{proof}
\end{shownto}

\section{Second-order parabolic equations} \label{sect:parabolic}

In this section, classes of second-order parabolic equations are shown to define \prop{} relations. 
The notation $H^s(\Omega)$ denotes the $L^2$-Sobolev space of order $s$ on the domain $\Omega$.
We refer to 
\begin{shownto}{SIAM}
	\cite[Section 8.1]{xia23} f
\end{shownto}
\begin{shownto}{arXiv}
	\Cref{sect:sobolevspaces} f
\end{shownto}
or the properties of $H^s(\Omega)$ required in this section.
\begin{shownto}{SIAM}
	\change{Due to space constraints, only the Neumann boundary condition will be studied.
	Similar results for the Dirichlet boundary condition are available in \cite[Seciton 5.1]{xia23}.}
\end{shownto}

\subsection{Parabolic equations with Dirichlet boundary condition} \label{sect:parabolic_dirichlet}

In this section, we consider the parabolic equation with Dirichlet boundary condition
\begin{align} \label{eq:parabolic_dirichlet}
	\left\{
	\begin{aligned}
		\partial_t x &= \divergence(a \gradient x) + R x + Bu &&\text{ on } \RRge{0} \times \Omega, \\
		x &= 0 &&\text{ on } \RRge{0} \times \partial \Omega,
	\end{aligned}
	\right.
\end{align}
where
\begin{assumption} \label{asp:parabolic_dirichlet}
	$\Omega$, $a$, $R$, and $B$ satisfy
	\begin{itemize} \itemsep 0mm
		\item $\Omega \subseteq \RR^d$ is a compact $d$-dimensional $C^{\infty}$-submanifold with boundary.
		\item $a \in C^{\infty}(\Omega, \CC^{d \times d})$ is uniformly positive-definite.
		\item $R$ is an order\footnote{An operator $L : C^{\infty}(\Omega) \rightarrow C^{\infty}(\Omega')$ has order $a$ if $\| Lf \|_{H^k(\Omega')} \lesssim \|f\|_{H^{k+a}(\Omega)}$ uniformly in $f \in C^{\infty}(\Omega)$ for all $k \in \ZZge{0}$.} $1$ operator $C^{\infty}(\Omega) \rightarrow C^{\infty}(\Omega)$.
		\item $B$ is an order $0$ operator $C^{\infty}(\Omega) \rightarrow C^{\infty}(\Omega)$.
	\end{itemize}
\end{assumption}

We will show that the $C^{\infty}(\Omega)$-relation $\mathcal{P}$ defined by the smooth solutions $(u, x)$ of \cref{eq:parabolic_dirichlet} is \prop{}.
Since the closed-loop relation $\sysclose(\mathcal{P}, \lambda)$ is also defined by \cref{eq:parabolic_dirichlet}, except with $R$ replaced by $R + \lambda B$, the closed-loop systems $\sysclose(\mathcal{P}, \lambda)$ remains \prop{}.

To prove \propnoun{}, $\mathcal{P}$ is associated to an abstract Cauchy problem.
Define the operator $A$ on $C^{\infty}(\Omega, \CC)$ as 
\begin{align} \label{eq:elliptic_dirichlet}
	\begin{aligned}
		A &: \domain(A) \rightarrow C^{\infty} : f \mapsto \divergence(a \gradient f) + R f \\
		&\domain(A) := \{f \in C^{\infty}(\Omega, \CC) \mid f = 0 \text{ on } \partial \Omega \},
	\end{aligned}
\end{align}
then $\mathcal{P}$ coincides with the $C^{\infty}(\Omega)$-relation defined by $\dot{x} = Ax + Bu$.
However, $C^{\infty}(\Omega)$ is not a Banach space.
This issue is resolved by embedding $C^{\infty}(\Omega)$ in the Banach space $L^2(\Omega)$ and extending the equation $\dot{x} = Ax + Bu$ into an abstract Cauchy problem on $L^2$.

Since $B$ is an operator of order $0$, it extends uniquely to a bounded operator on $L^2$, denoted by $B_e$.
To extend $A$, think of $Af \in C^{\infty}$ as a distribution (i.e., a functional $C^{\infty}_{c}(\Omega^{\circ}) \rightarrow \CC$), then
\begin{align*}
	\langle Af, g \rangle_{C^{\infty}_{c}(\Omega^{\circ})^*, C^{\infty}_{c}(\Omega^{\circ})}
	= \int \divergence(a\gradient f) g + (Rf)g dV 
	= - \int_{\Omega} (a \gradient f) \cdot (\gradient g) - (Rf)g dV.
\end{align*}
Since $R$ is an order $1$ operator, $A$ has a unique continuous extension $\tilde{A} : H^1_0(\Omega^{\circ}) \rightarrow H^1_0(\Omega^{\circ})^*$, given by
\begin{align} \label{eq:elliptic_dirichlet_extension}
	\langle \tilde{A}f, g \rangle 
	&= - \int (a \gradient f) \cdot (\gradient g) - (Rf)g dV.
\end{align}
The desired operator on $L^2$ is obtained by applying \cref{lem:hermitian_sectorial} with $X = L^2(\Omega)$, $Y = H^1_0(\Omega^{\circ})$, $\iota$ the standard inclusion, and $Q = \tilde{A}$.
Note that the maps $\iota$, $\iota^*$, and $\Phi$ in \cref{lem:hermitian_sectorial} are identities, when $X$ and $Y$ are treated as subspaces of distributions.
	
\begin{proposition} \label{pro:elliptic_dirichlet_extension}  \stepcounter{refer}
	Consider $A$ and $\tilde{A}$ defined by \cref{eq:elliptic_dirichlet} and \cref{eq:elliptic_dirichlet_extension}.
	The following statements hold true:
	\begin{enumerate}[(a)] \itemsep 0mm
		\item\labellocal{item:b} There exists $\alpha \in \RR$ such that $\lVert f\rVert_{H^1}^2 \lesssim \alpha \lVert f\rVert_{L^2}^2 - \re \langle \tilde{A} f, \bar{f} \rangle$ for all $f \in H^1_0$.
		\item\labellocal{item:c} Let $A_{e}$ be the maximal operator\footnote{$A_e x = y$ whenever there are members in the remaining three spots completing the diagram.} completing the diagram
		\begin{equation} \labellocal{eq:commute} \begin{tikzcd}[column sep=small]
			L^2(\Omega) \arrow[rrrd, bend right=10, "A_{e}"] & H^{1}_{0}(\Omega^{\circ}) \arrow[l, hook'] \arrow[r, "\tilde{A}"] & H^{1}_{0}(\Omega^{\circ})^* & L^2(\Omega)^* \arrow[l, hook'] \\
			& & & L^2(\Omega) \arrow[u, hook, two heads]
		\end{tikzcd} \end{equation}
		then $A_{e}$ is a sectorial operator on $L^2$ with compact resolvent.
		Furthermore, $\domain(A) = \allowbreak \domain(A_{e}) \allowbreak \cap C^{\infty}$ and $A_e = A$ on $\domain(A)$.
	\end{enumerate}
\end{proposition}

\begin{proof}
	Proof of \reflocal{item:b}. Since $a$ is uniformly positive-definite, there exists $\epsilon > 0$ such that $\langle a \gradient f, \gradient f \rangle_{L^2} \geq \epsilon \lVert \gradient f\rVert_{L^2}^2$.
	Combined with the inequality $\lVert R f\rVert_{L^2} \leq \beta \lVert f\rVert_{H^1}$,
	\begin{align*}
		\re \langle \tilde{A} f, \bar{f} \rangle &\leq -\epsilon \lVert \gradient f\rVert_{L^2}^2 + \beta \lVert f\rVert_{L^2} \lVert f\rVert_{H^1} \\ 
		&\leq -\epsilon \lVert f\rVert_{H^1}^2 + \epsilon \lVert f\rVert_{L^2}^2 + \frac{\epsilon}{4} \lVert f\rVert_{H^1}^2+ \frac{\beta^2}{\epsilon} \lVert f\rVert_{L^2}^2,
	\end{align*}
	where the Cauchy-Schwarz inequality $2ab \leq \epsilon a^2 + b^2 / \epsilon$ is used in the second line. 
	Isolating $\lVert f\rVert_{H^1}^2$ completes the proof.
	
	Proof of \reflocal{item:c}. That $A_e$ extends $A$ by construction.
	By \cref{lem:hermitian_sectorial}, $A_{e}$ is a sectorial operator with compact resolvent\footnote{The inclusion $\iota : H^1(\Omega^{\circ}) \hookrightarrow L^2(\Omega)$ is compact by Rellich-Kondrachov 
	\begin{shownto}{SIAM}
		\cite[Lemma 8.3]{xia23}
	\end{shownto}
	\begin{shownto}{arXiv}
		(\Cref{lem:sobolev_compact})
	\end{shownto}
	.}. 
	It remains to show $\domain(A_{e}) \cap C^{\infty} \subseteq \domain(A)$.
	Since $\domain(A_{e}) \subseteq H^1_0$, it suffices to prove $H^1_0 \cap C^{\infty} \subseteq \domain(A)$.
	For any $v \in C^{\infty}(\Omega, \CC^d)$, consider the map
	\begin{align*}
		\Lambda_v : C^{\infty}(\Omega) \rightarrow \CC 
		: f \mapsto \int_{\partial \Omega} \hat{n} \cdot fv  dS = \int_{\Omega} \divergence(fv) dV
	\end{align*}
	where $\hat{n}$ is the outward unit normal vector on $\partial \Omega$.
	Since $\lvert \Lambda_v f \rvert \lesssim \lVert f\rVert_{H^1}$, $\Lambda_v$ extends to a continuous functional on $H^1$.
	Because $\Lambda_v = 0$ on $C^{\infty}_{c}(\Omega^{\circ})$ which is dense in $H^1_0(\Omega)$, it follows that $\Lambda_v = 0$ on $H^1_0$.
	In particular, $\Lambda_v f = 0$ for every $f \in H^1_0 \cap C^{\infty}$.
	Since $v \in C^{\infty}(\Omega, \CC^d)$ is arbitrary, $f = 0$ on $\partial \Omega$ and $f \in \domain(A)$.
\end{proof}

Since the extension $A_{e}$ is a sectorial operator with compact resolvent, we conclude by \cref{pro:transition_closed}, \cref{cor:transition_welldefined} and \cref{pro:acp_resilient} the following

\begin{corollary} \label{cor:elliptic_dirichlet_ext_inputresilient}
	The $L^2(\Omega)$-relation $\mathcal{P}_e$ defined by $\dot{x} = A_e x + B_e u$ is a closed well-defined relation.
	Furthermore, the following are equivalents:
	\begin{enumerate}[(a)]
		\item $\mathcal{P}_e$ is stable.
		\item If $(0, x) \in \mathcal{P}_e$ with $x(t) = e^{\mu t} x(0)$ for some $\mu \in \CCge{0}$, then $x = 0$.
		\item $\mathcal{P}_e$ is VIVO.
		\item $\mathcal{P}_e$ is smoothly VIVO.
	\end{enumerate}
\end{corollary}

These properties of $\mathcal{P}_e$ can be transferred to the $C^{\infty}$-relation $\mathcal{P}$ of $\dot{x} = Ax + Bu$, by exploiting the connections between these relations.
It is clear that $\mathcal{P} \subseteq \mathcal{P}_{e}$ since $A_{e}$, $B_e$ extend $A$, $B$.
Conversely, we have

\begin{lemma} \label{lem:parabolic_dirichlet_extension} \stepcounter{refer}
	The $L^2(\Omega)$-relation $\mathcal{P}_e$ defined by $\dot{x} = A_e x + B_e u$ satisfies
	\begin{enumerate}[(a)]
		\item\labellocal{item:a} For any $k \in \ZZge{0}$, $\|f\|_{H^{k+2}} \lesssim \|f\|_{L^2} + \|A_e f\|_{H^k}$ uniformly in $f \in \domain(A_e)$.
		\item\labellocal{item:b} For any $k \in \ZZge{0}$, if $(u, x) \in \mathcal{P}_{e}$ and $u \in C^{\infty}_{\local}(\RRge{0}, H^k)$, then $x \in C^{\infty}_{\local}(\RRge{0}, \allowbreak H^{k+2})$.
		Furthermore, if $x$ vanishes smoothly in $L^2$ and $u$ vanishes smoothly in $H^k$, then $x$ vanishes smoothly in $H^{k+2}$.
		\item\labellocal{item:c} If $(u, x) \in \mathcal{P}_{e}$ and $u \in C^{\infty}_{\local}(\RRge{0}, C^{\infty})$, then $x \in C^{\infty}_{\local}(\RRge{0}, C^{\infty})$.
		Furthermore, if $x$ vanishes smoothly in $L^2$ and $u$ vanishes smoothly in $C^{\infty}$, then $x$ vanishes smoothly in $C^{\infty}$.
	\end{enumerate}
\end{lemma}
\begin{proof}
Proof of \reflocal{item:a}.
	By definition, $\int (a \gradient f) \cdot (\gradient g) + (A_e f - Rf)g dV = 0$ for all $f \in \domain(A_e)$ and $g \in H^{1}_{0}(\Omega^{\circ})$.
	By elliptic regularity 
	\begin{shownto}{SIAM}
		\cite[Lemma 8.7]{xia23}
	\end{shownto}
	\begin{shownto}{arXiv}
		(\cref{lem:elliptic_regularity_0})
	\end{shownto}
	, 
	\begin{align*}
		\| f \|_{H^{k+2}} &\lesssim \| f \|_{H^{k+1}} + \| A_e f \|_{H^{k}} + \| Rf \|_{H^k}.
	\end{align*}
	Since $R$ is an order $1$ operator, $\| Rf \|_{H^k}$ is bounded by $\|f\|_{H^{k+1}}$. 
	By Ehrling interpolation 
	\begin{shownto}{SIAM}
		\cite[Lemma 8.4]{xia23}
	\end{shownto}
	\begin{shownto}{arXiv}
		(\cref{lem:ehrling})
	\end{shownto}
	, $\|f\|_{H^{k+1}}$ is bounded by $\epsilon \|f\|_{H^{k+2}} + \alpha \|f\|_{L^2}$ for arbitrarily small $\epsilon > 0$.
	Choosing $\epsilon$ sufficiently small and moving $\epsilon \| f \|_{H^{k+2}}$ to the left-hand side yields the desired inequality.

	Proof of \reflocal{item:b}.
	By assumption, $x \in C^{\infty}_{\local}(\RRge{0}, L^{2})$ and $A_{e} x = B_e u - \dot{x} = B_e u - \dot{x} \in C^{\infty}_{\local}(\RRge{0}, H^k)$.
	By \reflocal{item:a}, $x \in C^{0}_{\local}(\RRge{0}, H^{k+2})$.
	Since $(\partial_t^j u, \partial_t^j x) \in \mathcal{P}_{e}$ for all $j \in \ZZge{0}$, the same argument shows that $\partial_t^j x \in C^{0}_{\local}(\RRge{0}, H^{k+2})$, which means $x \in C^{\infty}_{\local}(\RRge{0}, H^{k+2})$.
	Smooth vanishing of $x$ is proved analogously.
	
	Proof of \reflocal{item:c}
	follows from \reflocal{item:b} and Sobolev inequalities 
	\begin{shownto}{SIAM}
		\cite[Lemma 8.1]{xia23}
	\end{shownto}
	\begin{shownto}{arXiv}
		(\cref{lem:sobolev_inequality})
	\end{shownto}
	.
\end{proof}

The notable consequence of \cref{lem:parabolic_dirichlet_extension} is that $(u, x)$ belongs to the $C^{\infty}$-relation $\mathcal{P}$ if and only if $(u, x) \in \mathcal{P}_e$ and $u \in C^{\infty}_{\local}(\RRge{0}, C^{\infty})$.
Therefore, $\mathcal{P}$ is the restriction of the $L^2$-relation $\mathcal{P}_e \subseteq C^{\infty}_{\local}(\RRge{0}, L^2) \times C^{\infty}_{\local}(\RRge{0}, L^2)$ to $C^{\infty}_{\local}(\RRge{0}, C^{\infty}) \times C^{\infty}_{\local}(\RRge{0}, C^{\infty})$.
Combined with \cref{cor:elliptic_dirichlet_ext_inputresilient}, we deduce

\begin{theorem} \label{pro:parabolic_dirichlet_resilient} \stepcounter{refer}
	The $C^{\infty}(\Omega)$-relation $\mathcal{P}$ defined by \cref{eq:parabolic_dirichlet} is a closed well-defined relation.
	Furthermore, the following are equivalent:
	\begin{enumerate}[(a)]
		\item\labellocal{item:a} $\mathcal{P}$ is stable in $L^2$.
		\item\labellocal{item:b} If $(0, x) \in \mathcal{P}_e$ with $x(t) = e^{\mu t} x(0)$ for some $\mu \in \CCge{0}$, then $x = 0$.
		\item\labellocal{item:c} $\mathcal{P}$ is VIVO in $L^2$.
		\item\labellocal{item:d} $\mathcal{P}$ is smoothly VIVO in $L^2$.
		\item\labellocal{item:e} $\mathcal{P}$ is smoothly VIVO in $H^k$ for every $k \in \ZZge{0}$.
		\item\labellocal{item:f} $\mathcal{P}$ is smoothly VIVO in $C^{\infty}$.
	\end{enumerate}
\end{theorem}


\subsection{Parabolic equations with Neumann boundary condition}

In this section, we consider the parabolic equation
\begin{align} \label{eq:parabolic_neumann}
	\left\{
	\begin{aligned}
		\partial_t x &= \divergence(a \gradient x) + R x + B u &&\text{ on } \RRge{0} \times \Omega \\
		\hat{n} \cdot a \gradient x &= Kx &&\text{ on } \RRge{0} \times \partial \Omega
	\end{aligned}
	\right.
\end{align}
where $\hat{n}$ denotes the outward unit normal vector on $\partial \Omega$ and
\begin{assumption} \label{asp:parabolic_neumann}
	$\Omega$, $a$, $R$, $B$, and $K$ satisfy
	\begin{itemize}
		\item $\Omega \subseteq \RR^d$ is a compact $d$-dimensional $C^{\infty}$-submanifold with boundary.
		\item $a \in C^{\infty}(\Omega, \CC^{d \times d})$ is uniformly positive-definite.
		\item $R$ is an order\footnote{An operator $L : C^{\infty}(\Omega) \rightarrow C^{\infty}(\Omega')$ has order $a$ if $\| Lf \|_{H^k(\Omega')} \lesssim \|f\|_{H^{k+a}(\Omega)}$ uniformly in $f \in C^{\infty}(\Omega)$ for all $k \in \ZZge{0}$.} $1$ operator $C^{\infty}(\Omega) \rightarrow C^{\infty}(\Omega)$.
		\item $B$ is an order $0$ operator $C^{\infty}(\Omega) \rightarrow C^{\infty}(\Omega)$.
		\item $K$ is an order $1$ operator $C^{\infty}(\Omega) \rightarrow C^{\infty}(\partial \Omega)$.
	\end{itemize}
\end{assumption}

We would like to show that the $C^{\infty}(\Omega)$-relation $\mathcal{P}$ defined by \cref{eq:parabolic_neumann} with input $u$ and output $x$ is \prop{}.
Since the closed-loop relation $\sysclose(\mathcal{P}, \lambda)$ also has the form \cref{eq:parabolic_neumann}, the closed-loop systems $\sysclose(\mathcal{P}, \lambda)$ remains \prop{}.

The relation $\mathcal{P}$ coincides with the $C^{\infty}$-relation of $\dot{x} = Ax + Bu$ for $A$ defined by
\begin{align} \label{eq:elliptic_neumann}
	\begin{aligned}
		A &: \domain(A) \rightarrow C^{\infty}(\Omega) : f \mapsto \divergence(a \gradient f) + R f, \\
		&\domain(A) :=  \{f \in C^{\infty}(\Omega) \mid \hat{n} \cdot a \gradient f = Kf \text{ on } \partial \Omega \}.
	\end{aligned}
\end{align}
\change{We will find a sectorial operator $A_e$ on $L^2(\Omega)$ extending the operator $A$, then solutions of $\dot{x} = Ax + Bu$ will be solutions of the abstract Cauchy problem $\dot{x} = A_e x + B_e u$, where $B_e : L^2 \rightarrow L^2$ is the unique bounded extension of $B$ (which exists since $B$ has order $0$).}

Think of $A$ as a linear map from $\domain(A) \subseteq C^{\infty}(\Omega)$ to $C^{\infty}(\Omega)^*$, then
\begin{align*}
	\langle Af, g \rangle_{C^{\infty}(\Omega)^*, C^{\infty}(\Omega)} &= \int_{\Omega} \divergence(a\gradient f) g + (Rf)g dV \\
	&= \int_{\partial \Omega} (Kf) g dS - \int_{\Omega} (a \gradient f) \cdot (\gradient g) dV 
	+ \int_{\Omega} (Rf)g dV
\end{align*}
By the regularities of $K$ and $R$, $A$ extends to a continuous operator $\tilde{A} : H^1(\Omega) \rightarrow H^1(\Omega)^*$ given by
\begin{align} \label{eq:elliptic_neuman_extension}
	\begin{aligned}
	\langle \tilde{A}f, g \rangle 
	&= \int_{\partial \Omega} (Kf) g dS - \int_{\Omega} (a \gradient f) \cdot (\gradient g) - (Rf)g dV.
	\end{aligned}
\end{align}
\change{The operator $A_e$ is constructed by \cref{lem:hermitian_sectorial} with $X = L^2(\Omega)$, $Y = H^1(\Omega)$, $Q = \tilde{A}$, and $\iota$ being the standard inclusion.
Note that the maps $\iota$, $\iota^*$, and $\Phi$ in \cref{lem:hermitian_sectorial} are identities, when $X$ and $Y$ are embedded in the space of distributions $C^{\infty}(\Omega)^*$.}

\begin{proposition} \label{pro:elliptic_neumann_extension} \stepcounter{refer}
	Consider the operator $A$ and $\tilde{A}$ defined by \cref{eq:elliptic_neumann} and \cref{eq:elliptic_neuman_extension}.
	The following statements hold true:
	\begin{enumerate}[(a)] \itemsep 0mm
		\item\labellocal{item:a} 
		There exists $\alpha \in \RR$ such that $\lVert f\rVert_{H^1}^2 \lesssim \alpha \lVert f\rVert_{L^2}^2 - \re \langle \tilde{A} f, \bar{f} \rangle$ for all $f \in H^1$.
		\item\labellocal{item:b} Let $A_{e}$ be the maximal operator\footnote{$A_e x = y$ whenever there are members in the remaining three spots completing the diagram.} completing the diagram
	\begin{equation} \labellocal{eq:commute}
	\begin{tikzcd}[column sep=small]
		L^2(\Omega) \arrow[rrrd, bend right=10, "A_{e}"] & H^{1}(\Omega) \arrow[l, hook'] \arrow[r, "\tilde{A}"] & H^{1}(\Omega)^* & L^2(\Omega)^* \arrow[l, hook'] \\
		& & & L^2(\Omega) \arrow[u, hook, two heads]
	\end{tikzcd}
	\end{equation}
	then $A_{e}$ is a sectorial operator on $L^2$ with compact resolvent.
	Furthermore, $\domain(A) = \allowbreak \domain(A_{e}) \allowbreak \cap C^{\infty}$ and $A_e = A$ on $\domain(A)$.
	\end{enumerate}
\end{proposition}
\begin{proof}
	Proof of \reflocal{item:a}.
	Consider \cref{eq:elliptic_neuman_extension}.
	By uniform positive-definiteness of $a$, there exists $\epsilon > 0$ such that $\langle a \gradient f, \gradient f \rangle_{L^2} \geq \epsilon \lVert \gradient f\rVert_{L^2}^2$.
	In view of $\lVert R f\rVert_{L^2} \lesssim \lVert f\rVert_{H^1}$, $\lVert Kf\rVert_{L^2(\partial \Omega)} \lesssim \lVert f\rVert_{H^1(\Omega)}$, and trace inequalities 
	\begin{shownto}{SIAM}
		\cite[Lemma 8.2]{xia23}
	\end{shownto}
	\begin{shownto}{arXiv}
		(\cref{lem:trace})
	\end{shownto}
	,
	there exists $\beta > 0$ such that
	\begin{align*}
		\epsilon \lVert f \rVert_{H^1}^2 &\leq \beta \lVert f\rVert_{H^1} \lVert f\rVert_{H^{1/2}} - \re \langle \tilde{A}f, \bar{f} \rangle.
	\end{align*}
	By Ehrling interpolation 
	\begin{shownto}{SIAM}
		\cite[Lemma 8.4]{xia23}
	\end{shownto}
	\begin{shownto}{arXiv}
		(\cref{lem:ehrling})
	\end{shownto}
	, there exists $\gamma$ such that $\beta \lVert f\rVert_{H^{1/2}} \leq \gamma \lVert f\rVert_{L^2} + \epsilon \lVert f\rVert_{H^1} / 2$.
	This implies
	\begin{align*}
		\frac{\epsilon}{2} \| f \|_{H^1}^2 &\leq \gamma \lVert f\rVert_{H^1} \lVert f\rVert_{L^2} - \re \langle \tilde{A}f, \bar{f} \rangle.
	\end{align*}
	The desired inequality follows from Cauchy-Schwarz inequality.

	Proof of \reflocal{item:b}.
	By \cref{lem:hermitian_sectorial}, $A_{e}$ a sectorial operator with compact resolvent\footnote{The inclusion $\iota : H^1(\Omega) \hookrightarrow L^2(\Omega)$ is compact by Rellich-Kondrachov 
	\begin{shownto}{SIAM}
		\cite[Lemma 8.3]{xia23}
	\end{shownto}
	\begin{shownto}{arXiv}
		(\Cref{lem:sobolev_compact})
	\end{shownto}
	.}.
	
	It is clear by the definition of $\tilde{A}$ that $\domain(A) \subseteq \domain(A_e) \cap C^{\infty}$ and $A = A_e$ on $\domain(A)$.
	Assume $f \in \domain(A_{e}) \cap C^{\infty}$, then $f \in H^1$ and $A_e f = \tilde{A} f$ as distributions.
	This means $\int (A_e f) g dV = \langle \tilde{A}f, g \rangle$ for all $g \in H^1$, which writes
	\begin{align*}
		 \int_{\Omega} (A_e f - \divergence (a \gradient f) - Rf) g dV &=
		\int_{\partial \Omega} (Kf - \hat{n} \cdot a \gradient f) g dS.
	\end{align*}
	Substituting $g \in C^{\infty}_{c}(\Omega^{\circ})$ (i.e., vanishing near $\partial \Omega$) shows $A_e f = \divergence (a \gradient f) + Rf$.
	Since every smooth function on $\partial \Omega$ is the restriction of some $g \in C^{\infty}(\Omega)$, we conclude that $\hat{n} \cdot a \gradient f = Kf$ on $\partial \Omega$, hence $f \in \domain(A)$.
\end{proof}

\change{

Since the extension $A_{e}$ is a sectorial operator with compact resolvent, we conclude by \cref{pro:transition_closed}, \cref{cor:transition_welldefined} and \cref{pro:acp_resilient} the following

\begin{corollary} \label{cor:elliptic_neumann_ext_inputresilient}
	The $L^2(\Omega)$-relation $\mathcal{P}_e$ defined by $\dot{x} = A_e x + B_e u$ is closed and well-defined.
	Furthermore, the following are equivalents:
	\begin{enumerate}[(a)]
		\item $\mathcal{P}_e$ is stable.
		\item If $(0, x) \in \mathcal{P}_e$ with $x(t) = e^{\mu t} x(0)$ for some $\mu \in \CCge{0}$, then $x = 0$.
		\item $\mathcal{P}_e$ is VIVO.
		\item $\mathcal{P}_e$ is smoothly VIVO.
	\end{enumerate}
\end{corollary}

Similar properties can be concluded about the $C^{\infty}$-relation $\mathcal{P}$ defined by $\dot{x} = Ax + Bu$, using the following identification of $\mathcal{P}$ as a subspace of $\mathcal{P}_e$.

\begin{lemma} \label{lem:parabolic_neumann_extension} \stepcounter{refer}
	The $L^2(\Omega)$-relation $\mathcal{P}_e$ defined by $\dot{x} = A_e x + B_e u$ satisfies
	\begin{enumerate}[(a)]
		\item\labellocal{item:a} For any $k \in \ZZge{0}$, $\|f\|_{H^{k+2}} \lesssim \|f\|_{L^2} + \|A_e f\|_{H^k}$ uniformly in $f \in \domain(A_e)$.
		\item\labellocal{item:b} For any $k \in \ZZge{0}$, if $(u, x) \in \mathcal{P}_{e}$ and $u \in C^{\infty}_{\local}(\RRge{0}, H^k)$, then $x \in C^{\infty}_{\local}(\RRge{0}, \allowbreak H^{k+2})$.
		Furthermore, if $u$ vanishes smoothly in $H^k$ and $x$ vanishes smoothly in $L^2$, then $x$ vanishes smoothly in $H^{k+2}$.
		\item\labellocal{item:c} If $(u, x) \in \mathcal{P}_{e}$ and $u \in C^{\infty}_{\local}(\RRge{0}, C^{\infty})$, then $x \in C^{\infty}_{\local}(\RRge{0}, C^{\infty})$.
		Furthermore, if $x$ vanishes smoothly in $L^2$ and $u$ vanishes smoothly in $C^{\infty}$, then $x$ vanishes smoothly in $C^{\infty}$.
	\end{enumerate}
\end{lemma}

\begin{proof}
	Proof of \reflocal{item:a}.
	By definition, $\int (a \gradient f) \cdot (\gradient g) + (A_e f - Rf)g dV = \int_{\partial \Omega} (Kf)g dS$ for all $f \in \domain(A_e)$ and $g \in H^{1}$.
	By elliptic regularity 
	\begin{shownto}{SIAM}
		\cite[Lemma 8.9]{xia23}
	\end{shownto}
	\begin{shownto}{arXiv}
		(\cref{lem:elliptic_regularity_1})
	\end{shownto}
	, 
	\begin{align*}
		\| f \|_{H^{k+2}} &\lesssim \| f \|_{H^{k+1}} + \| A_e f \|_{H^{k}} + \| Rf \|_{H^k} + \| Kf \|_{H^{k+\frac{1}{2}}(\partial \Omega)}.
	\end{align*}
	By \cref{asp:parabolic_neumann}, both $\| Rf \|_{H^k}$ and $\|Kf\|_{H^{k+\frac{1}{2}}}$ are bounded by $\|f\|_{H^{k+\frac{3}{2}}}$.
	By Ehrling interpolation 
	\begin{shownto}{SIAM}
		\cite[Lemma 8.4]{xia23}
	\end{shownto}
	\begin{shownto}{arXiv}
		(\cref{lem:ehrling})
	\end{shownto}
	, $\|f\|_{H^{k+\frac{3}{2}}}$ is bounded by $\epsilon \|f\|_{H^{k+2}} + \alpha \|f\|_{L^2}$ for arbitrarily small $\epsilon > 0$.
	Choosing $\epsilon$ sufficiently small and moving $\epsilon \| f \|_{H^{k+2}}$ to the left-hand side yields the desired inequality.

	Proof of \reflocal{item:b}.
	By assumption, $x \in C^{\infty}_{\local}(\RRge{0}, L^{2})$ and $A_{e} x = B_e u - \dot{x} = B_e u - \dot{x} \in C^{\infty}_{\local}(\RRge{0}, H^k)$.
	By \reflocal{item:a}, $x \in C^{0}_{\local}(\RRge{0}, H^{k+2})$.
	Since $(\partial_t^j u, \partial_t^j x) \in \mathcal{P}_{e}$ for all $j \in \ZZge{0}$, the same argument shows that $\partial_t^j x \in C^{0}_{\local}(\RRge{0}, H^{k+2})$, which means $x \in C^{\infty}_{\local}(\RRge{0}, H^{k+2})$.
	Smooth vanishing of $x$ is proved analogously.
	
	Proof of \reflocal{item:c}
	follows from \reflocal{item:b} and Sobolev inequalities 
	\begin{shownto}{SIAM}
		\cite[Lemma 8.1]{xia23}
	\end{shownto}
	\begin{shownto}{arXiv}
		(\cref{lem:sobolev_inequality})
	\end{shownto}
	.
\end{proof}

Combining \cref{lem:parabolic_neumann_extension} and \cref{cor:elliptic_neumann_ext_inputresilient} yields

}

\begin{theorem} \label{pro:parabolic_neumann_resilient} \stepcounter{refer}
	The $C^{\infty}(\Omega)$-relation $\mathcal{P}$ defined by \cref{eq:parabolic_neumann} is a closed well-defined relation.
	Furthermore, the following are equivalent:
	\begin{enumerate}[(a)]
		\item\labellocal{item:a} $\mathcal{P}$ is stable in $L^2$.
		\item\labellocal{item:b} If $(0, x) \in \mathcal{P}_e$ with $x(t) = e^{\mu t} x(0)$ for some $\mu \in \CCge{0}$, then $x = 0$.
		\item\labellocal{item:c} $\mathcal{P}$ is VIVO in $L^2$.
		\item\labellocal{item:d} $\mathcal{P}$ is smoothly VIVO in $L^2$.
		\item\labellocal{item:e} $\mathcal{P}$ is smoothly VIVO in $H^k$ for every $k \in \ZZge{0}$.
		\item\labellocal{item:f} $\mathcal{P}$ is smoothly VIVO in $C^{\infty}$.
	\end{enumerate}
\end{theorem}

\subsection{Parabolic equations with Neumann boundary input} \label{sect:parabolicsystem_boundaryinput}

The previous section only studied parabolic equations with in-domain inputs.
This section considers a generalized equation with boundary input.
The equation of interest has the form
\begin{align} \label{eq:parabolic_neumann_input}
	\left\{
	\begin{aligned}
		\partial_t x &= \divergence(a \gradient x) + R x + Bu && \text{ on } \RRge{0} \times \Omega, \\
		\hat{n} \cdot a \gradient x &= Kx + Mu &&\text{ on } \RRge{0} \times \partial \Omega,
	\end{aligned}
	\right.
\end{align}
where
\begin{assumption} \label{asp:parabolic_neumann_input}
	$\Omega$, $a$, $R$, $B$, $K$, and $M$ satisfy
	\begin{itemize} \itemsep 0mm
		\item $\Omega$ is a compact $d$-dimensional $C^{\infty}$-submanifold with boundary.
		\item $a \in C^{\infty}(\Omega, \CC^{d \times d})$ is uniformly positive-definite.
		\item $R$ is an order $1$ operator $C^{\infty}(\Omega) \rightarrow C^{\infty}(\Omega)$.
		\item $B$ is an order $0$ operator $C^{\infty}(\Omega) \rightarrow C^{\infty}(\Omega)$.
		\item $K$ and $M$ are order $1$ operators $C^{\infty}(\Omega) \rightarrow C^{\infty}(\partial \Omega)$.
	\end{itemize}
\end{assumption}

We will show that the $C^{\infty}(\Omega)$-relation $\mathcal{P}$ defined by \cref{eq:parabolic_neumann_input} is \prop{}.
\change{The strategy is to establish a correspondence between \cref{eq:parabolic_neumann_input} and a parabolic equation without boundary input, and to infer \propnoun{} from the latter equation.}
The following lemma is needed for the correspondence.

\begin{lemma} \label{lem:boundary_solution} \stepcounter{refer}
	Given $a$ and $R$ described in \cref{asp:parabolic_neumann_input}, 
	there exists a continuous linear map $J : C^{\infty}(\partial \Omega, \CC) \rightarrow C^{\infty}(\Omega, \CC)$ and some $\mu \in \RRge{0}$ such that 
	\begin{align*}
		\left\{
		\begin{aligned}
			\divergence(a \gradient Jh) + R Jh &= \mu Jh &&\text{ on } \Omega, \\
			\hat{n} \cdot a \gradient (Jh) &= KJh - h &&\text{ on } \partial \Omega
		\end{aligned}
		\right.
	\end{align*}
	for all $h \in C^{\infty}(\partial \Omega)$.
	Furthermore, $\lVert Jh\rVert_{H^{k+2}(\Omega)} \lesssim \lVert h\rVert_{H^{k+(1/2)}(\partial \Omega)}$ for every $k \in \ZZge{0}$.
\end{lemma}
\begin{proof}
	Recall the operator $\tilde{A} : H^1 \rightarrow (H^1)^*$ defined in \cref{eq:elliptic_neuman_extension}.
	By \cref{pro:elliptic_neumann_extension} and Lax-Milgram, there exists $\mu \in \RRge{0}$ such that $\tilde{A} - \mu$ is invertible.
	Let $\iota : C^{\infty}(\partial \Omega) \rightarrow H^1(\Omega)^*$ denote the boundary operator taking each $h \in C^{\infty}(\partial \Omega)$ to
	\begin{align*}
		\iota h : H^1(\Omega) \rightarrow \CC : g \mapsto \int_{\partial \Omega} h g dS.
	\end{align*}
	We claim that $J = (\tilde{A} - \mu)^{-1} \iota$.
	
	A priori, $J = (\tilde{A} - \mu)^{-1} \iota$ is continuous map from $H^{1/2}(\partial \Omega)$ to $H^1(\Omega)$.
	The equality $(\tilde{A} - \mu)f = \iota h$ means
	\begin{align} \labellocal{eq:1}
		\int_{\partial \Omega} (Kf - h) g dS &= \int_{\Omega} (a \gradient f) \cdot (\gradient g) + (\mu f - Rf)g dV
	\end{align}
	for all $g \in H^1(\Omega)$. 
	By elliptic regularity 
	\begin{shownto}{SIAM}
		\cite[Lemma 8.9]{xia23} a
	\end{shownto}
	\begin{shownto}{arXiv}
		(\cref{lem:elliptic_regularity_1}) a
	\end{shownto}
	nd \cref{asp:parabolic_neumann_input}, 
	we deduce $\|f\|_{H^{k+2}} \lesssim \|f\|_{H^{k+3/2}} + \|h\|_{H^{k+1/2}(\partial \Omega)}$.
	By Ehrling interpolation, $\|f\|_{H^{k+3/2}}$ may be replaced by $\|f\|_{H^1}$.
	Therefore, $J$ is continuous from $H^{k+1/2}(\partial \Omega)$ to $H^{k+2}(\Omega)$ for every $k \in \ZZge{0}$.
	By Sobolev inequalities, $J$ maps $C^{\infty}(\partial \Omega)$ continuously into $C^{\infty}(\Omega)$.
	
	When $f$, $g$, $h$ are smooth, \creflocal{eq:1} may be written as
	\begin{align*}
		\int_{\partial \Omega} (Kf - h - \hat{n} \cdot (a \gradient f)) g dS &= \int_{\Omega} (\mu f - Rf - \divergence(a\gradient f)) g dV
	\end{align*}
	This implies $\divergence(a \gradient f) + R f = \mu f$ on $\Omega$ and $\hat{n} \cdot (a \gradient f) = Kf - h$ on $\partial \Omega$.
\end{proof}

For any solution $(u, x)$ of \cref{eq:parabolic_neumann_input}, the function $z := JMu \in C^{\infty}_{\local}(\RRge{0}, C^{\infty})$ is a solution of
\begin{align*}
	\left\{
	\begin{aligned}
		\partial_t z &= \divergence(a \gradient z) + R z + JM\partial_t u - \mu JMu &&\text{ on } \Omega, \\
		\hat{n} \cdot a \gradient z &= Kz - Mu &&\text{ on } \partial \Omega.
	\end{aligned}
	\right.
\end{align*}
The boundary input in \cref{eq:parabolic_neumann_input} can be removed by adding $z$.
This allows us to obtain

\begin{theorem} \label{thm:parabolic_neumann_input} \stepcounter{refer}
	The $C^{\infty}(\Omega)$-relation $\mathcal{P}$ defined by \cref{eq:parabolic_neumann_input} is a closed well-defined relation.
	Furthermore, the following are equivalent:
	\begin{enumerate}[(a)] 
		\item $\mathcal{P}$ is stable in $L^2$.
		\item $\mathcal{P}$ is smoothly VIVO in $H^k$ for every $k \in \ZZge{2}$.
		\item $\mathcal{P}$ is smoothly VIVO in $C^{\infty}$.
	\end{enumerate}
\end{theorem}
\begin{proof}
	Let $J$ be the operator defined in \cref{lem:boundary_solution}, then $(u, x) \in \mathcal{P}$ iff $(u, z) = (u, x + JMu)$ belongs to the $C^{\infty}$-relation $\mathcal{P}'$ defined by
	\begin{align} \labellocal{eq:1}
		\left\{
		\begin{aligned}
			\partial_t z &= \divergence(a \gradient z) + R z + Fu &&\text{ on } \RRge{0} \times \Omega, \\
			\hat{n} \cdot a \gradient z &= Kz &&\text{ on } \RRge{0} \times \partial \Omega
		\end{aligned}
		\right.
	\end{align}
	where $Fu = Bu + JM \partial_t u - \mu JM u$.
	Consider $\mathcal{P}'$ as the composition of the $C^{\infty}$-relation defined by \creflocal{eq:1} with input $Fu$ and output $z$ and the $C^{\infty}$-relation defined by $F$.
	Both relations are closed and well-defined (\cref{pro:parabolic_neumann_resilient}), hence so are $\mathcal{P}'$ and $\mathcal{P}$.
	The equivalence follows from the fact that for $k \in \ZZge{0}$, $\|JMf\|_{H^k} \lesssim \|f\|_{H^k}$ and $Fu$ smoothly vanishes in $H^k$ whenever $u$ smoothly vanishes in $H^k$.
\end{proof}

\change{
Having established \propnoun{} of parabolic equations with Neumann boundary inputs,
\cref{cor:sync} specializes to a sufficient and necessary condition for synchronization in a network of such systems.
It is worth noting that the general synchronization results for infinite-dimensional systems in \cite{demetriou13} are not applicable to this network due to the boundary coupling (the associated ``$B$'' operator is unbounded).
Unlike most results in the literature (e.g. \cite{pilloni16,wu16,deutscher22}), which provide sufficient conditions for synchronization in the $L^2$ or $H^2$ norm using Lyapunov analysis, our approach provides sufficient and necessary conditions for synchronization in the $C^{\infty}$ (and $H^k$) norm in terms of the stability of the decoupled subsystems in the same norm.
The application of our results to synchronization analysis is illustrated in the next subsection.
}

\subsection{A synchronization example} \label{sect:example}

\change{
Consider a network of $n$ identical heat equations
\begin{align*}
	\left\{
	\begin{aligned}
		\partial_t x_j(t, \xi) &= \partial_{\xi}^2 x_j(t, \xi) \qquad \text{on } \RRge{0} \times [0,1], \\
		\partial_{\xi} x_j(t, 0) &= 0, \\
		\partial_{\xi} x_j(t, 1) &= u_j(t), \\
		y_j(t) &= x_j(t, 0),
	\end{aligned}
	\right.
\end{align*} 
indexed by $j \in \{1, \ldots, n\}$, and coupled by diffusive coupling $u_j = \sum_{i=1}^{n} \sigma_{ji} (y_i - y_j)$.
We seek conditions on the coupling parameters $\sigma_{ji}$ for which the states $x_j$ converge to each other in the $C^{\infty}([0,1])$ norm. 
}

\change{
To apply \cref{cor:sync} for state synchronization, we consider an equivalent representation of the network as the interconnections of
\begin{align} \label{eq:heat_equation}
	\left\{
	\begin{aligned}
		\partial_t x_j(t, \xi) &= \partial_{\xi}^2 x_j(t, \xi) \qquad \text{on } \RRge{0} \times [0,1], \\
		\partial_{\xi} x_j(t, 0) &= 0, \\
		\partial_{\xi} x_j(t, 1) &= v_j(t, 0) =: M v_j(t),
	\end{aligned}
	\right.
\end{align}
with the new input $v_j := \sum_{i=1}^{n} \sigma_{ji} (x_i - x_j)$ and $M$ defined as the evaluation operator at $\xi = 0$.
The network is $\sysclose(\mathcal{P}^{\oplus n}, L)$, where $\mathcal{P}$ is the $C^{\infty}([0,1])$-relation defined by \cref{eq:heat_equation} with input $v_j$ and output $x_j$, and $L$ is the matrix
\begin{align*}
	L &= \begin{bmatrix} \sigma_{11} & \cdots & \sigma_{1n} \\ \vdots & \ddots & \vdots \\ \sigma_{n1} & \cdots & \sigma_{nn} \end{bmatrix} - \begin{bmatrix} \sum_{i=1}^{n} \sigma_{1i} & & 0 \\ & \ddots & \\ 0 & & \sum_{i=1}^{n} \sigma_{ni} \end{bmatrix}. 
\end{align*} 
The closed-loop $\sysclose(\mathcal{P}, \lambda)$ is the $C^{\infty}$-relation defined by
\begin{align} \label{eq:heat_closed}
	\left\{
	\begin{aligned}
		\partial_t x(t, \xi) &= \partial_{\xi}^2 x(t, \xi) \qquad \text{on } \RRge{0} \times [0,1], \\
		\partial_{\xi} x(t, 0) &= 0, \\
		\partial_{\xi} x(t, 1) &= \lambda M x(t) + M v(t),
	\end{aligned}
	\right.
\end{align}
with input $v$ and output $x$.
By \cref{thm:parabolic_neumann_input}, the system $\sysclose(\mathcal{P}, \lambda)$ is \prop{} in $C^{\infty}([0,1])$.
It follows from \cref{cor:sync} that the network $\sysclose(\mathcal{P}^{\oplus n}, L)$ synchronizes in $C^{\infty}$ iff the $\sysclose(\mathcal{P}, \lambda)$ is stable in $C^{\infty}$ for every $\lambda \in \spec(L)$, excluding the zero eigenvalue associated to the eigenvector $\mathbf{1}_n$.
}

\change{
Since stability in $C^{\infty}$ is weaker than the smooth VIVO property in $C^{\infty}$ and stronger than stability in $L^2$, we deduce that stability of $\sysclose(\mathcal{P}, \lambda)$ in $C^{\infty}$ is equivalent to stability in $L^2$.
In view of condition (b) of \cref{pro:parabolic_neumann_resilient}, $\sysclose(\mathcal{P}, \lambda)$ is stable iff the trivial solution is the boundary-value problem $\mu x(\xi) = \partial_{\xi}^2 x(\xi)$, $\partial_{\xi} x(0) = 0$, $\partial_{\xi} x(1) = \lambda x(0)$ has no non-trivial solution with $\mu \in \CCge{0}$. 
For fixed $\mu$, the boundary-value problem has a non-trivial solution iff $\sqrt{\mu} \sinh(\sqrt{\mu}) = \lambda$.
Therefore, $\sysclose(\mathcal{P}, \lambda)$ is stable iff $\lambda \in \CC$ is outside the set
\begin{align*}
	\{re^{i\theta} \sinh(re^{i\theta}) \mid r \geq 0,\, \theta \in [-\pi/4, \pi/4] \}.
\end{align*}
We conclude that the network $\sysclose(\mathcal{P}^{\oplus n}, L)$ synchronizes in $C^{\infty}$ iff all eigenvalues of $L$, except one instance of zero, belong to the shaded region in \cref{fig:heat_equation_stabreg}.
}

\begin{figure}[t]
	\change{
	\centering
	\includegraphics[width=6.7cm]{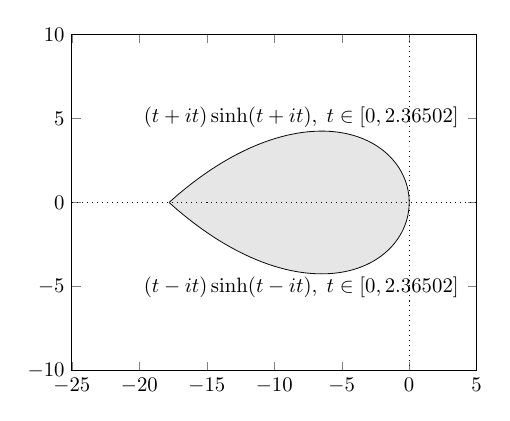}
	\caption{\change{The shaded region is the set of $\lambda \in \CC$ for which the closed-loop system $\sysclose(\mathcal{P}, \lambda)$ in \Cref{sect:example} is stable.}}
	\label{fig:heat_equation_stabreg}
	}
\end{figure}

\change{
The evolution of a network of three heat equations with the coupling matrix
\begin{align} \label{eq:heat_equation_coupling}
	L &= \alpha \begin{bmatrix} -1 & 1 & 0 \\ 0 & -1 & 1 \\ 1 & 0 & -1 \end{bmatrix}
\end{align}
is simulated by approximating each PDE by an ODE using the finite difference method with a spatial step-size of $\Delta \xi = 0.02$ and solving the system of three ODEs using the Matlab ode45 solver with a temporal step-size of $\Delta t = 0.05$. 
The simulation results for several choices of $\alpha$ are shown in \cref{fig:heat_simulation}.
Since the nonzero eigenvalues of $L$ are $\alpha(-3 \pm \sqrt{3}) / 2$, \cref{fig:heat_equation_stabreg} predicts synchronization when $\alpha$ is in the interval $(0, 4.8709)$.
The prediction agrees with the simulation results.
}

\begin{figure}[t]
	\begin{subfigure}{\textwidth}
		\begin{minipage}{0.2\textwidth}
			\caption{$\alpha = 4$}
		\end{minipage}%
		\begin{minipage}{0.8\textwidth}
		\includegraphics[width=0.5\textwidth]{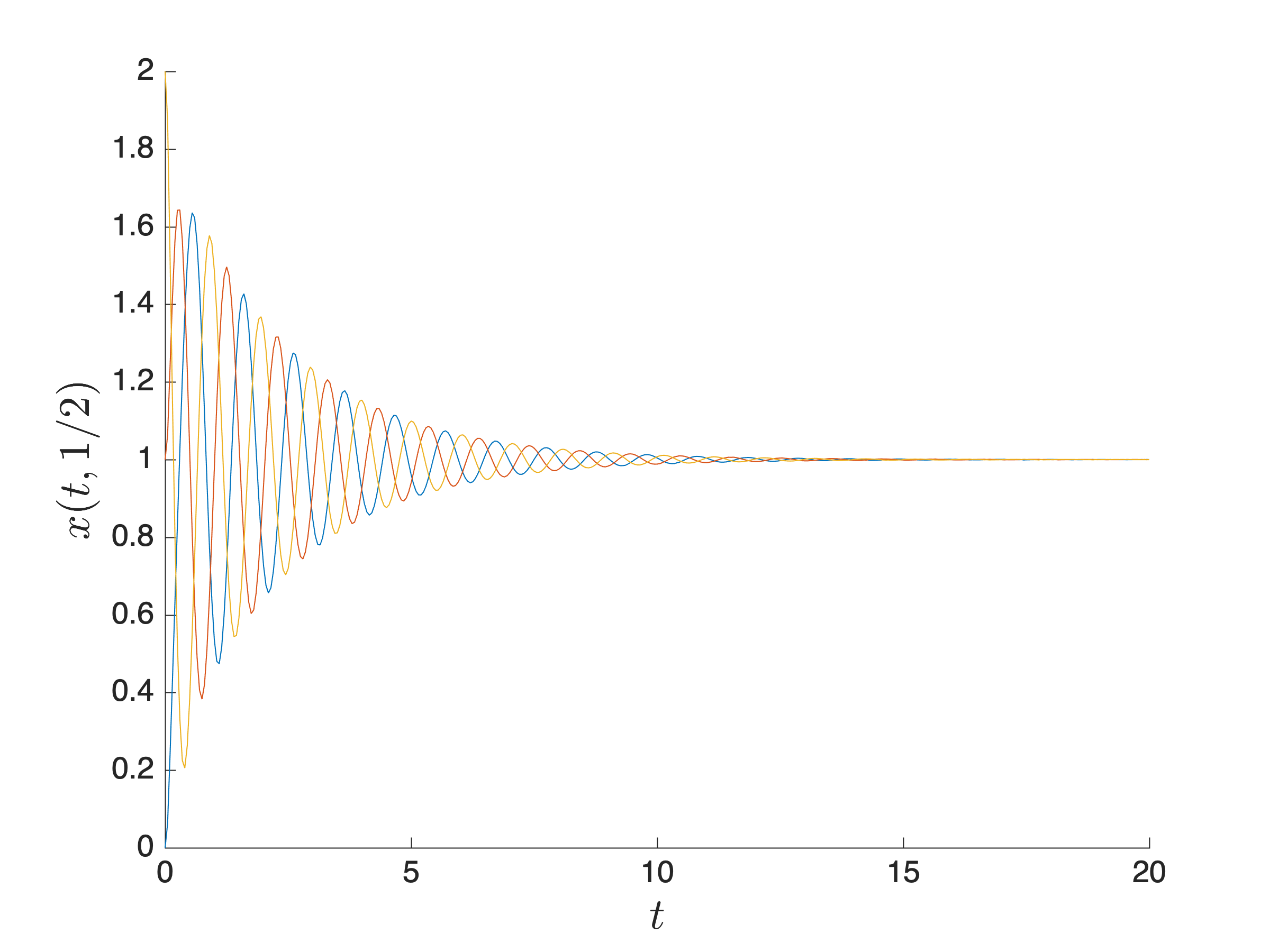}%
		\includegraphics[width=0.5\textwidth]{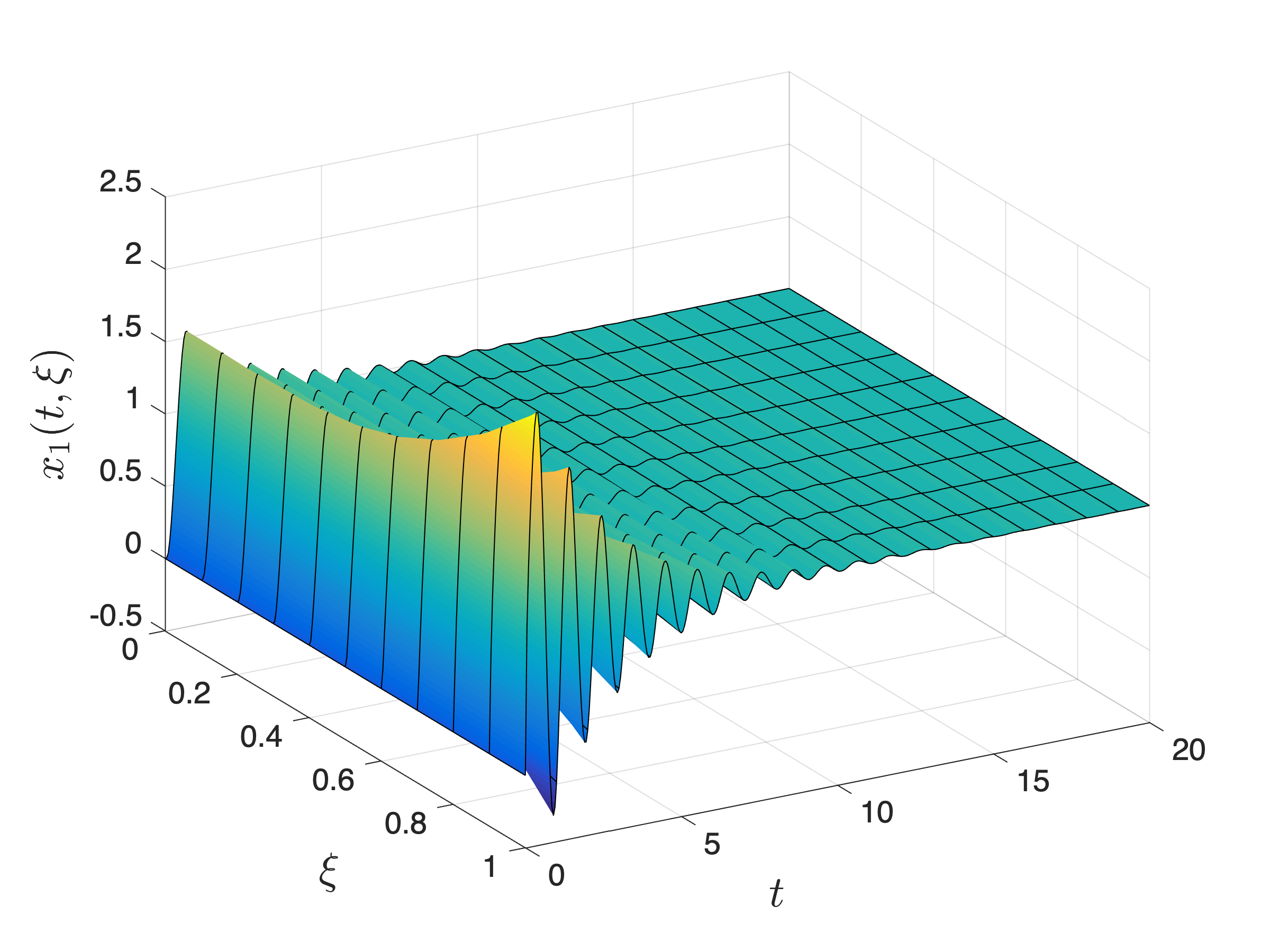}
		\end{minipage}%
	\end{subfigure}%
	
	\begin{subfigure}{\textwidth}
		\begin{minipage}{0.2\textwidth}
			\caption{$\alpha = 4.77$}
		\end{minipage}%
		\begin{minipage}{0.8\textwidth}
		\includegraphics[width=0.5\textwidth]{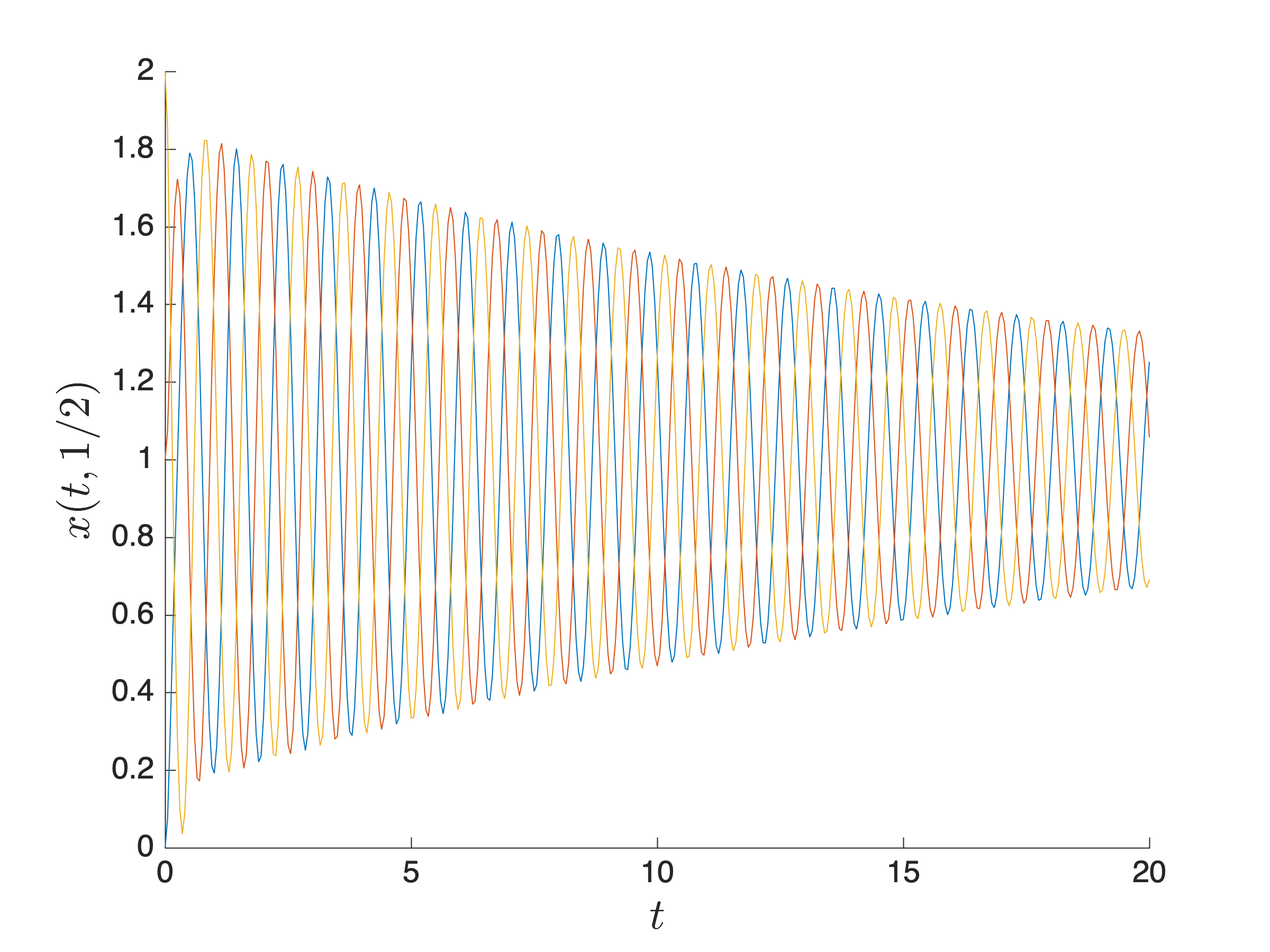}%
		\includegraphics[width=0.5\textwidth]{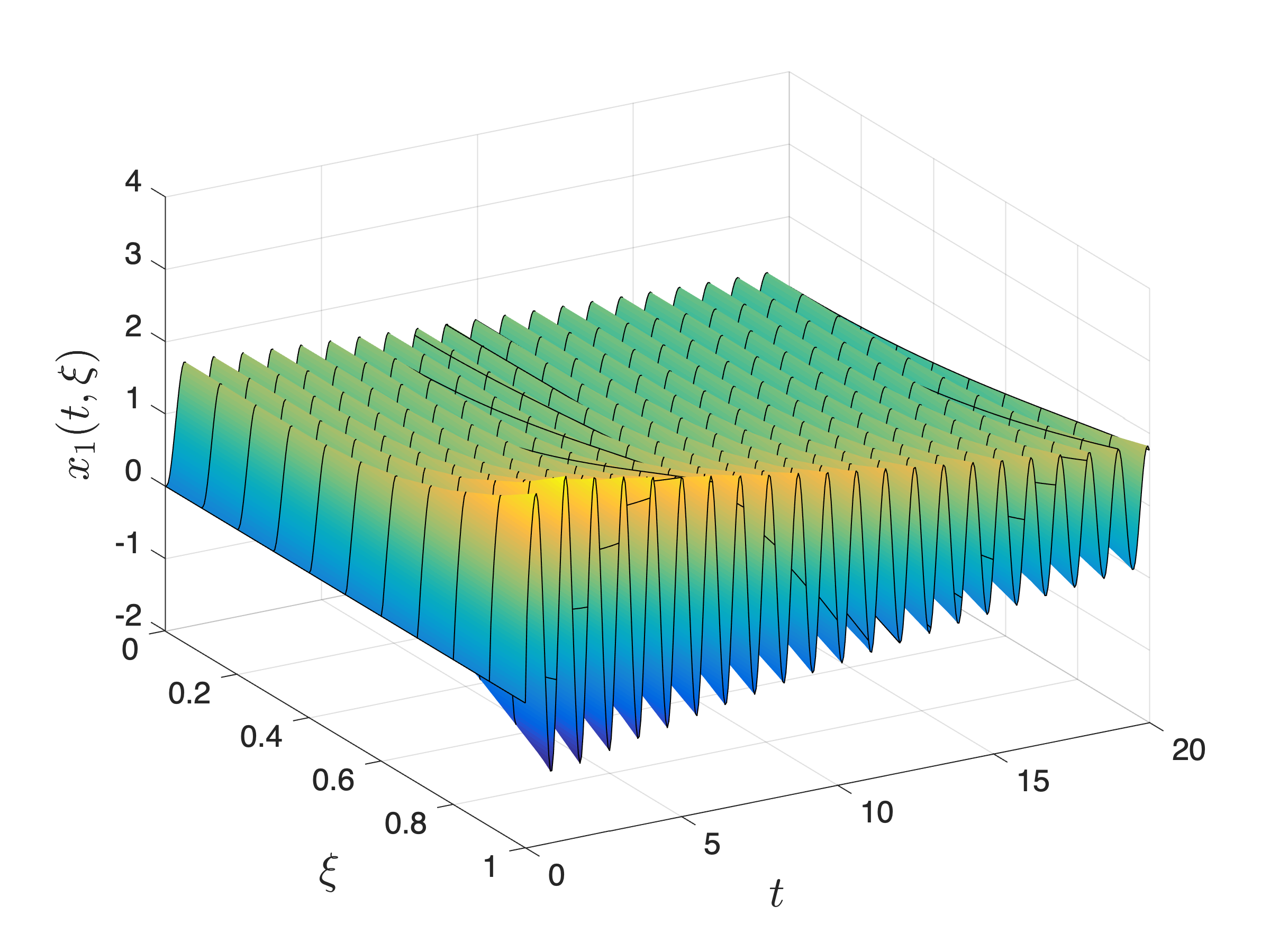}
		\end{minipage}
	\end{subfigure}%
	
	\begin{subfigure}{\textwidth}
		\begin{minipage}{0.2\textwidth}
			\caption{$\alpha = 4.87$}
		\end{minipage}%
		\begin{minipage}{0.8\textwidth}
		\includegraphics[width=0.5\textwidth]{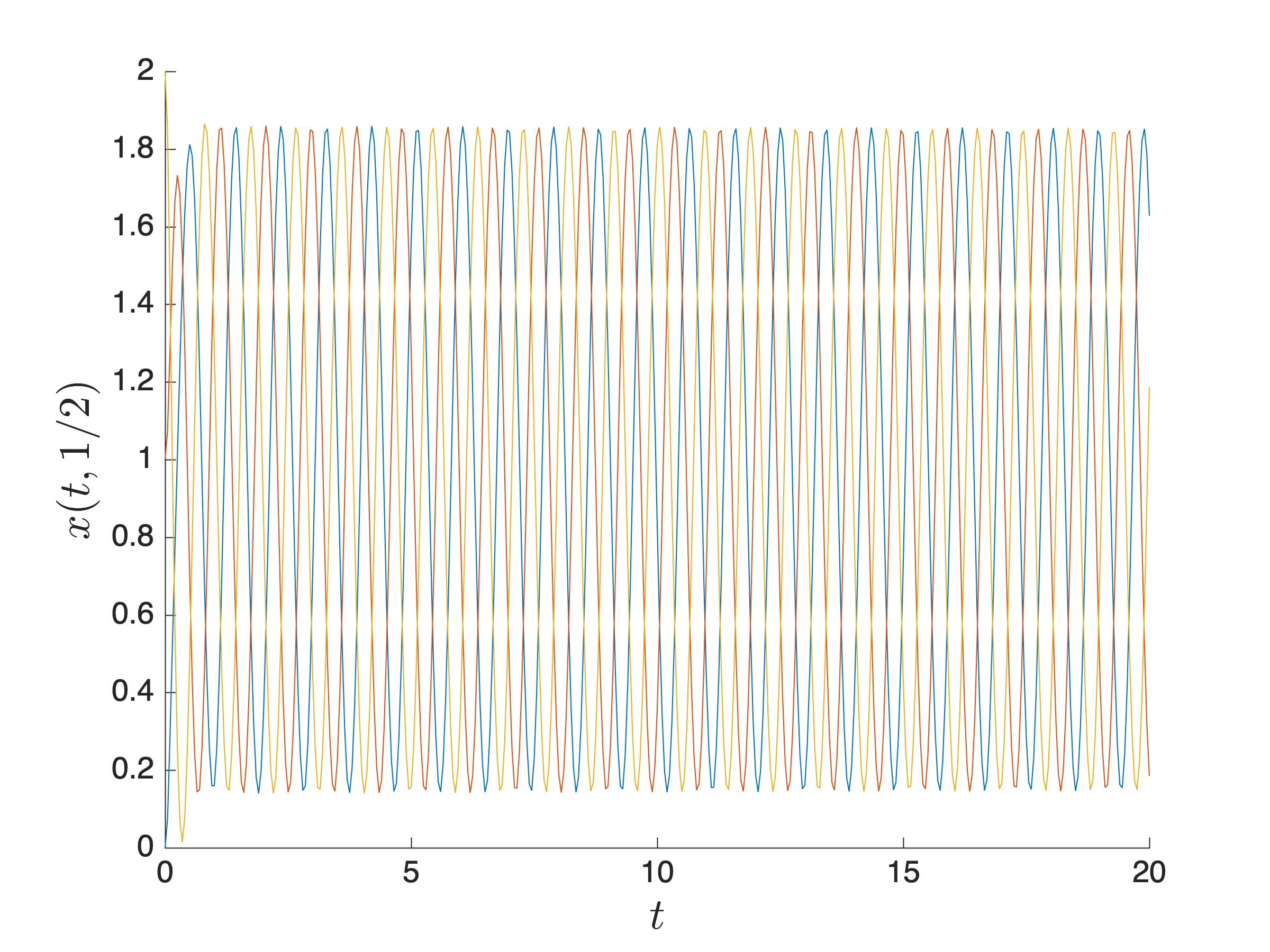}%
		\includegraphics[width=0.5\textwidth]{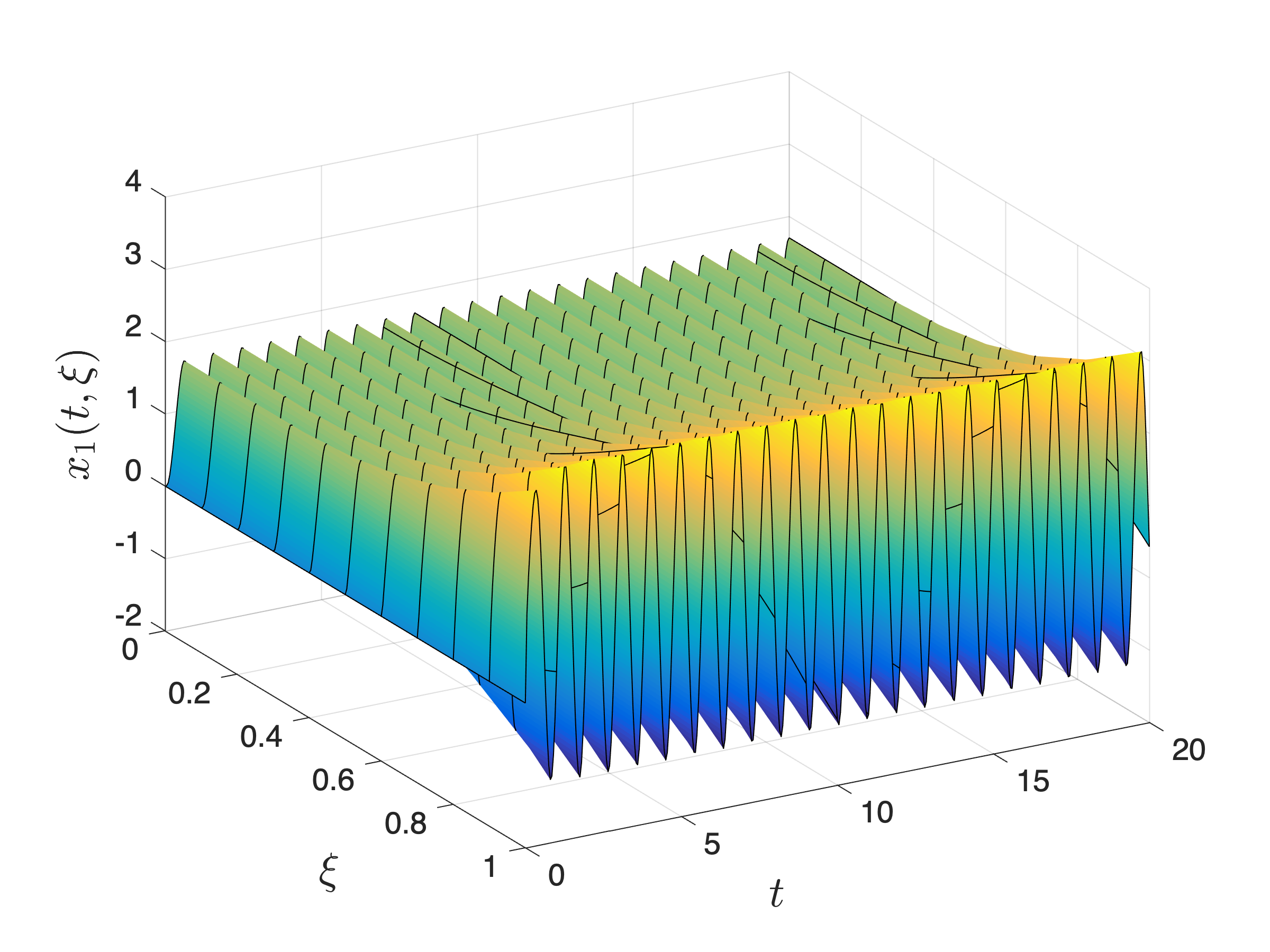}
		\end{minipage}%
	\end{subfigure}%
	
	\begin{subfigure}{\textwidth}
		\begin{minipage}{0.2\textwidth}
			\caption{$\alpha = 4.97$}
		\end{minipage}%
		\begin{minipage}{0.8\textwidth}
		\includegraphics[width=0.5\textwidth]{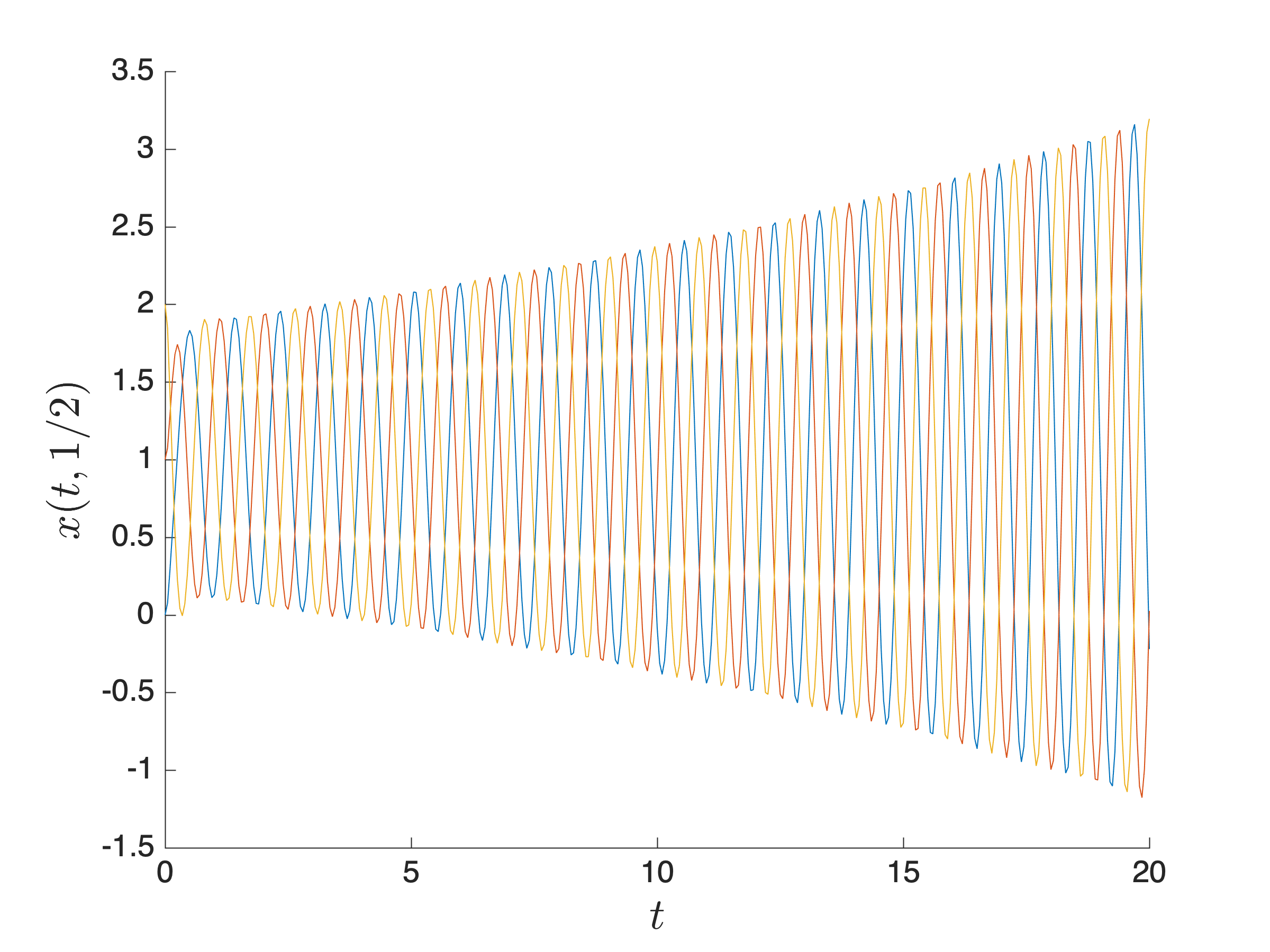}%
		\includegraphics[width=0.5\textwidth]{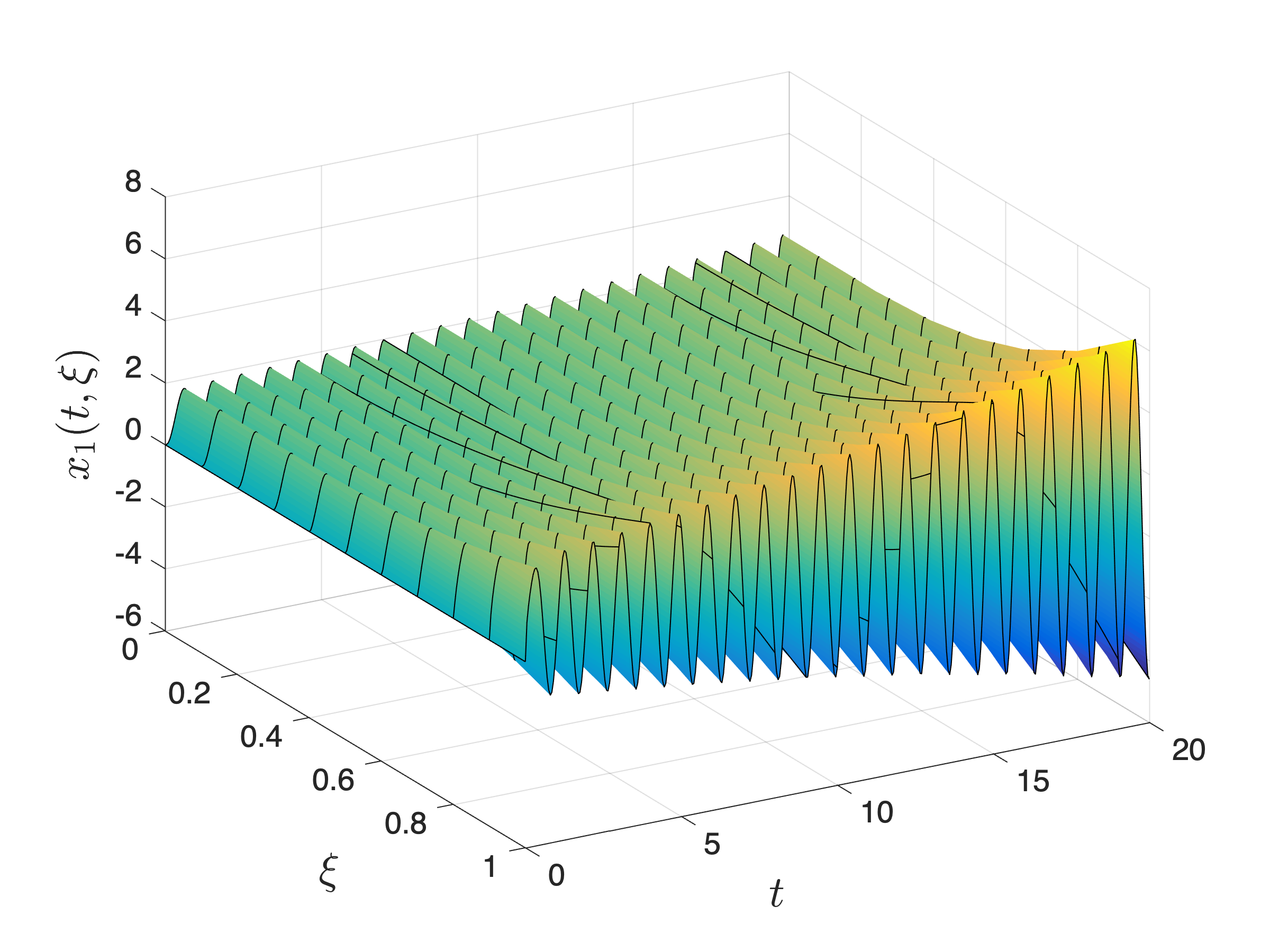}
		\end{minipage}%
	\end{subfigure}%
	
	\caption{\change{Simulation of the network $\sysclose(\mathcal{P}^{\oplus 3}, L)$ with $\mathcal{P}$ defined by \cref{eq:heat_equation} and $L$ defined by \cref{eq:heat_equation_coupling}. 
	Left: the evolution of the states $x_1, x_2, x_3$ evaluated at $\xi = 1/2$.
	Right: the evolution of the $x_1$.}}
	\label{fig:heat_simulation}
\end{figure}

\section{Delay differential equations} \label{sect:delayed}

Consider the equation
\begin{align} \label{eq:delayed}
	\begin{aligned}
		\dot{x}(t) &= \sum_{j=0}^{m} \left( A_j x(t - t_j) + B_j u(t-t_j) \right),
	\end{aligned}
\end{align}
where $0 = t_0 < t_1 < \dots < t_m$, $A_j \in \Endom(X)$, and $B_j \in \Homom(U, X)$ for finite-dimensional spaces $U$ and $X$.
We will show that the relation $\mathcal{P}$ defined by \cref{eq:delayed} with input $u \in C^{\infty}_{\local}(\RRge{0}, U)$ and output $x \in C^{\infty}_{\local}(\RRge{0}, X)$ is \prop{}.
To be precise, $\mathcal{P}$ contains smooth $(u, x)$ such that \cref{eq:delayed} holds on $t \geq t_m$.

\begin{proposition} \label{pro:delayed_welldefined} \stepcounter{refer}
	The relation $\mathcal{P}$ defined by \cref{eq:delayed} satisfies
	\begin{enumerate}[(a)]
		\item\labellocal{item:a}
		Given $u \in C^{\infty}_{\local}(\RRge{0}, U)$ and $x \in C^{0}_{\local}(\RRge{0}, X)$, the inclusion $(u, x) \in \mathcal{P}$ holds if and only if the restriction of $x$ to $[0, t_m]$ is a smooth function whose left-sided derivatives at $t_m$ satisfy
		\begin{talign} \labellocal{eq:1}
			\begin{aligned}
				\partial_t^{i+1} x(t_m) &= \sum_{j=0}^{m} A_j \partial_t^{i} x(t_m - t_j)
				 + \sum_{j=0}^{m} B_j \partial_t^{i} u(t_m-t_j)
			\end{aligned}
		\end{talign}
		for all $i \in \ZZge{0}$, and for $t \geq t_m$,
		\begin{talign} \label{eq:delay_integral}
			\begin{aligned}
				x(t) &= e^{A_0 (t - t_m)} x(t_m) 
				 + \int_{t_m}^{t} e^{A_0 (t - s)} \sum_{j=1}^{m} A_j x(s-t_j) ds \\
				&\qquad + \int_{t_m}^{t} e^{A_0 (t - s)} \sum_{j=0}^{m} B_j u(s-t_j) ds.
			\end{aligned}
		\end{talign}
		\item\labellocal{item:b} $\mathcal{P}$ is well-defined.
		\item\labellocal{item:c} $\mathcal{P}$ is closed.
	\end{enumerate}
\end{proposition}
\begin{proof}
	Proof of \reflocal{item:a}.
	Necessity is obvious. 
	In fact, given a continuous input and a initial condition $x \in C^{0}_{\local}([0, t_m], X)$, the unique continuous extension of $x$ solving \cref{eq:delayed} on $t > t_m$ is given by \cref{eq:delay_integral}. 
	This can be seen by writing \cref{eq:delayed} as $\dot{x} = A_0 x + v$ where $v(t) = v(t) = \sum_{j=1}^{m}A_j x(t - t_j) + \sum_{j=0}^{m}B_j u(t - t_j)$.
	
	To prove sufficiency, note that $x$ must verify \cref{eq:delayed} on $t > t_m$ by \cref{eq:delay_integral}.
	Therefore, we deduce by induction that $x \in C^{\infty}_{\local}(\RRg{t_m}, X)$. 
	By continuity, the right-sided derivatives of $x$ at $t_m$ verify \creflocal{eq:1}.
	Since the left-sided and right-sided derivatives agree at $t_m$, $x \in C^{\infty}_{\local}(\RRge{0}, X)$.
	
	Proof of \reflocal{item:b}.
	Fix any $u \in C^{\infty}_{\local}(\RRge{0}, U)$ and $x \in C^{\infty}([0, t_m-t_1], X)$, we will extend $x$ such that $(u, x) \in \mathcal{P}$.
	Pick any the sequence $\{\partial_t^{i} x(t_m)\}_{i}$ verifying \creflocal{eq:1}.
	There exists by Borel's lemma 
	\begin{shownto}{SIAM}
		\cite[Lemma 8.5]{xia23} a
	\end{shownto}
	\begin{shownto}{arXiv}
		(\cref{lem:borel}) a
	\end{shownto}
	n extension $x \in C^{\infty}([0, t_m], X)$ satisfying \creflocal{eq:1}.
	Define $x$ on $\RRge{t_m}$ using \cref{eq:delay_integral} yields a function $x$ satisfying $(u, x) \in \mathcal{P}$.
	
	Proof of \reflocal{item:c}. $\mathcal{P}$ is an intersection of continuous constraints.
\end{proof}

To prove \propnoun{} of $\mathcal{P}$, we will study a discrete-time representation of \cref{eq:delayed}.
The discrete-time representation is obtained using 

\begin{lemma} \label{pro:delayed_solution} \stepcounter{refer}
	Consider \cref{eq:delayed}.
	There exist functions $p \in C^{0}_{\local}(\RRge{t_m}, \Endom(X))$, $f \in L^{\infty}_{\local}(\RRge{t_m} \times [0, t_m], \Endom(X))$, $g \in L^{\infty}_{\local}(\RRge{t_m} \times \RRge{0}, \Homom(U,X))$, such that every $(u, x) \in C^{0}_{\local}(\RRge{0}, U) \times C^{0}_{\local}(\RRge{0}, X)$ solves \cref{eq:delay_integral} for all $t \geq t_m$ if and only if it solves
	\begin{align} \labellocal{eq:2}
			\begin{aligned}
			x(t) &= p(t) x(t_m) + \int_{0}^{t_m} f(t, s) x(s) ds 
			 + \int_{0}^{t} g(t, s) u(s) ds
			\end{aligned}
	\end{align}
	for all $t \geq t_m$.
	Furthermore, $f, g$ may be chosen to be continuous outside of 
	\begin{align*}
		S := \bigcup_{j \in \{0, \ldots, m\}} \{(t, s) \mid s = t_m - t_j \text{ or } s = t - t_j\}.
	\end{align*}
\end{lemma}
\begin{proof}
	For every input $u \in C^{0}_{\local}(\RRge{0}, U)$ and initial condition $x \in C^{0}_{\local}([0,t_m], X)$, \cref{eq:delay_integral} has a unique solution $x \in C^{0}_{\local}(\RRge{0}, X)$.
	Thus, it suffices to prove that solutions of \cref{eq:delay_integral} are solutions of \creflocal{eq:2}.
	Furthermore, at fixed $t \geq 0$, the functions $p(t), f(t, \cdot), g(t, \cdot)$ are unique if they exist, so it suffices to prove their existence on bounded $t$-intervals.
	Let $T_k = [t_m, t_m + (k+1) t_1]$.
	We will prove by induction on $k$ that there exists $p$, $f$, $g$ such that continuous solutions $(u,x)$ of \cref{eq:delay_integral} solve \creflocal{eq:2} on $t \in T_k$.
	The functions $p$, $f$, $g$ obtained at step $k$ are distinguished as $p_k$, $f_k$, $g_k$.
	
	By simple algebraic manipulations, \cref{eq:delay_integral} is equivalent to
	\begin{align} \labellocal{eq:3}
		x(t) &= p_0(t) x(t_m) + \int_{0}^{t} f_0(t, s) x(s) ds 
		 + \int_{0}^{t} g_0(t, s) u(s) ds
	\end{align}
	with
	\begin{align*}
		p_0(t) &= e^{A_0(t-t_m)},\quad f_0(t, s) = \sum_{j=1}^{m} e^{A_0(t-t_j-s)} A_j \ind{\{t_m - t_j \leq s \leq t - t_j\}},\\
		g_0(t, s) &= \sum_{j=0}^{m} e^{A_0(t-t_j-s)} B_j \ind{\{t_m - t_j \leq s \leq t - t_j\}}.
	\end{align*}
	When $t \in T_0 = [t_m, t_m + t_1]$, the function $f_0(t, \cdot)$ is supported on $[0,t_m]$.
	This proves the base case $k = 0$.
	
	Consider the inductive case $k > 0$.
	When $t \in T_k$, the first integral in \creflocal{eq:3} may be restricted to $[0,t_m+kt_1]$. 
	Expanding $x$ in the integrand \creflocal{eq:3} using the inductive hypothesis shows that \creflocal{eq:2} holds with
	\begin{align*}
		p_k(t) &= p_0(t) + \int_{0}^{t_m + k t_1} f_0(t, r) p_{k-1}(r) dr,
	\quad	f_k(t, s) = \int_{0}^{t_m + k t_1} f_0(t, r) f_{k-1}(r, s) dr, \\
		g_k(t, s) &= g_0(t, s) + \int_{0}^{t_m + k t_1} f_0(t, r) g_{k-1}(r, s) dr.
	\end{align*}
	It is easy verified that $p_k, f_k, g_k$ are locally bounded. 
	It suffices to prove their continuity properties.
	By the inductive hypothesis, the function
	\begin{align*}
		\RRge{0} &\rightarrow L^2([0, t_m + k t_1], \Endom(X)) : t \mapsto f_0(t, \cdot)
	\end{align*}
	is continuous, hence $p_k(t) - p_0(t)$ is continuous.
	Similarly,
	\begin{align*}
		[0, t_m] &\rightarrow L^2([0, t_m + k t_1], \Endom(X)) : s \mapsto f_{k-1}(\cdot, s), \\
		\RRge{0} &\rightarrow L^2([0, t_m + k t_1], \Homom(U, X)) : s \mapsto g_{k-1}(\cdot, s), 
	\end{align*}
	are continuous unless $s \in \{t_m - t_0, \ldots, t_m - t_m\}$, so $f_k(t, s)$ and $g_k(t, s) - g_0(t, s)$ are continuous away from $s \in \{t_m - t_0, \ldots, t_m - t_m\}$.
\end{proof}

For any $(u, x) \in C^{0}_{\local}(\RRge{0}, U) \times C^{0}_{\local}(\RRge{0}, X)$ solving \cref{eq:delayed} on $t > t_m$,
let $u_d[k] \in C^{0}([0,t_m], U)$ be the restriction of $u$ to the time interval $[kt_m, (k+1)t_m]$, and let $x_d[k] \in C^{0}([0,t_m], X)$ be analogously defined, then $(u_d, x_d)$ solves the discrete-time equation
%
\begin{align} \label{eq:delayed_discrete}
	x_d[k+1] &= P x_d[k] + Q u_d[k] + R u_d[k + 1],
\end{align}
where $P : C^{0}([0, t_m], X) \rightarrow C^{0}([0, t_m], X)$, $Q, R : C^{0}([0, t_m], U) \rightarrow C^{0}([0, t_m], X)$ are bounded operators defined by
\begin{align} \label{eq:delayed_discrete_operator}
	\begin{aligned}
		(Pv)(t) &= p(t_m + t) v(t_m) + \int_{0}^{t_m} f(t_m + t, s) v(s) ds, \\
		(Qv)(t) &= \int_{0}^{t_m} g(t_m + t, s) v(s) ds, \quad
		(Rv)(t) = \int_{0}^{t} g(t_m + t, s + t_m) v(s) ds,
	\end{aligned}
\end{align}
where $p, f, g$ are the functions defined in \cref{pro:delayed_solution}.
This connection allows us to prove the following characterization for the stability and VIVO property of $\mathcal{P}$.

\begin{theorem} \stepcounter{refer}
	\hspace{3mm} Let $\mathcal{P}$ be the $(U, X)$-relation defined by \cref{eq:delayed} and $\mathcal{P}_d$ be the $(C^{0}([0, t_m], U), \allowbreak C^{0}([0,t_m], X))$-relation defined by \cref{eq:delayed_discrete}.
	The following are equivalent
	\begin{enumerate}[(a)]
		\item\labellocal{item:a} If $(0, x_d) \in \mathcal{P}_d$ and $x_d[k+1] = \lambda x_d[k]$ for some $|\lambda| \geq 1$, then $x_d = 0$.
		\item\labellocal{item:b} $\mathcal{P}_d$ is stable.
		\item\labellocal{item:c} $\mathcal{P}_d$ is VIVO.
		\item\labellocal{item:d} $\mathcal{P}$ is VIVO.
		\item\labellocal{item:e} $\mathcal{P}$ is smoothly VIVO. 
		\item\labellocal{item:f} $\mathcal{P}$ is stable.
	\end{enumerate}
\end{theorem}
\begin{proof}
	Proof of \reflocal{item:a} $\Leftrightarrow$ \reflocal{item:b} $\Rightarrow$ \reflocal{item:c}.
	$\mathcal{P}_d$ is the serial composition of the relation defined by $v_d[k+1] = Q u_d[k] + R u_d[k + 1]$ with input $u_d$, output $v_d$, and the relation defined by $x_d[k+1] = P x_d[k] + v_d[k+1]$ with input $v_d$, output $x_d$.
	By \cref{lem:delayed_discrete_compact} below, the operator $P$ is compact.
	The conclusions follow from applying \cref{pro:transition_io_input_resilient} to the latter relation.
	
	Proof of \reflocal{item:c} $\Rightarrow$ \reflocal{item:d}
	follows from the discussion above.
	
	Proof of \reflocal{item:d} $\Rightarrow$ \reflocal{item:e}
	follows from closedness of $\mathcal{P}$.
	
	Proof of \reflocal{item:e} $\Rightarrow$ \reflocal{item:a}.
	Assume that \reflocal{item:a} fails and let $x_d$ be a function violating \reflocal{item:a}.
	There is a unique $x \in C^{0}_{\local}(\RRge{0}, X)$ such that $x_d[k](t) = x(k t_m + t)$.
	By \cref{pro:delayed_solution}, $(0, x_d) \in \mathcal{P}_d$ implies that $(0, x)$ solves \cref{eq:delayed} on $t > t_m$.
	Fix $\eta \in C^{\infty}_{c}(\RRg{0}, \RR)$, the function $\tilde{x}$ defined by
	\begin{align*}
		\tilde{x}(t) = \int_{0}^{\infty} x(t+s) \eta(s) ds
	\end{align*}
	is smooth and satisfies $(0, \tilde{x}) \in \mathcal{P}$.
	Since $x$ is periodic up to the scaling factor $\lambda$, $\tilde{x}$ is non-vanishing as long as we pick $\eta$ such that $\tilde{x} \neq 0$.
\end{proof}

\begin{lemma} \label{lem:delayed_discrete_compact} \stepcounter{refer}
	The operator $P$ on $C^{0}([0, t_m], X)$ defined in \cref{eq:delayed_discrete_operator} is compact.
\end{lemma}
\begin{proof}
	Write $P = P_1 + P_2$ for $(P_1v)(t) = p(t_m+t)v(t_m)$ and $(P_2v)(t) = \int_{0}^{t_m} f(t_m + t, s) v(s) ds$.
	The operator $P_1$, which factors through $\CC^n$, is compact.
	It suffices to prove that $P_2$ is compact, by showing that it is the limit of finite-rank operators. 
	
	By the continuity property of $f$ stated in \cref{pro:delayed_solution}, for every $\epsilon > 0$, there exists (justified later) a function $h \in L^{\infty}([0, t_m] \times [0, t_m], \Endom(X))$ of the form
	\begin{align*} 
		h(t, s) &= \sum_{j=1}^{k} a_j(t) b_j(s),
	\end{align*}
	where $a_j \in C^{0}([0, t_m], \Endom(X))$ and $b_j \in L^{\infty}([0, t_m], \RR)$, 
	such that $h$ approximates $f$ in the sense
	\begin{align} \labellocal{eq:2}
		\sup_{t \in [0, t_m]} \lVert h(t, \cdot) - f(t + t_m, \cdot) \rVert_{L^1([0, t_m], \Endom(X))} \leq \epsilon.
	\end{align}
	This implies that
	\begin{align*}
		\left\lvert \int_{0}^{t_m} \left(h(t, s) - f(t_m + t, s)\right) v(s) ds\right\rvert &\leq \epsilon \lVert v \rVert_{L^{\infty}}
	\end{align*}
	for all $t \in [0, t_m]$, hence $P_2$ is approximated by the integral operator $(Hv)(t) = \int_{0}^{t} h(t,s) v(s) ds$.
	Since the rank of $H$ is at most $k \dim X < \infty$, $P_2$ is a limit of finite-rank operators, hence is compact.
	
	Proof of the existence of $h$.
	We first consider the case $f$ is continuous on $[t_m, 2t_m] \times [0, t_m]$.
	In this case, pick points $0 = s_0 < s_1 < \dots < s_k = t_m$ on the interval $[0, t_m]$;
	partition $[t_m, 2t_m] \times (0, t_m]$ into strips $[t_m, 2t_m] \times (s_{j-1}, s_j]$;
	and approximate $f$ on each strip by $f(t, s_j)$.
	In other words, we set $h(t, s) = \sum_{j=1}^{k} f(t_m + t, s_j) \ind{(s_{j-1}, s_j]}(s)$, then \creflocal{eq:2} holds as long as the strips are sufficiently fine.
	
	In general, $f$ may be discontinuous, but it is only discontinuous on the set $S$ defined in \cref{pro:delayed_solution}.
	For every open set $V \supseteq S$, there is a function $f' \in C^{0}([t_m, 2t_m] \times [0, t_m], \Endom(X))$ such that $f' = f$ on the complement of $V$ and $\lVert f' \rVert_{L^{\infty}}$ is bounded by $\lVert f \rVert_{L^{\infty}([t_m, 2t_m] \times [0, t_m])}$.
	Since $S$ is a finite union of lines, none of which parallel to the second coordinate axis of $[t_m, 2t_m] \times [0, t_m]$, choosing $V$ sufficiently small ensures that
	\begin{align*}
		\sup_{t \in [t_m, 2t_m]} \lVert f'(t, \cdot) - f(t, \cdot) \rVert_{L^1([0, t_m])} \leq \epsilon / 2.
	\end{align*}
	Therefore, any $h$ approximating $f'$ is our desired function.
\end{proof}

\section{Conclusion}

We have presented an input-output framework for the analysis of networks of infinite-dimensional systems.
Under mild assumptions, we derived a synchronization criterion, which can be seen as a generalization of an existing finite-dimensional result.
We have shown that these conditions can be verified for classes of abstract Cauchy problems, parabolic equations, and time-delay differential equations. A direction of future research is the application of the proposed input-output framework for control design.

\begin{shownto}{arXiv}

\section{Appendix} \label{sect:appendix}

\subsection{Sobolev spaces} \label{sect:sobolevspaces}

This section serves as a quick introduction to Sobolev spaces and a clarification of terminology to avoid possible ambiguities.
Acquaintance with distributions, tempered distributions, and Fourier transform is assumed.
The latter topics are well explained in many functional analysis or PDE books (e.g. \cite[Sec. 6, 7]{rudin91}).

For any $s \in \RR$, the Sobolev space $H^s(\RR^d)$ consists of tempered distribution $f$ whose Fourier transform $\hat{f}$ is equal to $(1 + \|\cdot\|^2)^{-s/2} \hat{f}_s(\cdot)$ for some $\hat{f}_s \in L^2(\RR^d)$. 
With the norm $\|f\|_{H^2(\RR^d)} = \|\hat{f}_s\|_{L^2}$, $H^s(\RR^d)$ is a Banach space.
By the properties of Fourier transform, (distributional) derivative is a bounded operator $H^{s}(\RR^d) \rightarrow H^{s-1}(\RR^d)$, and when $s \in \ZZge{0}$, $H^s(\RR^d)$ coincides with $L^2$ functions having $L^2$ derivatives of order up to $s$.
Since the Schwartz space $\mathcal{S}(\RR^d)$ is dense in $L^2$, its inverse Fourier transform (also $\mathcal{S}$) is dense in $H^s(\RR^d)$.
Therefore, the usual integral pairing $\mathcal{S} \times \mathcal{S} \rightarrow \CC : (f, g) \mapsto \int f(x) g(x) dV = \int \hat{f}(-\xi) \hat{g}(\xi) dV$ extends uniquely to a continuous pairing $H^s \times H^{-s} \rightarrow \CC$, which often is also denoted by $\int f g dV$ even though the integrand is not a measurable function.
Under this pairing, $H^{-s}$ is the dual of $H^s$.
Since $C^{\infty}_{c}(\RR^d)$ is dense in $\mathcal{S}(\RR^d)$, the former is also dense in $H^s(\RR^d)$, so operations on $C^{\infty}_c(\RR^d)$ can be extended to $H^s(\RR^d)$ by continuity.
For example, the integration by part formula $\int f \partial_1 g dV = - \int g \partial_1 f dV$ holds for all $f, g \in C^{\infty}_{c}(\RR^d)$, hence it also holds for $f \in H^{s}$ and $g \in H^{1-s}$.
Three essential results of Sobolev spaces are

\begin{lemma}[Sobolev inequality] \label{lem:sobolev_inequality}
	If $s > d/2$, then $H^s(\RR^d) \subseteq C^0(\RR^d)$ (both as subsets of distributions) and $\|f\|_{L^{\infty}(\RR^d)} \lesssim \|f\|_{H^s(\RR^d)}$ uniform in $f \in H^s(\RR^d)$.
\end{lemma}
\begin{proof}
	For any $f \in C^{\infty}_c(\RR^d)$, write $\hat{f} = (1 + \|\cdot\|^2)^{-s/2} \hat{f}_s$, then
	\begin{align*}
		\|f\|_{L^{\infty}} &\leq \|\hat{f}\|_{L^1} \leq \left\|(1 + \|\cdot\|^2)^{-s/2} \right\|_{L^2} \|\hat{f}_s\|_{L^2} \lesssim \|f\|_{H^s}
	\end{align*}
	uniform in $f \in C^{\infty}_c(\RR^d)$.
	Furthermore, $f$ is continuous because it is the Fourier transforms of an $L^1$ function.
\end{proof}

\begin{lemma}[Trace inequality] \label{lem:trace}
	If $s > 0$, then $\|f(0, \cdot)\|_{H^s(\RR^d)} \lesssim \|f\|_{H^{s+1/2}(\RR \times \RR^d)}$ uniform in $f \in C^{\infty}_{c}(\RR \times \RR^d)$, where $f(0, \cdot)$ denotes the function on $\RR^d$ obtained from $f$ by restricting the first coordinate to $0$.
	In particular, the restriction map extends uniquely to a bounded operator $H^{s+1/2}(\RR \times \RR^d) \rightarrow H^s(\RR^d)$, called the trace operator.
\end{lemma}
\begin{proof}
	Note that the Fourier transform of $f(0, \cdot)$ is $\int_{\RR} \hat{f}(\xi, \cdot) d\xi$.
	Applying H\"{o}lder inequality to $\hat{f}(\xi, \eta) =  (1 + \xi^2 + \eta^2)^{-\frac{s+1/2}{2}} \hat{f}_{s + 1/2}(\xi, \eta)$, we get 
	\begin{align*}
		\left\lvert \int \hat{f}(\xi, \eta) d\xi \right\rvert^2 &\lesssim (1+\eta^2)^{-s} \int \lvert \hat{f}_{s+1/2}(\xi, \eta) \rvert^2 d\xi
	\end{align*}
	uniform in $\eta \in \RR^d$ and $f$.
	Multiplying by $(1+\eta^2)^{s}$ and integrating over $\eta$ yields the desired conclusion.
\end{proof}

\begin{lemma}[Compact embedding] \label{lem:sobolev_compact}
	Let $K \subseteq \RR^d$ be a compact subset and $s > t$.
	Let $S$ be the set of $f \in H^s(\RR^d)$ of norm at most $1$ and supported in $K$, then $S$ is precompact in $H^t(\RR^d)$.
\end{lemma}
\begin{proof}
	It suffices to show that for any $\epsilon > 0$, there is a finite subset of $H^t$ whose $\epsilon$-neighbourhood contains $S$.
	For any $f \in S$, write $\hat{f} = (1 + \|\cdot\|^2)^{-s/2} \hat{f}_s$ and $\hat{f} = (1 + \|\cdot\|^2)^{-t/2} \hat{f}_t$, then $\hat{f}_t = (1 + \|\cdot\|^2)^{(t-s)/2} \hat{f}_s$.
	Since $\|\hat{f}_s\|_{L^2} \leq 1$, we have $\|\hat{f}_t\|_{L^2(\RR^d \setminus B_{r})} \leq (1 + r^2)^{(t-s)/2}$ where $B_r$ denotes the solid ball of radius $r > 0$.
	Therefore, it suffices to show that $\{ \hat{f}_t \in L^2(B_r) \mid f \in S\}$ is precompact in $L^2(B_r)$.
	
	Since $f \in S$ is compact supported, there is $\eta \in C^{\infty}_{c}(\RR^d)$ such that $f = f \eta$.
	Therefore, $\hat{f} = \hat{f} * \hat{\eta}$ and $\hat{f} \in C^{\infty}_{\local}(\RR^d)$.
	The convolution writes $\hat{f}(\xi) = \langle \hat{f}_s, (1 + \|\cdot\|^2)^{-s/2} \hat{\eta}(\xi - \cdot) \rangle$, we see that $\sup_{f \in S} \|\hat{f}\|_{L^{\infty}(B_r)} < \infty$.
	Similarly, $\partial_j \hat{f} = \hat{f} * \partial_j \hat{\eta}$ implies that $\sup_{f \in S} \|\partial_j \hat{f}\|_{L^{\infty}(B_r)} < \infty$.
	By Arzel\`{a}-Ascoli, $\{ \hat{f}_t \mid f \in S\}$ is precompact in $C^0(B_r)$, and also $L^2(B_r)$.
\end{proof}

For any subset $U \subseteq \RR^d$, $H^s(U)$ denotes the set of distribution (i.e., linear functional $C^{\infty}_{c}(U^\circ) \rightarrow \CC$) obtained as restrictions of $H^s(\RR^d)$ to $U$.
$H^s(U)$ is a Banach space equipped with the norm $\|f\|_{H^s(U)} = \inf \{\|g\|_{H^s(\RR^d)} \mid g = f \text{ on } U\}$.
Clearly, $C^{\infty}_{c}(U)$ is dense in $H^s(U)$ when $U$ is closed.
In general, the closure of $C^{\infty}_{c}(U)$ in $H^s(U)$ is denoted by $H^s_0(U)$.
It is worth noting (for conceptual clarity only) that when $s \in \ZZge{0}$ and $U$ has Lipschitz boundary, $H^s(U)$ coincides with the set of $L^2(U)$-functions whose distributional derivatives of order up to $s$ belong to $L^2(U)$ (by the extension theorem \cite[Chpt. IV, Sec. 3]{stein70}).
For a compact manifold $M$, $H^s(M)$ denotes the set of distribution $f$ such that on any coordinate chart $(U, \phi)$ and compact subset $K \subseteq U$, we have $f \circ \phi^{-1} \in H^s(\phi(K))$.
All results stated for $H^s(\RR^d)$ have obvious analogues for $H^s(U)$ and $H^s(M)$.

For a compact manifold $M$ and $r < s < t$, we have compact inclusions $H^{r}(M) \hookrightarrow H^{s}(M) \hookrightarrow H^{t}(M)$.
The following simple result turns out to be very helpful in many situations.

\begin{lemma}[Ehrling interpolation] \label{lem:ehrling} \stepcounter{refer}
	Assume $X \hookrightarrow Y \hookrightarrow Z$ are continuous injections of Banach spaces such that $X \hookrightarrow Y$ is compact.
	For every $\epsilon > 0$, there exists $\alpha > 0$ such that $\lVert x \rVert_{Y} \leq \epsilon \lVert x \rVert_{X} + \alpha \lVert x \rVert_{Z}$ for all $x \in X$.
\end{lemma}

\subsection{Technical Lemmas}
This section collects some purely technical lemmas. These results are well-known, but their proofs are not always easily accessible. 
The proofs are included for the readers' convenience.

\begin{lemma}[Borel's lemma] \label{lem:borel} \stepcounter{refer}
	For any sequence $\{a_k\}_{k \in \ZZge{0}}$ of elements in a Banach space $X$, there exists a function $f \in C^{\infty}(\RR, X)$ such that $\partial_t^{k} f(0) = a_k$ for every $k \in \ZZge{0}$.
\end{lemma}
\begin{proof}
	Fix any $\eta \in C^{\infty}(\RR, X)$ such that $\eta(t) = 0$ when $\lvert t \rvert \geq 1$ and $\eta(t) = 1$ when $\lvert t \rvert \leq 1/2$.
	Define functions
	\begin{align*}
		f_k(t) &= a_k \cdot \frac{t^k}{k!} \cdot \eta\left(\frac{t}{b_k}\right)
	\end{align*}
	for parameters $b_k > 0$ to be determined.
	Clearly, $\partial_t^j f_k(0) = a_k$ if $j = k$ and $\partial_t^j f_k(0) = 0$ otherwise.
	If $\sum_{k \in \ZZge{0}} f_k$ converges in $C^{\infty}(\RR, X)$, then the sum is the desired $f$.
	
	To prove the convergence, it suffices to show that 
	\begin{align*}
		\sum_{j \in \ZZg{k}} \lVert  \partial_t^{k} f_j \rVert_{L^{\infty}} &< \infty
	\end{align*}
	for every $k \in \ZZge{0}$.
	When $j \geq k$,
	\begin{align*}
		\partial_t^{k} f_j(t)  &= a_j \sum_{i=0}^{k} {k \choose i} \frac{t^{j+i-k}}{(j+i-k)!} \cdot \frac{1}{b_j^i} \cdot \eta^{(i)}\left(\frac{t}{b_j}\right)
	\end{align*}
	where $\eta^{(i)}$ is the $i$th derivative of $\eta$.
	Because $\eta^{(i)}(t)$ is identically zero outside of $\lvert t \rvert < 1$, we have the bound
	\begin{align*}
		\lVert  \partial_t^{k} f_j \rVert_{L^{\infty}} &\leq \lVert a_j\rVert_{X} \sum_{i=0}^{k} {k \choose i} \frac{b_j^{j-k}}{(j+i-k)!} \cdot \lVert \eta^{(i)}\rVert_{L^{\infty}} \\
		&\leq b_j^{j-k} \lVert a_j\rVert_{X} \cdot 2^k \lVert \eta\rVert_{C^{k}}.
	\end{align*} 
	Picking $b_j$ such that $b_j \lVert a_j\rVert_{X} \leq 1$ and $b_j \leq 1/2$, then
	\begin{align*}
		\lVert  \partial_t^{k} f_j \rVert_{L^{\infty}} &\leq \frac{2^k \lVert \eta\rVert_{C^k}}{2^{j-k-1}}
	\end{align*}
	for all $j > k$.
	This yields the bound
	\begin{align*}
		\sum_{j \in \ZZg{k}} \lVert  \partial_t^{k} f_j \rVert_{L^{\infty}} &\leq 2^{k+1} \lVert \eta\rVert_{C^k}
	\end{align*}
	for every $k \in \ZZge{0}$, completing the proof.
\end{proof}

\begin{lemma}[Local elliptic regularity] \label{lem:elliptic_regularity_local_0} \stepcounter{refer}
	Given $\Omega = \RRge{0} \times \RR^{d-1}$ and a uniformly positive-definite $a \in C^{\infty}_{\local}(\Omega, \CC^{d \times d})$,
	denote by $S$ the set of all $(u, f, h) \in H^1_0(\Omega^{\circ}) \times L^2(\Omega) \times H^1(\Omega, \CC^d)$ solving
	\begin{align} \labellocal{eq:1} 
		-\int_{\Omega} (a \gradient u) \cdot (\gradient v) dV &= \int_{\Omega} f v + g \cdot \gradient v dV
	\end{align}
	for all $v \in C^{\infty}_{c}(\Omega^{\circ})$.
	For every $k \in \ZZge{0}$,
	\begin{align*}
		\lVert u \rVert_{H^{k+2}} &\lesssim \lVert u \rVert_{H^{k+1}} + \lVert f \rVert_{H^{k}} + \|g\|_{H^{k+1}}
	\end{align*}
	uniform in $(u, f, g) \in S$.
\end{lemma}
\begin{proof}
	We complete the proof by the following steps
	\begin{enumerate}[(a)]
		\item\labellocal{item:0} Let $\partial_j^{\epsilon}$ denote the difference quotient
		\begin{align*}
			\partial_j^{\epsilon} f(x) &= \frac{f(x + \epsilon e _j) - f(x)}{\epsilon}
		\end{align*}
		where $e_j$ is the $j$th standard unit vector in $\RR^d$. 
		For every $j \neq 1$ (i.e., derivative in directions parallel to $\partial \Omega$), we have $\lVert \partial_j^{\epsilon} \gradient u \rVert_{L^2(\Omega, \CC^d)} \lesssim \|u\|_{H^1} + \|f\|_{L^2} + \|g\|_{H^1}$ uniform in $\epsilon \in \RR$ near $0$ and $(u, f, g) \in S$.
		\item\labellocal{item:a} $\|\partial_j \gradient u\|_{L^2} \lesssim \|u\|_{H^1} + \|f\|_{L^2} + \|g\|_{H^1}$ for $j \neq 1$.
		\item\labellocal{item:b} $\|u\|_{H^2} \lesssim \|u\|_{H^1} + \|f\|_{L^2} + \|g\|_{H^1}$
		\item\labellocal{item:c} $\|u\|_{H^{k+2}} \lesssim \|u\|_{H^{k+1}} + \|f\|_{H^k} + \|g\|_{H^{k+1}}$ for $k \in \ZZge{1}$.
	\end{enumerate}
	
	Proof of \reflocal{item:0}.
	Since $C^{\infty}_{c}$ is dense in $H^1_{0}$, \creflocal{eq:1} also holds for all $v \in H^1_0(\Omega^{\circ})$.
	Substitute $v = \partial_j^{-\epsilon} \partial_j^{\epsilon} \bar{u}$ into \creflocal{eq:1} gives
	\begin{align} \labellocal{eq:2}
		\int_{\Omega} \partial_j^{\epsilon} (a \gradient u) \cdot (\overline{\partial_j^{\epsilon} \gradient u}) dV &= \int_{\Omega} f \partial_j^{-\epsilon} (\overline{\partial_j^{\epsilon} u}) dV \\
		&\qquad + \int_{\Omega} (\partial_j^{\epsilon} g) \cdot (\overline{\partial_j^{\epsilon} \gradient u}) dV \nonumber
	\end{align}
	Expanding $\partial_j^{\epsilon}(a\gradient u)$ using the product rule and using positive-definiteness of $a$, the left-hand side of \creflocal{eq:2} is greater than
	\begin{align*}
		\alpha \lVert \partial_j^{\epsilon} \gradient u \rVert_{L^2}^2 - \beta \lVert \gradient u \lVert_{L^2} \lVert \partial_j^{\epsilon} \gradient u \rVert_{L^2}
	\end{align*}
	for some small $\alpha > 0$ and some large $\beta > 0$.
	Using the fact $\lVert \partial_j^{\epsilon} v \rVert_{L^2} \leq \lVert \partial_j v \rVert_{L^2}$, the right-hand side of \creflocal{eq:2} is bounded by
	\begin{align*}
		\beta \left( \lVert f \rVert_{L^2} \lVert \partial_j^{\epsilon} \gradient u \rVert_{L^2} + \|f\|_{L^2} \|\gradient u\|_{L^2} + \lVert g \rVert_{H^1} \lVert \partial_j^{\epsilon} \gradient u \rVert_{L^2} \right)
	\end{align*}
	for some large $\beta > 0$.
	Therefore, \creflocal{eq:2} implies
	\begin{align*}
		\lVert \partial_j^{\epsilon} \gradient u \rVert_{L^2}^2 &\lesssim \lVert \gradient u \lVert_{L^2} \lVert \partial_j^{\epsilon} \gradient u \rVert_{L^2} + \|f\|_{L^2} \lVert \partial_j^{\epsilon} \gradient u \rVert_{L^2} \\
		&\qquad + \|f\|_{L^2} \|\gradient u\|_{L^2} + \| g \|_{H^1} \lVert \partial_j^{\epsilon} \gradient u \rVert_{L^2}
	\end{align*}
	Bound all products on the right-hand side using Cauchy-Schwarz inequality $2ab \leq a^2 / \delta + \delta b^2$.
	By picking $\delta$ appropriately, all terms $\lVert \partial_j^{\epsilon} \gradient u \rVert_{L^2}^2$ on the right can be absorbed by the left-hand side, and we obtain the desired inequality.
	
	Proof of \reflocal{item:a}.
	By \reflocal{item:0} and Banach-Alaoglu, there is a sequence $\epsilon_k$ converging to $0$ such that $\partial_j^{\epsilon_k} \gradient u$ converges weakly in $L^2(\Omega)$.
	The limit $w$ satisfies
	\begin{align*}
		\langle w, \phi \rangle &= \lim_{k \rightarrow \infty} \langle \partial_j^{\epsilon_k} \overline{\gradient u}, \phi \rangle_{L^2} \\
		&= \lim_{k \rightarrow \infty} \langle \overline{\gradient u}, \partial_j^{-\epsilon_k}  \phi \rangle_{L^2} = - \langle \gradient u, \partial_j \phi \rangle
	\end{align*}
	for all $\phi \in C^{\infty}_{c}(\Omega)$, hence $w$ is the weak derivative $\partial_j \gradient u$.
	The bound in \reflocal{item:0} also applies to $\partial_j \gradient u$.
	
	Proof of \reflocal{item:b}.
	\creflocal{eq:1} implies that $\divergence(a \gradient u) = f - \divergence g$ as distributions.
	Expanding the equation yields
	\begin{align} \labellocal{eq:3}
		\sum_{i,j=1}^{d} a_{ij} (\partial_i \partial_j u) + \sum_{i,j=1}^{d} (\partial_i a_{ij}) (\partial_j u) &= f - \divergence g.
	\end{align}
	The conclusion follows from isolating $a_{11} \partial_1^2 u$ and noting that $a_{11} > \alpha > 0$.
	
	Proof of \reflocal{item:c}.
	Fix $j \neq 1$, replacing $v$ in \creflocal{eq:1} by $\partial_j v$ and integrating by parts yields
	\begin{align*}
		- \int_{\Omega} (a \gradient u') \cdot (\gradient v) dV &= \int_{\Omega} f' v + g' \cdot (\gradient v) dV
	\end{align*}
	for $u' = \partial_j u$, $f' = \partial_j f$, and $g' = \partial_j g + (\partial_j a) \gradient u$.
	By induction, we get $\|u'\|_{H^{k+1}} \lesssim \|u'\|_{H^k} + \|f'\|_{H^{k-1}} + \|g'\|_{H^k}$, which gives $\lVert \partial_j u \rVert_{H^{k+1}(V)} \lesssim \|u\|_{H^{k+1}} + \|f\|_{H^k} + \|g\|_{H^{k+1}}$ for $j \neq 1$.
	By \creflocal{eq:3}, it also holds with $j = 1$.
\end{proof}

\begin{lemma}[Elliptic regularity] \label{lem:elliptic_regularity_0} \stepcounter{refer}
	Let $\Omega$ be a compact $d$-dimensional manifold with boundary, and $a \in C^{\infty}(\Omega, \CC^{d \times d})$ be uniformly positive-definite.
	Let $S$ be the set of $(u, f, g) \in H^1_0(\Omega^{\circ}) \times L^2(\Omega) \times L^2(\Omega, \CC^d)$ solving
	\begin{align} \labellocal{eq:asp}
		- \int_{\Omega} (a \gradient u) \cdot (\gradient v) dV &= \int_{\Omega} f v dV + \int_{\Omega} g \cdot \gradient v dV
	\end{align}
	for all $v \in C^{\infty}_{c}(\Omega^{\circ})$, 
	then
	\begin{align*}
		\lVert u \rVert_{H^{k+2}} &\lesssim \lVert u \rVert_{H^{k+1}} + \lVert f \rVert_{H^k} +  \lVert g \rVert_{H^{k+1}}
	\end{align*}
	uniform in $(u, f, g) \in S$.
\end{lemma}
\begin{proof}
	Let $U$ be any open subset of $\Omega$ diffeomorphic to an open subset of $\RRge{0} \times \RR^{d-1}$.
	Let $\eta$ be any smooth bump function supported in a compact subset of $U$.
	Replacing $v$ in \creflocal{eq:asp} by $\eta v$ and manipulating the equation yields
	\begin{align*}
		- \int_{\Omega} (a \gradient u') \cdot (\gradient v) dV &= \int_{\Omega} f' v + g' \cdot \gradient v dV
	\end{align*}
	for $u' = \eta u$, $f' = \eta f + g \gradient \eta + (a \gradient u) \cdot \gradient \eta$, and $g' = \eta g - ua \gradient \eta$.
	Assuming $\lVert u' \rVert_{H^{k+2}} \lesssim \lVert u' \rVert_{H^{k+1}} + \lVert f' \rVert_{H^k} +  \lVert g' \rVert_{H^{k+1}}$, then $\lVert \eta u \rVert_{H^{k+2}} \lesssim \lVert u \rVert_{H^{k+1}} + \lVert f \rVert_{H^k} +  \lVert g \rVert_{H^{k+1}}$.
	Choosing $\eta$ among a finite partition of unity over $\Omega$ completes the proof.
	Therefore, it suffices to prove this result in the special case that $u$, $f$, $g$ are compactly supported in $U$.
	Let $\phi^{-1}$ denote the diffeomorphism $U \hooktwoheadrightarrow U' \subseteq \RRge{0} \times \RR^{d-1}$.
	Shrinking $U$ if necessary, we may assume that the total derivatives $\|D\phi^{-1}\|_{L^{\infty}(U)}$ and $\|D\phi\|_{L^{\infty}(U')}$ are bounded.
	By a change of variable, \creflocal{eq:asp} writes
	\begin{align*}
		\int_{U'} (a \gradient u) \cdot (\gradient v) dV &= \int_{U'} f' v + g' \cdot \gradient v dV
	\end{align*}
	for $a' = \lvert \det D\phi \rvert (D\phi^{-1}) a (D\phi^{-1})^T$, $f' = f \lvert \det D\phi \rvert$, and $g' = \lvert \det D\phi \rvert (D\phi^{-1})g$.
	Note that $\gradient$ denotes the Euclidean gradient on $\RR^{d}$, and all functions are interpreted as functions on $U'$ (composition with $\phi$ is omitted).
	Extending $u$, $f'$, $g'$ extend by zero and $a'$ to any uniformly positive-definite function on $\RRge{0} \times \RR^{d-1}$, then the above identity holds for all $v \in C^{\infty}_{c}(\RRg{0} \times \RR^{d-1})$.
	The conclusion then follows from \cref{lem:elliptic_regularity_local_0}.
\end{proof}

\begin{lemma}[Local elliptic regularity] \label{lem:elliptic_regularity_local_1}
	Given $\Omega = \RRge{0} \times \RR^{d-1}$ and a uniformly positive-definite $a \in C^{\infty}_{\local}(\Omega, \CC^{d \times d})$,
	denote by $S$ the set of all $(u, f, g, h) \in H^1(\Omega) \times L^2(\Omega) \times L^2(\Omega, \CC^d) \times L^2(\partial \Omega)$ solving
	\begin{align} \labellocal{eq:1}
		\begin{aligned}
		- \int_{\Omega} (a \gradient u) \cdot (\gradient v) dV &= \int_{\Omega} (f v + g \cdot \gradient v) dV + \int_{\partial \Omega} h v dS   
		\end{aligned}
	\end{align}
	for all $v \in C^{\infty}_{c}(\Omega)$.
	For every $k \in \ZZge{0}$,
	\begin{align*}
		\begin{aligned}
		\lVert u \rVert_{H^{k+2}} &\lesssim \lVert u \rVert_{H^{k+1}} + \lVert f \rVert_{H^{k}} + \lVert g \rVert_{H^{k+1}} + \lVert h \rVert_{H^{k+\frac{1}{2}}(\partial \Omega)}
		\end{aligned}
	\end{align*}
	uniform in $(u, f, g, h) \in S$.
\end{lemma}
\begin{proof}
	Similar to the proof of \cref{lem:elliptic_regularity_local_0}, the main steps are 
	\begin{enumerate}[(a)]
		\item\labellocal{item:0} $\lVert \partial_j^{\epsilon} \gradient u \rVert_{L^2(V, \CC^d)} \lesssim \|u\|_{H^1} + \|f\|_{L^2} + \|g\|_{H^1} + \|h\|_{H^{1/2}(\partial \Omega)}$ for $j \neq 1$.
		\item $\|\partial_j \gradient u\| \lesssim \|u\|_{H^1} + \|f\|_{L^2} + \|g\|_{H^1} + \|h\|_{H^{1/2}(\partial \Omega)}$ for $j \neq 1$.
		\item $\|u\|_{H^2} \lesssim \|u\|_{H^1} + \|f\|_{L^2} + \|g\|_{H^1} + \|h\|_{H^{1/2}(\partial \Omega)}$
		\item $\|u\|_{H^{k+2}} \lesssim \|u\|_{H^{k+1}} + \|f\|_{H^k} + \|g\|_{H^{k+1}} + \|h\|_{H^{k+1/2}(\partial \Omega)}$ for $k \in \ZZge{1}$.
	\end{enumerate}
	The only difference in the proof occurs in \reflocal{item:0}.
	When substituting $v = \partial_j^{-\epsilon} (\partial_j^{\epsilon} \bar{u})$ into \creflocal{eq:1}, the right-hand side gets the additional term $\int_{\partial \Omega} h \partial_j^{-\epsilon} (\overline{\partial_j^{\epsilon} u}) dS$.
	It is bounded by $\|h\|_{H^{1/2}(\partial \Omega)} \| \partial_j (\partial_j^{\epsilon} u)\|_{H^{-1/2}(\partial \Omega)}$.
	The latter factor is dominated by $\| \partial_j^{\epsilon} u\|_{H^{1/2}(\partial \Omega)}$ by definition and $\| \partial_j^{\epsilon} u\|_{H^{1}(\Omega)}$ by the trace inequality (\cref{lem:trace}).
	This justifies \reflocal{item:0} with the additional term $\|h\|_{H^{1/2}(\partial \Omega)}$ on the right-hand side.
\end{proof}

\begin{lemma}[Elliptic regularity] \label{lem:elliptic_regularity_1} \stepcounter{refer}
	Let $\Omega$ be a compact $d$-dimensional manifold with boundary.
	Let $S$ be the set of $(u, f, g, h) \in H^1(\Omega) \times L^2(\Omega) \times L^2(\Omega, \CC^d) \times L^2(\partial \Omega)$ solving
	\begin{align*}
		- \int_{\Omega} (a \gradient u)(\gradient v) dV &= \int_{\Omega} (f v + g \cdot \gradient v) dV + \int_{\partial \Omega} h v dS
	\end{align*}
	for all $v \in C^{\infty}(\Omega)$, then
	\begin{align*}
		\lVert u \rVert_{H^{k+2}} &\lesssim \lVert u \rVert_{H^1} + \lVert f \rVert_{H^k} + \| g \|_{H^{k+1}} + \lVert h \rVert_{H^{k+\frac{1}{2}}(\partial \Omega)}
	\end{align*}
	uniform in $(u, f, g, h) \in S$.
\end{lemma}
\begin{proof}
	The proof is nearly identical to that of \cref{lem:elliptic_regularity_0}, except that \cref{lem:elliptic_regularity_local_1} is used in place of \cref{lem:elliptic_regularity_local_0}.
\end{proof}

\end{shownto}

\bibliographystyle{./siamplain} 
\bibliography{LinearNetwork}

\end{document}